\documentclass[10pt]{amsart}
\usepackage{amsmath}
\usepackage{amsfonts}
\usepackage{amsxtra}
\usepackage{amscd}
\usepackage{amsthm}
\usepackage{amssymb}
\usepackage{eucal}
\usepackage[all]{xy}
\usepackage{graphicx}

\makeatletter
\def\amsbb{\use@mathgroup \M@U \symAMSb}
\makeatother

\usepackage{bbold}

\renewcommand{\mathbb}{\amsbb}
\newcommand{\tri}{\blacktriangle}

\theoremstyle{plain}
\newtheorem{Thm}{Theorem}[section]
\newtheorem{Thm*}{Theorem}[section]
\newtheorem{Thm'}[Thm]{"Theorem"}
\newtheorem{Cor}[Thm]{Corollary}

\newtheorem{Prop}[Thm]{Proposition}
\newtheorem{Lem}[Thm]{Lemma}
\newtheorem{Cl}[Thm]{Claim}

\theoremstyle{definition}
\newtheorem{Def}[Thm]{Definition}

\newtheorem{Emp}[Thm]{}

\newtheorem{Q}[Thm]{Question}
\numberwithin{equation}{section}
\newcommand{\nc}{\newcommand}
\nc{\lm}{\lambda}
\newcommand{\pp}{\boxtimes}
\newcommand{\ov}{\overline}

\newcommand{\B}[1]{\mathbb#1}
\newcommand{\cal}[1]{\mathcal{#1}}
\newcommand{\isom}{\overset {\thicksim}{\to}}

\newcommand{\om}{\omega}

\newcommand{\lra}{\longrightarrow}
\newcommand{\lla}{\longleftarrow}
\newcommand{\hra}{\hookrightarrow}
\newcommand{\wt}{\widetilde}
\newcommand{\Gm}{\Gamma}
\newcommand{\lan}{\langle}

\newcommand{\ran}{\rangle}

\renewcommand{\P}{P}

\newcommand{\Dt}{\Delta}

\newcommand{\un}[1]{\underline{#1}}

\newcommand{\al}{\alpha}

\newcommand{\form}[1]{(\ref{Eq:#1})}

\newcommand{\rs}[1]{Section~\ref{S:#1}}
\newcommand{\rl}[1]{Lemma~\ref{L:#1}}

\newcommand{\rcl}[1]{Claim~\ref{C:#1}}
\newcommand{\rp}[1]{Proposition~\ref{P:#1}}

\newcommand{\re}[1]{\ref{E:#1}}
\newcommand{\rco}[1]{Corollary~\ref{C:#1}}

\newcommand{\rt}[1] {Theorem~\ref{T:#1}}

\newcommand{\sm}{\smallsetminus}

\newcommand{\pr}{\operatorname{pr}}

\newcommand{\Isom}{\operatorname{Isom}}

\newcommand{\Spec}{\operatorname{Spec}}

\newcommand{\Aut}{\operatorname{Aut}}

\newcommand{\Gr}{\operatorname{Gr}}

\newcommand{\Tr}{\operatorname{Tr}}
\newcommand{\red}{\operatorname{red}}

\newcommand{\ev}{\operatorname{ev}}

\newcommand{\Fr}{\operatorname{Fr}}
\newcommand{\Id}{\operatorname{Id}}

\newcommand{\Fun}{\operatorname{Funct}}

\newcommand{\Ind}{\operatorname{Ind}}

\newcommand{\colim}{\operatorname{colim}}

\newcommand{\pt}{\operatorname{pt}}
\newcommand{\fq}{\B{F}_q}

\newcommand{\qlbar}{{\overline{\mathbb Q}_\ell}}
\newcommand{\fqbar}{\overline{\B{F}}_q}

\newcommand{\cD}{\mathcal{D}}

\newcommand{\cont}{\on{cont}}
\newcommand{\Fix}{\on{Fix}}
\newcommand{\unit}{\on{unit}}

\theoremstyle{definition}

\theoremstyle{remark}

\numberwithin{equation}{section}

\nc{\renc}{\renewcommand}
\nc{\ssec}{\subsection}
\nc{\sssec}{\subsubsection}
\nc{\on}{\operatorname}

\nc\ol{\overline}
\nc\tboxtimes{\wt{\boxtimes}}
\nc\tstar{\wt{\star}}
\nc{\alp}{a}

\nc{\ZZ}{{\mathbb Z}}
\nc{\NN}{{\mathbb N}}
\nc{\OO}{{\mathbb O}}
\renc{\SS}{{\mathbb S}}
\nc{\DD}{{\mathbb D}}
\nc{\GG}{{\mathbb G}}

\nc{\Fq}{{\mathbb F}_q}
\nc{\Fqb}{\ol{{\mathbb F}}_q}
\nc{\Ql}{\ol{{\mathbb Q}_\ell}}
\nc{\id}{\text{id}}
\nc\X{\mathcal X}

\nc{\Loc}{\on{Loc}}
\nc{\Pic}{\on{Pic}}
\nc{\Bun}{\on{Bun}}
\nc{\IC}{\on{IC}}
\nc{\Fls}{\on{Fl}^{\frac{\infty}{2}}}
\nc{\ICs}{\on{IC}^{\frac{\infty}{2}}}
\nc{\ICsl}{\on{IC}^{\lambda+\frac{\infty}{2}}}
\nc{\ICslm}{\on{IC}^{\lambda+\frac{\infty}{2},-}}
\nc{\ICsm}{\on{IC}^{\frac{\infty}{2},-}}
\nc{\Sh}{\frak{Sh}}
\nc{\pos}{{\on{pos}}}
\nc{\Conv}{\on{Conv}}
\nc{\Sph}{\on{Sph}}
\nc{\Sat}{\on{Sat}}
\nc{\Sym}{\on{Sym}}
\nc{\BunBb}{\overline{\Bun}_B}
\nc{\BunNb}{\overline{\Bun}_N}
\nc{\BunTb}{\overline{\Bun}_T}
\nc{\BunBbm}{\overline{\Bun}_{B^-}}
\nc{\BunBbel}{\overline{\Bun}_{B,el}}
\nc{\BunBbmel}{\overline{\Bun}_{B^-,el}}
\nc{\Buno}{\overset{o}{\Bun}}
\nc{\BunPb}{{\overline{\Bun}_P}}
\nc{\BunBM}{\Bun_{B(M)}}
\nc{\BunBMb}{\overline{\Bun}_{B(M)}}
\nc{\BunPbw}{{\widetilde{\Bun}_P}}
\nc{\BunBP}{\widetilde{\Bun}_{B,P}}
\nc{\GUb}{\overline{G/U}}
\nc{\GUPb}{\overline{G/U(P)}}

\nc{\Hhom}{\underline{\on{Hom}}}
\nc\syminfty{\on{Sym}^{\infty}}
\nc\lal{\ol{\kappa_x}}
\nc\xl{\ol{x}}
\nc\thl{\ol{\theta}}
\nc\nul{\ol{\nu}}
\nc\mul{\ol{\mu}}
\nc\Sum\Sigma
\nc{\oX}{\overset{o}{X}{}}
\nc{\hl}{\overset{\leftarrow}h{}}
\nc{\hr}{\overset{\rightarrow}h{}}
\nc{\M}{{\mathcal M}}
\nc{\N}{{\mathcal N}}
\nc{\F}{{\mathcal F}}
\nc{\D}{{\mathcal D}}

\nc{\Y}{{\mathcal Y}}
\nc{\G}{{\mathcal G}}
\nc{\E}{{\mathcal E}}
\nc{\CalC}{{\mathcal C}}
\nc\Dh{\widehat{\D}}
\renewcommand{\O}{{\mathcal O}}
\nc{\C}{{\mathcal C}}
\nc{\K}{{\mathcal K}}
\renewcommand{\H}{{\mathcal H}}

\nc{\T}{{\mathcal T}}
\nc{\V}{{\mathcal V}}
\renc{\P}{{\mathcal P}}
\nc{\A}{{\mathcal A}}
\nc{\U}{{\mathcal U}}

\nc{\frn}{{\check{\mathfrak u}(P)}}

\nc{\fC}{\mathfrak C}
\nc{\p}{\mathfrak p}
\nc{\q}{\mathfrak q}
\nc\f{{\mathfrak f}}

\nc{\qo}{{\mathfrak q}}
\nc{\po}{{\mathfrak p}}
\nc{\s}{{\mathfrak s}}
\nc\w{\text{w}}
\nc{\act}{\on{act}}

\nc\mathi\iota
\nc\Mod{\on{Mod}}

\nc{\pw}{\widetilde{\mathfrak p}}
\nc{\qw}{\widetilde{\mathfrak q}}
\nc{\jw}{\widetilde j}

\nc{\grb}{\overline{\Gr}}
\nc{\I}{\mathcal I}

\nc{\kappach}{{\check\kappa_x}}
\nc{\Lambdach}{{\check\Lambda}{}}
\nc{\much}{{\check\mu}}
\nc{\omegach}{{\check\omega}}
\nc{\nuch}{{\check\nu}}
\nc{\etach}{{\check\eta}}
\nc{\alphach}{{\checka}}
\nc{\oblvtach}{{\check\oblvta}}
\nc{\pich}{{\check\pi}}
\nc{\ch}{{\check h}}

\nc{\Hb}{\overline{\H}}

\nc{\BA}{{\mathbb{A}}}
\nc{\BC}{{\mathbb{C}}}
\nc{\BE}{{\mathbb{E}}}
\nc{\BF}{{\mathbb{F}}}
\nc{\BG}{{\mathbb{G}}}
\nc{\BM}{{\mathbb{M}}}
\nc{\BO}{{\mathbb{O}}}
\nc{\BD}{{\mathbb{D}}}
\nc{\BL}{{\mathbb{L}}}
\nc{\BN}{{\mathbb{N}}}
\nc{\BP}{{\mathbb{P}}}
\nc{\BQ}{{\mathbb{Q}}}
\nc{\BR}{{\mathbb{R}}}
\nc{\BV}{{\mathbb{V}}}
\nc{\BW}{{\mathbb{W}}}
\nc{\BZ}{{\mathbb{Z}}}
\nc{\BS}{{\mathbb{S}}}

\nc{\CA}{{\mathcal{A}}}
\nc{\CB}{{\mathcal{B}}}

\nc{\CE}{{\mathcal{E}}}
\nc{\CF}{{\mathcal{F}}}
\nc{\CG}{{\mathcal{G}}}
\nc{\CH}{{\mathcal{H}}}

\nc{\CL}{{\mathcal{L}}}
\nc{\CC}{{\mathcal{C}}}
\nc{\CM}{{\mathcal{M}}}
\nc{\CN}{{\mathcal{N}}}
\nc{\cCN}{\check{{\mathcal{N}}}}
\nc{\CK}{{\mathcal{K}}}
\nc{\CO}{{\mathcal{O}}}
\nc{\CP}{{\mathcal{P}}}
\nc{\CQ}{{\mathcal{Q}}}
\nc{\CR}{{\mathcal{R}}}
\nc{\CS}{{\mathcal{S}}}
\nc{\CT}{{\mathcal{T}}}
\nc{\CU}{{\mathcal{U}}}
\nc{\CV}{{\mathcal{V}}}
\nc{\CW}{{\mathcal{W}}}
\nc{\CX}{{\mathcal{X}}}
\nc{\CY}{{\mathcal{Y}}}
\nc{\CZ}{{\mathcal{Z}}}
\nc{\CI}{{\mathcal{I}}}
\nc{\CJ}{{\mathcal{J}}}

\nc{\csM}{{\check{\mathcal A}}{}}
\nc{\oM}{{\overset{\circ}{\mathcal M}}{}}
\nc{\obM}{{\overset{\circ}{\mathbf M}}{}}
\nc{\oCA}{{\overset{\circ}{\mathcal A}}{}}
\nc{\obA}{{\overset{\circ}{\mathbf A}}{}}
\nc{\ooM}{{\overset{\circ}{M}}{}}
\nc{\osM}{{\overset{\circ}{\mathsf M}}{}}
\nc{\vM}{{\overset{\bullet}{\mathcal M}}{}}
\nc{\nM}{{\underset{\bullet}{\mathcal M}}{}}
\nc{\oD}{{\overset{\circ}{\mathcal D}}{}}
\nc{\obD}{{\overset{\circ}{\mathbf D}}{}}
\nc{\oA}{{\overset{\circ}{\mathbb A}}{}}
\nc{\op}{{\overset{\bullet}{\mathbf p}}{}}
\nc{\cp}{{\overset{\circ}{\mathbf p}}{}}
\nc{\oU}{{\overset{\bullet}{\mathcal U}}{}}
\nc{\oZ}{{\overset{\circ}{\mathcal Z}}{}}
\nc{\ofZ}{{\overset{\circ}{\mathfrak Z}}{}}
\nc{\oF}{{\overset{\circ}{\fF}}}

\nc{\fa}{{\mathfrak{a}}}
\nc{\fb}{{\mathfrak{b}}}
\nc{\fd}{{\mathfrak{d}}}
\nc{\ff}{{\mathfrak{f}}}
\nc{\fg}{{\mathfrak{g}}}
\nc{\fgl}{{\mathfrak{gl}}}
\nc{\fh}{{\mathfrak{h}}}
\nc{\fj}{{\mathfrak{j}}}
\nc{\fl}{{\mathfrak{l}}}
\nc{\fm}{{\mathfrak{m}}}
\nc{\fn}{{\mathfrak{n}}}
\nc{\fu}{{\mathfrak{u}}}
\nc{\fp}{{\mathfrak{p}}}
\nc{\fr}{{\mathfrak{r}}}
\nc{\fs}{{\mathfrak{s}}}
\nc{\ft}{{\mathfrak{t}}}
\nc{\fz}{{\mathfrak{z}}}
\nc{\fsl}{{\mathfrak{sl}}}
\nc{\hsl}{{\widehat{\mathfrak{sl}}}}
\nc{\hgl}{{\widehat{\mathfrak{gl}}}}
\nc{\hg}{{\widehat{\mathfrak{g}}}}
\nc{\htt}{{\widehat{\mathfrak{t}}}}
\nc{\chg}{{\widehat{\mathfrak{g}}}{}^\vee}
\nc{\hn}{{\widehat{\mathfrak{n}}}}
\nc{\chn}{{\widehat{\mathfrak{n}}}{}^\vee}

\nc{\fA}{{\mathfrak{A}}}
\nc{\fB}{{\mathfrak{B}}}
\nc{\fD}{{\mathfrak{D}}}
\nc{\fE}{{\mathfrak{E}}}
\nc{\fF}{{\mathfrak{F}}}
\nc{\fG}{{\mathfrak{G}}}
\nc{\fK}{{\mathfrak{K}}}
\nc{\fL}{{\mathfrak{L}}}
\nc{\fM}{{\mathfrak{M}}}
\nc{\fN}{{\mathfrak{N}}}
\nc{\fP}{{\mathfrak{P}}}
\nc{\fU}{{\mathfrak{U}}}
\nc{\fV}{{\mathfrak{V}}}
\nc{\fZ}{{\mathfrak{Z}}}

\nc{\ba}{{\mathbf{a}}}

\nc{\bc}{{\mathbf{c}}}
\nc{\bd}{{\mathbf{d}}}
\nc{\bbf}{{\mathbf{f}}}
\nc{\be}{{\mathbf{e}}}
\nc{\bi}{{\mathbf{i}}}
\nc{\bj}{{\mathbf{j}}}
\nc{\bn}{{\mathbf{n}}}
\nc{\bo}{{\mathbf{o}}}
\nc{\bp}{{\mathbf{p}}}
\nc{\bq}{{\mathbf{q}}}
\nc{\bu}{{\mathbf{u}}}
\nc{\bv}{{\mathbf{v}}}
\nc{\bx}{{\mathbf{x}}}
\nc{\by}{{\mathbf{y}}}
\nc{\bw}{{\mathbf{w}}}
\nc{\bA}{{\mathbf{A}}}
\nc{\bK}{{\mathbf{K}}}
\nc{\bB}{{\mathbf{B}}}
\nc{\bF}{{\mathbf{F}}}
\nc{\bC}{{\mathbf{C}}}
\nc{\bG}{{\mathbf{G}}}
\nc{\bD}{{\mathbf{D}}}
\nc{\bE}{{\mathbf{E}}}
\nc{\bH}{{\mathbf{H}}}
\nc{\bI}{{\mathbf{I}}}
\nc{\bL}{{\mathbf{L}}}
\nc{\bM}{{\mathbf{M}}}
\nc{\bN}{{\mathbf{N}}}
\nc{\bO}{{\mathbf{O}}}
\nc{\bV}{{\mathbf{V}}}
\nc{\bW}{{\mathbf{W}}}
\nc{\bX}{{\mathbf{X}}}
\nc{\bZ}{{\mathbf{Z}}}
\nc{\bS}{{\mathbf{S}}}

\nc{\sA}{{\mathsf{A}}}
\nc{\sB}{{\mathsf{B}}}
\nc{\sC}{{\mathsf{C}}}
\nc{\sD}{{\mathsf{D}}}
\nc{\sF}{{\mathsf{F}}}
\nc{\sK}{{\mathsf{K}}}
\nc{\sM}{{\mathsf{M}}}
\nc{\sO}{{\mathsf{O}}}
\nc{\sW}{{\mathsf{W}}}
\nc{\sQ}{{\mathsf{Q}}}
\nc{\sP}{{\mathsf{P}}}
\nc{\sZ}{{\mathsf{Z}}}
\nc{\sr}{{\mathsf{r}}}
\nc{\bk}{{\mathsf{k}}}
\nc{\sg}{{\mathsf{g}}}
\nc{\sff}{{\mathsf{f}}}
\nc{\sfe}{{\mathsf{e}}}
\nc{\sfj}{{\mathsf{j}}}
\nc{\sfb}{{\mathsf{b}}}
\nc{\sfc}{{\mathsf{c}}}
\nc{\sd}{{\mathsf{d}}}
\nc{\sv}{{\mathsf{v}}}

\nc{\BK}{{\bar{K}}}

\nc{\tA}{{\widetilde{\mathbf{A}}}}
\nc{\tB}{{\widetilde{\mathcal{B}}}}
\nc{\tg}{{\widetilde{\mathfrak{g}}}}
\nc{\tG}{{\widetilde{G}}}
\nc{\TM}{{\widetilde{\mathbb{M}}}{}}
\nc{\tO}{{\widetilde{\mathsf{O}}}{}}
\nc{\tU}{{\widetilde{\mathfrak{U}}}{}}
\nc{\TZ}{{\tilde{Z}}}
\nc{\tx}{{\tilde{x}}}
\nc{\tbv}{{\tilde{\bv}}}
\nc{\tfP}{{\widetilde{\mathfrak{P}}}{}}
\nc{\tz}{{\tilde{\zeta}}}
\nc{\tmu}{{\tilde{\mu}}}

\nc{\urho}{\underline{\pi}}
\nc{\uB}{\underline{B}}
\nc{\uC}{{\underline{\mathbb{C}}}}
\nc{\ui}{\underline{i}}
\nc{\uj}{\underline{j}}
\nc{\ofP}{{\overline{\mathfrak{P}}}}
\nc{\oB}{{\overline{\mathcal{B}}}}
\nc{\og}{{\overline{\mathfrak{g}}}}
\nc{\oI}{{\overline{I}}}

\nc{\eps}{\varepsilon}
\nc{\hrho}{{\hat{\pi}}}

\nc{\one}{{\mathbf{1}}}
\nc{\two}{{\mathbf{t}}}

\nc{\Hilb}{{\mathop{\operatorname{\rm Hilb}}}}
\nc{\CHom}{{\mathop{\operatorname{{\mathcal{H}}\it om}}}}
\nc{\de}{{\mathop{\operatorname{\rm def}}}}
\nc{\length}{{\mathop{\operatorname{\rm length}}}}
\nc{\supp}{{\mathop{\operatorname{\rm supp}}}}

\nc{\Cliff}{{\mathsf{Cliff}}}
\nc{\Fib}{{\mathsf{Fib}}}
\nc{\Coh}{{\mathsf{Coh}}}
\nc{\QCoh}{{\on{QCoh}}}
\nc{\IndCoh}{{\on{IndCoh}}}
\nc{\FCoh}{{\mathsf{FCoh}}}
\nc{\Func}{{\on{funct}}}
\nc{\reg}{{\text{\rm reg}}}

\nc{\cplus}{{\mathbf{C}_+}}
\nc{\cminus}{{\mathbf{C}_-}}
\nc{\cthree}{{\mathbf{C}_*}}
\nc{\Qbar}{{\bar{Q}}}
\nc\Eis{\on{Eis}}
\nc\Eisb{\ol\Eis{}}
\nc\Eisr{\on{Eis}^{rat}{}}
\nc\wh{\widehat}

\nc{\barZ}{\overline{Z}{}}
\nc{\barbarZ}{\overline{\barZ}{}}
\nc{\barpi}{\overline\iota}
\nc{\barbarpi}{\overline\barpi}
\nc{\barpip}{\overline\iota{}^+}
\nc{\barpim}{\overline\iota{}^-}

\nc{\fqb}{\ol{\fq}{}}
\nc{\fpb}{\ol{\fp}{}}
\nc{\fpr}{{\fp^{rat}}{}}
\nc{\fqr}{{\fq^{rat}}{}}

\nc{\hattimes}{\wh\otimes}

\nc{\bh}{{\bar{h}}}
\nc{\bOmega}{{\overline{\Omega(\check \fn)}}}

\nc{\seq}[1]{\stackrel{#1}{\sim}}

\nc{\cT}{{\check{T}}}
\nc{\cG}{{\check{G}}}
\nc{\cM}{{\check{M}}}
\nc{\cB}{{\check{B}}}
\nc{\cN}{{\check{N}}}

\nc{\ct}{{\check{\mathfrak t}}}
\nc{\cg}{{\check{\fg}}}
\nc{\hcg}{{\widehat{\check{\fg}}}}
\nc{\cb}{{\check{\fb}}}
\nc{\cn}{{\check{\fn}}}

\nc{\cLambda}{{\check\Lambda}}

\nc{\cla}{{\check\kappa_x}}
\nc{\cmu}{{\check\mu}}
\nc{\clambda}{{\check\lambda}}
\nc{\cnu}{{\check\nu}}
\nc{\ceta}{{\check\eta}}

\nc{\DefbE}{{\on{Def}_{\cB}(E_\cT)}}

\nc{\imathb}{{\ol{\imath}}}
\nc{\rlr}{\overset{\longrightarrow}{\underset{\longrightarrow}\longleftarrow}}

\nc{\KG}{K\backslash G}
\nc{\comult}{{co\text{-}mult}}
\nc{\counit}{{co\text{-}unit}}
\nc{\uHom}{{\underline{\Maps}}}
\nc{\dgSch}{\on{Sch}}
\nc{\affdgSch}{\on{Sch}^{\on{aff}}}
\nc{\affSch}{\on{Sch}^{\on{aff}}}
\nc{\Groupoids}{\on{Grpd}}
\nc{\inftygroup}{\on{Spc}}
\nc{\inftyCat}{\infty\on{-Cat}}
\nc{\StinftyCat}{\inftyCat^{\on{St}}}
\nc{\MoninftyCat}{\infty\on{-Cat}^{\on{Mon}}}
\nc{\SymMoninftyCat}{\infty\on{-Cat}^{\on{SymMon}}}
\nc{\SymMonStinftyCat}{\on{DGCat}^{\on{SymMon}}}
\nc{\MonStinftyCat}{\on{DGCat}^{\on{Mon}}}
\nc{\inftystack}{\on{Stk}}
\nc{\inftystackalg}{Stk^{1\text{-}alg}}
\nc{\inftyprestack}{\on{PreStk}}
\nc{\inftydgnearstack}{\on{NearStk}}
\nc{\inftydgstack}{\on{Stk}}
\nc{\inftydgstackalg}{DGStk^{1\text{-}alg}}
\nc{\inftydgprestack}{\on{PreStk}}
\nc{\dgindSch}{\on{indSch}}
\nc{\indSch}{{}^{\on{cl}}\!\on{indSch}}
\nc{\infSch}{\on{infSch}}
\nc{\dr}{{\on{dR}}}
\nc{\KER}{\on{KER}}

\nc{\mmod}{{\on{-}\!{\mathbf{mod}}}}

\nc{\starr}{\text{\dh}}
\nc{\Spectra}{\on{Spectra}}
\nc{\Crys}{\on{Crys}}
\nc{\oblv}{{\mathbf{oblv}}}
\nc{\ind}{{\mathbf{ind}}}
\nc{\coind}{{\mathbf{coind}}}
\nc{\inv}{{\mathbf{inv}}}
\nc{\triv}{{\mathbf{triv}}}
\nc{\CMaps}{{\mathcal Maps}}
\nc{\Maps}{\on{Maps}}
\nc{\bMaps}{\mathbf{Maps}}
\nc{\BMaps}{\ul{\on{Maps}}}
\nc{\Grid}{\on{Grid}}
\nc{\hGrid}{\on{Grid}^{\geq\,\on{dgnl}}}
\nc{\Diag}{\on{Diag}}
\nc{\bDelta}{\mathbf{\Delta}}
\nc{\tCateg}{(\infty\on{-2)-Cat}}
\nc{\ul}{\underline}
\nc{\Seg}{\on{Seq}}
\nc{\biSeg}{\on{bi-Seq}}
\nc{\triSeg}{\on{tri-Seq}}
\nc{\quadSeg}{\on{quad-Seq}}
\nc{\nSeg}{\on{n-Seq}}
\nc{\Segm}{\on{Seg}^{\on{mkd}}}
\nc{\fLm}{\fL^{\on{mkd}}}
\nc{\inftyCatm}{\inftyCat^{\on{mkd}}}
\nc{\Blocks}{\mathbf{Blocks}}
\nc{\Snakes}{\mathbf{Snakes}}
\nc{\bifL}{\on{bi-}\!\fL}
\nc{\Sets}{\on{Sets}}
\nc{\Ran}{{\on{Ran}}}
\nc{\ren}{{\on{ren}}}
\nc{\unren}{{\on{unren}}}
\nc{\constr}{{\on{constr}}}
\nc{\naive}{{\on{naive}}}
\nc{\true}{{\on{true}}}
\nc{\QProj}{{\on{QProj}}}
\nc{\Vect}{\on{Vect}}
\nc{\Shv}{\on{Shv}}
\nc{\unn}{\mathbf{union}}
\nc{\Spc}{\on{Spc}}
\nc{\ppart}{(\!(t)\!)}
\nc{\qqart}{[\![t]\!]}
\nc{\Dmod}{\on{D-mod}}
\nc{\ocD}{\overset{\circ}{\cD}}
\nc{\sfp}{\mathsf{p}}
\nc{\sfq}{\mathsf{q}}
\nc{\DGCat}{\on{DGCat}}
\renc{\det}{\on{det}}
\nc{\Conf}{\on{Conf}}
\nc{\Whit}{\on{Whit}}
\nc{\Reg}{\on{Reg}}
\nc{\BunNbx}{(\BunNb)_{\infty\cdot x}}
\nc{\bHecke}{\overset{\bullet}{\on{Hecke}}}
\nc{\Hecke}{\on{Hecke}}
\nc{\bCZ}{\ol\CZ}
\nc{\oCZ}{\overset{\circ}\CZ}
\nc{\boCZ}{\ol{\oCZ}}
\nc{\sotimes}{\overset{!}\otimes}
\nc{\semiinf}{\frac{\infty}{2}}
\nc{\coInd}{\on{coInd}}
\nc{\bCM}{\overset{\bullet}\CM{}}
\nc{\bCF}{\overset{\bullet}\CF{}}
\nc{\SI}{\on{SI}}
\nc{\KL}{\on{KL}}
\nc{\Av}{\on{Av}}
\nc{\Inv}{\on{Inv}}
\nc{\KM}{\on{KM}}
\nc{\LocSys}{\on{LocSys}}
\nc{\tr}{\on{tr}}
\nc{\LT}{\on{LT}}
\begin{document}

\title[Local terms for the categorical trace]
{Local terms for the categorical trace}

\date{\today}

\author{Dennis Gaitsgory}
\address{Max Planck Institute for Mathematics\\
Vivatsgasse 7, 53111 Bonn, Germany}
\email{gaitsgde@mpim-bonn.mpg.de} 

%{Institute of Mathematics\\
%The Hebrew University of Jerusalem\\
%Givat-Ram, Jerusalem,  91904\\
%Israel} \email{vyakov@math.huji.ac.il}

\author{Yakov Varshavsky}
\address{Einstein Institute of Mathematics\\
Edmond J. Safra Campus\\
The Hebrew University of Jerusalem\\
Givat Ram, Jerusalem, 9190401, Israel}
\email{yakov.varshavsky@mail.huji.ac.il}

%\thanks{The research of Y.V. was partially supported by the ISF grant 822/17. The research of D.G. is supported by NSF grant DMS-2005475.}
%\abstract
\begin{abstract}
In this paper we introduce the categorical ``true local terms'' maps for Artin stacks and show that they are additive and commute with proper pushforwards, smooth pullbacks and specializations. In particular, we generalizing results of \cite{Va} to this setting.

As an application, we supply proofs of two theorems stated in \cite{main}. Namely, we show that  the ``true local terms'' of the Frobenius endomorphism coincide with the ``naive local terms'' and that the ``naive local terms'' commute with $!$-pushforwards. The latter result is a categorical version of the classical Grothendieck--Lefschetz trace formula.
\end{abstract}

\maketitle

%\centerline{Preliminary version}

\tableofcontents

\section*{Introduction}

\begin{Emp}
%{\bf Trace of Frobenius.}
Let $k$ be an algebraically closed field.

\smallskip

(a) Let $X$ be an Artin stack of finite presentation over $k$. To $X$ we can associate two DG categories: category $\Shv(X)$ of $\qlbar$-adic sheaves on $X$ (see \cite[Appendix~F]{main}) and category $\Shv(X)^{\ren}:=\Ind\Shv(X)^{\constr}$ of
ind-constructible sheaves (see \cite[Section~F.5]{main}). We have a natural fully faithful {\em renormalization functor}
\[
\ren_X:\Shv(X)\to \Shv(X)^{\ren},
\]
which has a continuous right adjoint $\unren_X$.

\smallskip

(b) Let $c=(c_l,c_r):C\to X\times X$ be a morphism of Artin stacks of finite presentation over $k$, which we call a {\em correspondence}.
A correspondence $c$ gives rise to continuous endofunctors
\[
(c_l)_{\tri}\circ c_r^!:\Shv(X)\to \Shv(X)\text{ and } (c_l)_{*}\circ c_r^!:\Shv(X)^{\ren}\to \Shv(X)^{\ren},
\]
where $(-)_{\tri}$ denote the renormalized pushforward (see Section~\re{ren}),  and we denote both of these endofunctors
by $[c]$.

\smallskip

(c) Both DG categories $\Shv(X)$ and $\Shv(X)^{\ren}$ are compactly generated, thus dualizable, hence one can consider traces  $\Tr(\Shv(X),[c])$ and $\Tr(\Shv(X)^{\ren},[c])$, which are objects of the $\infty$-category $\Vect$ of $\qlbar$-vector spaces. Furthermore, by the trace formalism (see \cite[Section~3.2]{GKRV}) the renormalization functor $\ren_X$ gives rise to a morphism between traces
\[
\Tr(\ren_X,[c]):\Tr(\Shv(X),[c])\to \Tr(\Shv(X)^{\ren},[c]).
\]
\end{Emp}

\begin{Emp} \label{E:trueloc}
{\bf True local terms.} Let  $c=(c_l,c_r):C\to X\times X$ be a correspondence between Artin stacks of finite presentation over $k$, and
let $\A\in\Shv(X)^{\constr}$ be a constructible sheaf.

\smallskip

(a) Extending the construction of \cite[Section~1.2.2]{Va}, to this data one can associate the {\em trace map}
\[
{\cal Tr}_{c,\A}:\CHom_{\Shv(X)^{\ren}}(\A,[c](\A))\to  \Gm(\Fix(c),\om_{\Fix(c)}),
\]
where

\quad\quad $\bullet$ $\Fix(c):=C\times_{X\times X}X$ is the Artin stack of fixed points of $c$,

\quad\quad $\bullet$ $\om_Y$ denotes the dualizing sheaf of a stack $Y$, and

\quad\quad $\bullet$ $\Gm(Y,-):\Shv(Y)^{\ren}\to \Vect$ denotes the functor of global sections.

\smallskip

(b) On the other hand, using functoriality of trace maps (see \cite[Section~3.5.4]{GKRV})
one associates to this data the {\em Chern character} map
\[
\on{ch}_{c,\A}:\CHom_{\Shv(X)^{\ren}}(\A,[c](\A))\to  \Tr(\Shv(X)^{\ren},[c]).
\]

\smallskip

(c) The first goal of the paper is to associate to a correspondence $c$ a {\em true local terms} map
\[
\LT^{\true}_{c}: \Tr(\Shv(X)^{\ren},[c])\to \Gm(\Fix(c),\om_{\Fix(c)})
\]
such that the map of part (a) decomposes as
\[
{\cal Tr}_{c,\A}\simeq \LT^{\true}_{c}\circ\on{ch}_{c,\A}.
\]

\smallskip

(d) Slightly abusing the notation, we also denote the composition
\[
\Tr(\Shv(X),[c])\overset{\Tr(\ren_X,[c])}{\lra} \Tr(\Shv(X)^{\ren},[c])\overset{\LT^{\true}_{c}}{\lra} \Gm(\Fix(c),\om_{\Fix(c)})
\]
by $\LT^{\true}_{c}$ and call it the {\em true local terms} map  as well.
\end{Emp}

\begin{Emp}
{\bf Functoriality.} \hfill

\smallskip

The main technical result of the paper asserts that the true local term maps commute with proper pushforwards, smooth pullbacks and restrictions to closed subschemes $Z\subseteq X$ such that $c$ is {\em contracting near $Z$}, which we are going to formulate now (see \rs{functor} for a 
slightly more general assertion and more details): Consider the commutative diagram

\begin{equation} \label{Eq:morph of corr}
\begin{CD}
C @>c >> X\times X\\
@Vg VV @VVf\times f V\\
D @>d >> Y\times Y
\end{CD}
\end{equation}
of  Artin stacks of finite presentation over $k$, and denote by $g_{\Dt}:\Fix(c)\to \Fix(d)$ the induced map.
\end{Emp}

\begin{Thm} \label{T:Main} \hfill

\smallskip

\noindent (a) Every commutative  diagram \form{morph of corr} such that morphisms $f$ and $g$ are proper and safe\footnote{See Section~\re{safe} what {\em safe morphism} means.} gives rise to a homotopy commutative diagram
\begin{equation*}
\begin{CD}
\Tr(\Shv(X)^{\ren},[c])@>\LT_{c}^{\true}>> \Gm(\Fix(c),\om_{\Fix(c)})\\
@V\Tr([f_!]) VV @VV(g_{\Dt})_!V\\
\Tr(\Shv(Y)^{\ren},[c])@>\LT_{d}^{\true}>> \Gm(\Fix(d),\om_{\Fix(d)}).
\end{CD}
\end{equation*}

\smallskip

\noindent (b) Every commutative diagram \form{morph of corr} such that either

\smallskip

(i) morphisms $f$ and $g$ are smooth of the same relative dimension and $g_{\Dt}$ is \'etale

\smallskip

\noindent or

\smallskip

(ii) $f$ is a closed embedding, the diagram is Cartesian and $d$ is contracting near $f(X)\subseteq Y$

\smallskip

\noindent gives rise to a homotopy commutative diagram
\begin{equation*}
\begin{CD}
\Tr(\Shv(X)^{\ren},[c])@>\LT_{c}^{\true}>> \Gm(\Fix(c),\om_{\Fix(c)})\\
@A\Tr([f^*]) AA @VVg_{\Dt}^*V\\
\Tr(\Shv(Y)^{\ren},[c])@>\LT_{d}^{\true}>> \Gm(\Fix(d),\om_{\Fix(d)}).
\end{CD}
\end{equation*}
\end{Thm}

\begin{Emp}
{\bf Remarks.}

\smallskip

(a) As in \cite{Va}, in order to show \rt{Main}(b)(ii), we show that the true local terms commute with specializations,
and we apply this in the case of a specialization to the normal cone. Furthermore, for further reference, we divide the commutation with specialization assertion into two and prove the commutation with nearby cycles and with extensions of scalars.

\smallskip

(b) Using observations of Section~\re{trueloc}(c), the commutation of true local terms with proper pushforwards and specializations generalize the corresponding results of \cite{Va}. On the other hand, the commutation with smooth pullbacks seems to be new even in the classical setting of \cite{Va}. \footnote{Recently, the commutation of true local terms with smooth pullbacks in the classical setting appeared in \cite{FYZ}.}
\end{Emp}

\begin{Emp}
{\bf The case of the Frobenius endomorphism.}

\smallskip

(a) Assume from now that $k=\Fqb$, but that $X$ is defined over $\Fq$, so that it carries the \emph{geometric Frobenius} endomorphism,
denoted $\Fr$. Then we  can associate to $X$ the groupoid $X(\fq)$, and hence the (classical) vector space
$\Func(X(\fq),\qlbar)$.

In addition, the endomorphism $\Fr$ induces continuous endofunctors
\[
\Fr_{\tri}=\Fr_*:\Shv(X)\to \Shv(X)\text{ and }\Fr_*:\Shv(X)^{\ren}\to \Shv(X)^{\ren},
\]
hence we can form traces $\Tr(\Shv(X),\Fr_*), \Tr(\Shv(X)^{\ren},\Fr_*)\in \Vect$.

\smallskip

(b) Let $X$ and $Y$ be a pair of Artin stacks as above, and let $f:X\to Y$ be a morphism between them defined over $\fq$.
Then $f$ gives rise  to a map of groupoids $f(\fq):X(\fq)\to Y(\fq)$, and hence to a
map
\[
f(\fq)_!:\Func(X(\fq),\qlbar)\to \Func(Y(\fq),\qlbar)
\]
of $\qlbar$-vector spaces, given by the summation along the fibers.

Also $f$ gives rise to a functor $f_!:\Shv(X)\to \Shv(Y)$, admitting a continuous right adjoint, given by $f^!$, and interchanging the Frobenius actions.

Therefore by the trace formalism (see \cite[Section~3.2]{GKRV}) functor $f_!$ induces a map
\[
\Tr(f_!,\Fr_*):\Tr(\Shv(X),\Fr_*)\to \Tr(\Shv(Y),\Fr_*)
\]
in $\Vect$. Moreover, the assignment $f_!\mapsto \Tr(f_!,\Fr_*)$ is compatible with compositions of morphisms.

Assume in addition that $f$ is safe. Then $f$ gives rise to a functor $f_!:\Shv(X)^{\ren}\to \Shv(Y)^{\ren}$, admitting a continuous right adjoint, given by $f^!$, and hence induces a map between traces
\[
\Tr(f_!,\Fr_*):\Tr(\Shv(X)^{\ren},\Fr_*)\to \Tr(\Shv(Y)^{ \ren},\Fr_*)
\]
in $\Vect$.

%(c) The assignment $X\mapsto \Tr(\Shv(X),\Fr)$ defines a functor $\Tr(\Shv(-),\Fr)$ of $infty$-categories $\Art_{\fq}\to\Vect_e$.
\end{Emp}

\begin{Emp} \label{E:sheaffunc} {\bf Local terms and the sheaf--function correspondence.}
Let again $X$ be as above (an Artin stack over $\Fqb$, but defined over $\Fq$).

\smallskip

(a) Then one can associate to $X$ two pairs
\[
\LT_{X}^{\naive}, \LT_{X}^{\true}:\Tr(\Shv(X),\Fr_*)\to \Func(X(\fq),\qlbar)
\]
and
\[
\LT_{X}^{\naive}, \LT_{X}^{\true}:\Tr(\Shv(X)^{\ren},\Fr_*)\to \Func(X(\fq),\qlbar)
\]
of naturally defined maps called the {\em naive local term} and {\em true local term} maps, respectively.

\smallskip

Namely, the true local terms maps $\LT_{X}^{\true}$ are simply the maps $\LT_{c}^{\true}$ (see Section~\re{trueloc}), corresponding to the correspondence
$(\Fr,\Id):X\to X\times X$, while the naive local terms map $\LT_{X}^{\naive}$ is characterized by the condition that for every point
$x\in X(\fq)$ corresponding to the morphism $\eta_x:\pt:=\Spec\fqbar\to X$ the composition
\[
\Tr(\Shv(X),\Fr_*) \overset{\LT_{X}^{\naive}}{\lra} \Func(X(\fq),\qlbar)\overset{\ev_x}{\lra}\qlbar
\]
is equal to the map of traces
\[
\Tr(\eta_x^*,\Fr_*):\Tr(\Shv(X),\Fr_*)\to \Tr(\pt,\Fr_*)\simeq\qlbar,
\]
induced by the pullback $\eta_x^*:\Shv(X)\to \Shv(\pt)=\Vect$ (and similarly for $\Shv(X)^{\ren}$).

\smallskip

(b) Notice that for every $\A\in \Shv(X)^{\constr}$ we have $\Fr_*(\A)\in \Shv(X)^{\constr}$, so
the Chern character map of Section~\re{trueloc}(b) has the form
\[
\on{ch}_{X,\A}: \CHom_{\Shv(X)^{\constr}}(\A,\Fr_*(\A))\to \Tr(\Shv(X)^{\ren},\Fr_*).
\]
Furthermore, the composition
\[
\LT_{X}^{\naive}\circ  \on{ch}_{X,\A}: \CHom_{\Shv(X)^{\constr}}(\A,\Fr_*(\A))\to \Func(X(\fq),\qlbar)
\]
recovers the Grothendieck ``sheaf-function correspondence''.
%We will recall these constructions in \re{naive} and \re{truefrob} below.
\end{Emp}

As an application of \rt{Main} we prove the following result, stated as \cite[Theorems~22.1.9 and 22.2.8]{main}.

\begin{Thm} \label{T:local terms} \hfill

\smallskip

\noindent \em{(a)} For every $X$ as above, we have natural homotopies of morphisms
\[
\LT_{X}^{\naive}\simeq \LT_{X}^{\true}:\Tr(\Shv(X),\Fr_*)\to \Func(X(\fq),\qlbar)
\]
and
\[
\LT_{X}^{\naive}\simeq \LT_{X}^{\true}:\Tr(\Shv(X)^{\ren},\Fr_*)\to \Func(X(\fq),\qlbar)
\]

\smallskip

\noindent (b) The naive local term functor
\[
\LT^{\naive}:\Tr(\Shv(-),\Fr)\to \Func(-(\fq),\qlbar) \text{ (resp. }\LT^{\naive}:\Tr(\Shv(-)^{\ren},\Fr)\to \Func(-(\fq),\qlbar)\text{)}
\]
commutes with all $!$-pushforwards (resp. $!$-pushforwards with respect to the safe morphisms).

\smallskip

Namely, for every morphism $f:X\to Y$, the following diagram commutes up to a canonical homotopy:

\begin{equation} \label{Eq:LTF}
\begin{CD}
\Tr(\Shv(X),\Fr)@>\LT_{X}^{\naive}>>\Func(X(\fq),\qlbar)\\
@V\Tr(f_!,\Fr) VV @VVf(\fq)_! V\\
\Tr(\Shv(Y),\Fr)@>\LT_{Y}^{\naive}>>\Func(Y(\fq),\qlbar),
\end{CD}
\end{equation}
%
%that is, there exists a canonical homotopy
%\[
%f(\fq)_!\circ \LT_{X}^{\naive}\simeq \LT_{Y}^{\naive}\circ \Tr(f_!,\Fr).
%\]
and the corresponding result for $\Shv(-)^{\ren}$ holds when $f$ is safe.
\end{Thm}

Combining \rt{local terms}(a) and Section~\re{sheaffunc}(b), we get the following consequence used in \cite{trace}.

\begin{Cor} \label{C:sheaf function}
For every Artin stack $X$ over $\fqbar$, defined over $\fq$ and every $\A\in \Shv(X)^{\constr}$ the composition
\[
\LT_{X}^{\true}\circ  \on{ch}_{X,\A}: \CHom_{\Shv(X)^{\constr}}(\A,\Fr_*(\A))\to \Func(X(\fq),\qlbar)
\]
equals the Grothendieck ``sheaf-function correspondence'' map.
\end{Cor}

\begin{Emp}
{\bf Strategy of the proof.} Let us explain how to deduce \rt{local terms} from \rt{Main}.

\smallskip

(i) Notice that the correspondence $(\Fr,\Id)$ is contracting near  every closed substack $Z\subseteq X$ defined over $\fq$. Therefore it follows from \rt{Main} that the true local terms maps $\LT_X^{\true}$ commute with $!$-pushforwards with respect to proper safe morphisms and
$*$-pullbacks with respect to morphisms, which are either smooth or closed embeddings.

\smallskip

(ii) Next we observe that in order to show \rt{local terms}(a), it suffices to show that true local terms commute with $*$-pullbacks for every $\eta_x:\pt\to X$  with $x\in X(\fq)$. Let $G_x:=\Aut_X(x)$ be the group of automorphisms of $x$. Then $\eta_x$ decomposes
as a composition of morphisms
\[
\eta_x:\pt\to \pt/(G_x)_{\red}\to \pt/(G_x)\to X,
\]
the first of which is smooth, the second one is a proper universal homemorphism, while the last one is a composition of an open and a closed embedding.  Therefore the commutation of  $\LT_X^{\true}$ with $\eta_x^*$ follows from the observations of (i).

\smallskip

(iii) Note that it suffices to show the assertion of \rt{local terms}(b) under an assumption that $Y=\pt$. Indeed,
using base change the assertion of \rt{local terms}(b) for $f$ is equivalent to the corresponding assertions for every fiber
$f^{-1}(x)\to\pt$ of $f$. Moreover, by additivity of traces and Noetherian induction, we can replace
$X$ by its open non-empty substack. Thus we can assume that $f$ decomposes as $X\to X'\to \pt$, where $X'$ is an affine scheme, and
$f':X\to X'$ is a gerbe, hence all fibers of $f'$ are classifying stacks $BG$. Thus it suffices to show the assertion when $X$ is either an affine
scheme or a classifying stack $BG$.

\smallskip

(iv) In the affine scheme case, by additivity we can assume that $X$ is projective, in which case the assertion follows part (a)
and the fact that $\LT_X^{\true}$ commute with $!$-pushforwards with respect to proper safe morphisms. In the case, when $X=BG$, we can separately consider the case when $G$ is connected and $G$ is finite. In the first case, the assertion follows from the $X=G$ case and Lang's theorem. In the second case, we can imbed $G$ into $GL_n$, and deduce the assertion from the $B(GL_n)$-case and the scheme case.
\end{Emp}

\begin{Emp}
{\bf Plan of the paper.} The paper is organized as follows:

\smallskip

In Section~1 we recall basic properties of DG categories of sheaves on Artin stacks,
mainly recalling results from \cite{main, serre}, and formulate standard properties of safe stacks and safe morphisms,
whose proofs are recalled in Appendix A.

In Section~2 we introduce correspondences and discuss properties of functors, induced by correspondences.

In Section~3 we introduce the true local terms map and its refined version.

In Section~4 we formulate properties of true local terms and deduce generalizations of \cite{Va}.

In Section~5 we deduce \rt{local terms} from results formulated in Section~4.

In Section~6 we state the result that true local terms commute with nearby cycles and extensions of scalars, and deduce the
assertion about contracting correspondences, formulated in Section~4, using deformation to the normal cone.

In Section~7 we provide proof of Propositions~\ref{P:compatpf} and \ref{P:compatpb}.

In Sections~8 and 9 we provide proofs of functorial properties of true local terms, formulated in Sections~3, 4 and 6.

Finally, in Appendix B we review properties of quasi-smooth maps, used earlier. 
\end{Emp}

\begin{Emp}
{\bf Acknowledgments.}
%The research of D.G. is supported by NSF grant DMS-2005475. 
The research of Y.V. was partially supported by the ISF grant 2091/21. 
We thank Tony Feng from whom we learned about the notion of pullable squares (introduced in \cite{FYZ}) which allowed us to make the proof of the commutation of true local terms with smooth pullbacks more conceptual.
\end{Emp}

\section{Sheaves on Artin stacks}

Let $k$ be an algebraically closed field. All Artin stacks will be assumed to be of finite presentation over $k$.

\begin{Emp} \label{E:sheaves}
{\bf Sheaves on Artin stacks.}

\smallskip

  %and $\cD$ the spectrum of a DVR over $k$
(a) As in \cite[Appendix~F]{main} to every Artin stack $X$ one associates a $\qlbar$-linear stable $\infty$-category $\Shv(X)$ of
ind-constructible $\qlbar$-sheaves. This category is compactly generated, thus dualizable, and we denote by $\Shv(X)^c\subseteq\Shv(X)^{\constr}\subseteq\Shv(X)$
the full subcategories of compact objects and of constructible sheaves, respectively.

\smallskip

(b) To every morphism $f:X\to Y$ of Artin stacks, one can associate two adjoint pairs $(f_!,f^!)$ and $(f^*,f_*)$ of functors between
$\Shv(X)$ and $\Shv(Y)$. The functors $f_!,f^!$ and $f^*$ are automatically continuous.
\end{Emp}

\begin{Emp} \label{E:ren}
{\bf Renormalized $*$-pushforward.}

\smallskip

(a) As in \cite[Section~A.2.3]{serre}, to every morphism $f:X\to Y$ of Artin stacks one associates the renormalized $*$-pushforward functor
\[
f_{\tri}:\Shv(X)\to \Shv(Y),
\]
defined as the unique continuous functor, whose restriction to $\Shv(X)^c$ is $f_*|_{\Shv(X)^c}$.

\smallskip

(b) By definition, we have  a natural morphism of functors
\[
\on{can}_f:f_{\tri}\to f_*,
\]
whose restriction to $\Shv(X)^c$ is an isomorphism. Moreover, $\on{can}_f$ is an isomorphism if and only if $f_*$ is continuous (or, equivalently, when $f^*$ preserves compact objects).

\smallskip

(c) For a composition $X\overset{f}{\to} Y\overset{g}{\to}Z$ of Artin stacks, we have a canonical morphism
\begin{equation} \label{Eq:compren}
g_{\tri}\circ f_{\tri}\to (g\circ f)_{\tri}
\end{equation}
of functors $\Shv(X)\to \Shv(Z)$, whose restriction to  $\Shv(X)^c$ is the morphism
\[
g_{\tri}\circ f_{\tri}|_{\Shv(X)^c}=g_{\tri}\circ f_{*}|_{\Shv(X)^c}\overset{\on{can}_g}{\lra} g_{*}\circ f_{*}|_{\Shv(X)^c}\simeq (g\circ f)_{*}|_{\Shv(X)^c}=(g\circ f)_{\tri}|_{\Shv(X)^c}.
\]
In particular, morphism \form{compren} is automatically an isomorphism, if $f_{*}(\Shv(X)^c)\subseteq \Shv(Y)^c$.

\smallskip

(c)' By construction, morphism \form{compren} can be characterized as the unique morphism making the following diagram homotopy commutative:
\[
\CD
g_{\tri}\circ f_{\tri} @>\form{compren}>> (g\circ f)_{\tri}\\
@V\on{can}_g\circ \on{can}_f VV @VV\on{can}_{g\circ f}V\\
g_{*}\circ f_{*} @>\sim>> (g\circ f)_{*}.
\endCD
\]

\smallskip

(d)  For every Cartesian diagram of Artin stacks
\begin{equation*} %\label{Eq:basic}
\CD
A @>a>> C\\
@VgVV @VVfV\\
B @>b>> D,
\endCD
\end{equation*}
we have a canonical morphism
\begin{equation} \label{Eq:bcren}
g_{\tri}\circ a^!\to b^!\circ f_{\tri}
\end{equation}
of functors $\Shv(C)\to\Shv(B)$, whose restriction to
$\Shv(C)^c$ is the composition
\[
 g_{\tri}\circ a^!|_{\Shv(C)^c}\overset{\on{can}_g}{\lra} g_{*}\circ a^!|_{\Shv(C)^c}\overset{\text{base change}}{\simeq}
  b^!\circ f_{*}|_{\Shv(C)^c}=b^!\circ f_{\tri}|_{\Shv(C)^c}.
  \]
  In particular, morphism \form{bcren} is automatically an isomorphism, if $a^!(\Shv(C)^c)\subseteq \Shv(A)^c$.

\smallskip

(d)' By construction, morphism \form{bcren} can be characterized as the unique morphism making the following diagram homotopy commutative:
\[
\CD
g_{\tri}\circ a^! @>\form{bcren}>>  b^!\circ f_{\tri}\\
@V\on{can}_g VV @VV\on{can}_{f}V\\
g_{*}\circ a^! @>\text{base change}>\sim>  b^!\circ f_{*}.
\endCD
\]

\smallskip

(e) For every Artin stack $X$, we denote by $p_X:X\to\pt:=\Spec k$ the projection, and write $\Gm(X,-):\Shv(X)\to \Vect$ instead of $(p_X)_*$ and
$\Gm_{\tri}(X,-):\Shv(X)\to\Vect$ instead of $(p_X)_{\tri}$. By part (b), we have a canonical morphism
$\Gm_{\tri}(X,-)\to \Gm(X,-)$ of functors $\Shv(X)\to\Vect$.
\end{Emp}

\begin{Emp} \label{E:safe}
{\bf Safe morphisms.}  Let $f:X\to Y$ be a morphism of Artin stacks.

\smallskip

(a) To every geometric point $x$ of $X$, one associates the automorphism group
\[
\Aut_f(x):=\Aut_{f^{-1}(f(x))}(x)
\]
of $x$, viewed as a point of the Artin stack $f^{-1}(f(x))$.

\smallskip

(b) Following  \cite[Definition~10.2.2]{DG}, we say that a morphism $f$ is {\em safe} if for every geometric point $x$ of $X$, the connected component of the reduced part of $\Aut_f(x)$ is unipotent. We will say that $X$ is {\em safe} if and only if the projection $X\to\pt$ is safe.

\smallskip

(c) Let $X\overset{f}{\to}Y\overset{g}{\to}Z$ be a pair of morphisms of Artin stacks such that $g$ is safe. Then morphism $f$ is safe if and only if $g\circ f$ is safe.

\smallskip

(d) By \rp{safe} below, a morphism $f$ is safe if and only if the functor $f_*$ is continuous. Therefore it follows from Section~\re{ren}(b) that
this happens if and only if the morphism of functors $\on{can}_f:f_{\tri}\to f_*$ is an isomorphism.

\smallskip

(e) Notice that every representable morphism is safe, and every geometric Frobenius morphism $\Fr:X\to X$ is safe.

\smallskip

(f) Note that if $f$ is a separated morphism, then all automorphism groups $\Aut_f(x)$ are proper. Therefore a separated morphism $f$ is
safe if and only if all automorphism groups $\Aut_f(x)$ are finite. For example, a separated morphism between Artin stacks with affine diagonals is automatically safe.

%(g) For every commutative diagram of Artin stacks
%\begin{equation*} %\label{Eq:basic}
%\CD
%A @>a>> C\\
%@VgVV @VVfV\\
%B @>b>> D,
%\endCD
%\end{equation*}
%such that $g$ is safe we have a canonical morphism

%
%\begin{equation*} %\label{Eq:BC2}
%f^*\circ b_{\tri}\to a_{\tri}\circ g^*
%\end{equation*}
%
%of functors $\Shv(B)\to\Shv(C)$, whose restriction to $\Shv(B)^c$ is a base change morphism
%\[
%f^*\circ b_{\tri}|_{\Shv(B)^c}=f^*\circ b_*|_{\Shv(B)^c}\overset{\text{base change}}{\lra} a_*\circ g^*|_{\Shv(B)^c}=a_{\tri}\circ g^*|_{\Shv(B)^c},
%\]
%where the last equality holds since $g^*(\Shv(B)^c)\subseteq\Shv(A)^c$ because $g$ is safe.
\end{Emp}

The following assertion is an analog of  \cite[Theorem~10.2.4 and Corollary~10.2.7]{DG}. For completeness, we provide its proof in Appendix \ref{S:pfsafe}.

\begin{Prop} \label{P:safe}

\noindent{(a)} The following properties of an Artin stack $X$ are equivalent:

\smallskip

(i) $X$ is safe;

(ii) the constant sheaf $\qlbar\in\Shv(X)$ is compact;

(iii) every constructible sheaf $\A\in\Shv(X)^{\constr}$ is compact.

\smallskip

\noindent{(b)} The following properties of a morphism $f:X\to Y$ of Artin stacks are equivalent:

\smallskip

(i) $f$ is safe;

(ii) the functor $f_*:\Shv(X)\to\Shv(Y)$ is continuous;

(iii) the functor $f_!:\Shv(X)\to\Shv(Y)$ satisfies
$f_!(\Shv(X))^{\constr})\subseteq\Shv(X)^{\constr}$.
\end{Prop}

The proof of the following result will be given in Appendix \ref{S:pfsafe} as well.

\begin{Cor} \label{C:safe}
For every proper safe morphism $f:X\to Y$  between Artin stacks, we have a natural isomorphism $f_!\simeq f_*$ of functors
$\Shv(X)\to \Shv(Y)$.
\end{Cor}

\begin{Emp} \label{E:ren and unren}
{\bf Renormalized category of sheaves.}

\smallskip

(a) Let $\Shv(X)^{\ren}:=\Ind\Shv(X)^{\constr}$ be the ind-completion of $\Shv(X)^{\constr}$. Notice that we have a pair of continuous adjoint functors $(\ren_X,\unren_X)$, where
\[
\ren_X:\Shv(X)\to\Shv(X)^{\ren}
\]
be the ind-completion of the inclusion $\Shv(X)^c\hra\Shv(X)^{\constr}$, and
\[
\unren_X:\Shv(X)^{\ren}\to\Shv(X)
\]
is characterised by the condition that
$\unren_X|_{\Shv(X)^{\constr}}$ is the inclusion $\Shv(X)^{\constr}\hra\Shv(X)$.

\smallskip

(b) As it is explained in \cite[Section~F.5.2]{main}, both categories $\Shv(X)$ and $\Shv(X)^{\ren}$ are equipped with perverse $t$-structures, and functor $\unren_X$ induces an equivalence
\[
(\Shv(X)^{\ren})^{\geq -n}\to\Shv(X)^{\geq -n}
\]
for all $n$.
\end{Emp}

\begin{Emp} \label{E:remsafe}
{\bf Remark.} Notice that the functor $\ren_X:\Shv(X)\to\Shv(X)^{\ren}$ is an equivalence of categories if and only if
the inclusion $\Shv(X)^c\hra\Shv(X)^{\constr}$ is an equivalence. Thus, by \rp{safe}(a), this happens if and only if $X$ is safe.
\end{Emp}

\begin{Emp} \label{E:functren}
{\bf Functors between renormalized categories.} Let $f:X\to Y$ be a morphism of Artin stacks.

\smallskip

(a) Note that both pullbacks $f^*,f^!:\Shv(Y)\to\Shv(X)$ map $\Shv(Y)^{\constr}$ to $\Shv(X)^{\constr}$, thus give rise to
unique continuous functors
\[
(f^*)^{\ren},(f^!)^{\ren}:\Shv(Y)^{\ren}\to\Shv(X)^{\ren},
\]
 extending
\[
f^*|_{\Shv(Y)^{\constr}},f^!|_{\Shv(Y)^{\constr}}:\Shv(Y)^{\constr}\to\Shv(X)^{\constr}.
\]
To simplify the notation, we will denote functors $(f^*)^{\ren}$ and $(f^!)^{\ren}$ by $f^*$ and $f^!$, respectively.

\smallskip

(b) By construction, the functor $f^*:\Shv(Y)^{\ren}\to\Shv(X)^{\ren}$ from part~(a) maps compact objects to compact objects, thus has a continuous right adjoint $f_*:\Shv(X)^{\ren}\to\Shv(Y)^{\ren}$.

\smallskip

(c) Note that the functor $f^!:\Shv(Y)^{\ren}\to\Shv(X)^{\ren}$ from part~(a) has a left adjoint $(f_!)^{\ren}$ if and only if functor
$f_!:\Shv(X)\to\Shv(Y)$ preserves constructible objects. Using \rp{safe}(b), this happens if and only if $f$ is safe. In this case, $(f_!)^{\ren}$ is equal to the unique continuous extension of
$f_!|_{\Shv(X)^{\constr}}:\Shv(X)^{\constr}\to \Shv(Y)^{\constr}$ and will be denoted simply by $f_!$.

%(d) In order to avoid confusion we will denote by $\Gm^{\ren}(X,-)$ the pushforward functor
%$(p_X)_*:\Shv(X)^{\ren}\to \Shv(\pt)^{\ren}=\Shv(pt)=\Vect$ and by $\om^{\ren}_X\in \Shv(X)^{\ren}$ the $p_X^!$-pullback of the constant sheaf.
\end{Emp}

\begin{Lem} \label{L:functren}
Let $f:X\to Y$ be a morphism of Artin stacks.

\smallskip

(a) We have natural isomorphisms
 \[
 f^*\circ\unren_Y\simeq  \unren_X\circ f^*\text{ and  }f^!\circ\unren_Y\simeq\unren_X\circ f^!
 \]
of functors $\Shv(Y)^{\ren}\to\Shv(X)$.

\smallskip

(b) We have a natural morphism
\[
\unren_Y\circ f_* \to f_*\circ \unren_X
\]
of functors $\Shv(X)^{\ren}\to\Shv(Y)$, whose restriction to $\Shv(X)^{\constr}$ is an isomorphism.

\smallskip

(c) We have a natural isomorphism
 \[
 f_{\tri}\simeq \unren_Y\circ f_*\circ \ren_X
 \]
of functors $\Shv(X)\to\Shv(Y)$.

\smallskip

(d) We have natural morphisms
\[
\ren_X\circ f^*\to f^*\circ\ren_Y, \ren_X\circ f^!\to f^!\circ\ren_Y,
 \text{ (resp. }\ren_Y\circ f_{\tri}\to f_*\circ\ren_X)
 \]
 of functors $\Shv(Y)\to\Shv(X)^{\ren}$ (resp.  $\Shv(X)\to\Shv(Y)^{\ren}$).
 \end{Lem}

 \begin{proof}
 (a) By continuity, it suffices to show isomorphisms between the corresponding functors
 $\Shv(Y)^{\constr}\to\Shv(X)$, which follow immediately from definitions.

\smallskip

 (b) The morphism $\unren_Y\circ f_* \to f_*\circ \unren_X$ is obtained by adjunction from the (iso)morphism
 $f^*\circ \unren_Y\to \unren_X\circ f^*$ from part (a). To show the isomorphism assertion, we have to show that for every
  $\A\in\Shv(X)^{\constr}$ and $\mathcal{B}\in \Shv(Y)^c$ the natural morphism
 \[
\CHom_{\Shv(Y)}(\mathcal{B},(\unren_Y\circ f_*)(\A))\to \CHom_{\Shv(Y)}(\mathcal{B},  (f_*\circ\unren_X)(\A))
\]
is an isomorphism. By adjunction, the above morphism is the composition of isomorphisms
 \[
\CHom_{\Shv(Y)}((f^*\circ\ren_Y)(\mathcal{B}),\A)\overset{\unren_X}{\simeq}\CHom_{\Shv(X)}((\unren_X\circ f^*\circ \ren_Y)(\mathcal{B}),\unren_X(\A)) \simeq\]
\[ \simeq\CHom_{\Shv(X)}(f^*(\mathcal{B}),\unren_X(\A)),
\]
where the first map is isomorphism because $\unren_X|_{\Shv(X)^{\constr}}$ is fully faithful, and the second isomorphism
is induced by the isomorphism
\[
\unren_X\circ f^*\circ \ren_Y\overset{(a)}{\simeq} f^*\circ\unren_Y\circ\ren_Y\overset{\on{unit}}{\simeq} f^*,
\]
where we recall that functor $\ren_Y$ is fully faithful.

\smallskip

(c)  By continuity, it suffices to construct an isomorphism between the corresponding functors $\Shv(X)^{c}\to\Shv(Y)$, and we define the corresponding isomorphism to be the composition
\[
\unren_Y\circ f_*\circ \ren_X|_{\Shv(X)^{c}}\overset{(b)}{\simeq}  f_*\circ\unren_X\circ \ren_X|_{\Shv(X)^{c}}\overset{\on{unit}}{\simeq}
  f_*|_{\Shv(X)^{c}}\simeq f_{\tri}|_{\Shv(X)^{c}}.
\]

\smallskip

(d) follows from parts (a),(c), adjunction and isomorphism $\unren_Y\circ\ren_Y\simeq\Id$.
\end{proof}

 \begin{Cor} \label{C:bcren}
For every Cartesian diagram of Artin stacks
\begin{equation*} %\label{Eq:basic}
\CD
A @>a>> C\\
@VgVV @VVfV\\
B @>b>> D,
\endCD
\end{equation*}

\smallskip

(a) we have a canonical (base change) isomorphism
\[
b^!\circ f_*\simeq g_*\circ a^!:\Shv(C)^{\ren}\to\Shv(B)^{\ren};
\]

\smallskip

(b) the following diagram of functors $\Shv(C)\to\Shv(B)^{\ren}$ is homotopy commutative
\begin{equation*} %\label{Eq:bcren}
\begin{CD}
\ren_B\circ g_{\tri}\circ a^! @>\ref{L:functren}(d)>> g_{*}\circ\ren_A\circ a^!  @>\ref{L:functren}(d)>>
g_{*}\circ a^!\circ \ren_C\\
@V\re{ren}(d) VV @. @V\sim V(a)V\\
\ren_B\circ b^!\circ f_{\tri}  @>\ref{L:functren}(d)>> b^!\circ \ren_D\circ f_{\tri} @>\ref{L:functren}(d)>>
b^!\circ f_{*}\circ \ren_C.
\end{CD}
\end{equation*}
\end{Cor}

\begin{proof}
(a) By continuity, it suffices to construct a canonical isomorphism
\[
b^!\circ f_*(\F)\simeq g_*\circ a^!(\F)
\]
for every $\F\in\Shv(C)^{\constr}$. In this case, both $b^!\circ f_*(\F)$ and $g_*\circ a^!(\F)$ lie in $(\Shv(B)^{\ren})^{\geq -n}$ for some $n$. Hence, by Section~
 \re{ren and unren}(b), it suffices to construct an isomorphism
\[
\unren_B\circ b^!\circ f_*(\F)\simeq \unren_B\circ g_*\circ a^!(\F).
\]
Since $\F\in\Shv(C)^{\constr}$, we get $a^!(\F)\in\Shv(A)^{\constr}$. It therefore follows from \rl{functren}(a),(b) that it suffices to construct an isomorphism  $b^!\circ f_*\simeq g_*\circ a^!$ of functors $\Shv(C)^{\constr}\to\Shv(B)$, which is well-known.

\smallskip

(b) Using definitions of morphisms in \rl{functren}(d), it suffices to show that the following diagram of functors
$\Shv(C)\to\Shv(B)$ is homotopy commutative:
\begin{equation*}
\begin{CD}
g_{\tri}\circ a^! @>\ref{L:functren}(c)>\sim> \unren_B\circ g_*\circ \ren_A \circ a^! @>\ref{L:functren}(d)>>\unren_B \circ g_*\circ a^!\circ \ren_C\\
@V\re{ren}(d) VV @. @V\sim V(a)V\\
b^!\circ f_{\tri}  @>\ref{L:functren}(c)>\sim>  b^!\circ \unren_D\circ f_*\circ \ren_C @>\ref{L:functren}(a)>\sim> \unren_B\circ b^!\circ f_*\circ \ren_C.
\end{CD}
\end{equation*}
For this, it suffices to evaluate all functors on objects of $\Shv(C)^c$, in which case the assertion follows from Section~\re{ren}(d)' by unwinding definitions of all  morphisms involved.
\end{proof}

\begin{Emp} \label{E:safe1}
{\bf Application to safe morphisms.}

\smallskip

(a) Note that if $f:X\to Y$ is a safe morphism of Artin stacks, then the morphism
\begin{equation} \label{Eq:isompb}
\ren_X\circ f^*\to f^*\circ\ren_Y
\end{equation}
of functors $\Shv(Y)\to\Shv(X)^{\ren}$  from \rl{functren}(d) is an isomorphism. Indeed, since $f^*(\Shv(Y)^c)\subseteq\Shv(X)^c$ (because $f$ is safe), the restrictions of
morphism \form{isompb} to $\Shv(Y)^c$ is the identity endomorphism of the composition $\Shv(Y)^c\overset{f^*}{\lra}\Shv(X)^{\constr}\subseteq \Shv(X)^{\ren}$.

\smallskip

(b) For every commutative diagram of Artin stacks
\begin{equation*} %\label{Eq:basic}
\CD
A @>a>> C\\
@VgVV @VVfV\\
B @>b>> D,
\endCD
\end{equation*}
such that $g$ is safe we have a canonical morphism
\begin{equation*} %\label{Eq:BC2}
f^*\circ b_{\tri}\to a_{\tri}\circ g^*
\end{equation*}
of functors $\Shv(B)\to\Shv(C)$, defined to be as a composition
\[
f^*\circ b_{\tri}\overset{\ref{L:functren}(c)}{\simeq} f^*\circ \unren_D\circ  b_*\circ \ren_B
\overset{\ref{L:functren}(a)}{\simeq}\unren_C\circ f^*\circ  b_*\circ \ren_B
\overset{\text{base change}}{\lra}
\]
\[
\unren_C\circ a_*\circ  g^*\circ \ren_B\overset{\form{isompb}}{\simeq} \unren_C\circ a_*\circ\ren_A \circ  g^*
\overset{\ref{L:functren}(c)}{\simeq} a_{\tri}\circ  g^*.
\]

\smallskip

(b)' Unwinding definitions, morphism of part~(b) can be characterized as the unique morphism making the following diagram homotopy commutative:
\[
\CD
f^*\circ b_{\tri} @>(b)>> a_{\tri}\circ g^*\\
@V\on{can}_b VV @VV\on{can}_{a}V\\
f^*\circ b_{*} @>\text{base change}>> a_{*}\circ g^*.
\endCD
\]

\end{Emp}

\begin{Emp} \label{E:dualizing}
{\bf The dualizing sheaf.} We denote by $\om^{\ren}_X\in \Shv(X)^{\ren}$ the image of the dualizing sheaf $\om_X\in\Shv(X)^{\constr}$ under the
embedding $\Shv(X)^{\constr}\hra \Shv(X)^{\ren}$. Then it follows from \rl{functren}(b) that we have a canonical isomorphism
\[
\Gm(X,\om_X):=(p_X)_*(\om_X)\simeq(p_X)_*(\om^{\ren}_X)=:\Gm(X,\om^{\ren}_X)
\]
between objects of
\[
\Shv(\pt)\simeq\Vect\simeq\Shv(\pt)^{\ren}.
\]
\end{Emp}

\section{Generalized base change morphisms}

\begin{Emp}
{\bf Set-up.} Note that every commutative diagram of Artin stacks
\begin{equation} \label{Eq:basic}
\CD
A @>a>> C\\
@VgVV @VVfV\\
B @>b>> D
\endCD
\end{equation}
decomposes as
\begin{equation} \label{Eq:basic2}
\CD
A @>p>> B\times_D C @>\wt{b}>> C\\
@VgVV @V\wt{f}VV @VVfV\\
B @=B @>b>> D.
\endCD
\end{equation}
\end{Emp}

\begin{Emp} \label{E:pushable}
{\bf Pushable squares.}

\smallskip

(a) Motivated by \cite[Definition~3.1.1]{FYZ}, we call a commutative diagram \form{basic} {\em pushable}, if the morphism $p$ from diagram \form{basic2} is proper and safe.

\smallskip

(b) In the situation of part~(a), we have a canonical morphism
\begin{equation*} %\label{Eq:BC1}
f_!\circ a_{\tri}\to b_{\tri}\circ g_!
\end{equation*}
of functors $\Shv(A)\to\Shv(D)$, defined to be the composition
\[
f_!\circ a_{\tri}\simeq f_!\circ \wt{b}_{\tri}\circ p_{\tri}\to b_{\tri}\circ \wt{f}_!\circ p_{\tri}\simeq b_{\tri}\circ g_!,
\]
where

\smallskip

\quad\quad$\bullet$ the first morphism is induced by the inverse of the isomorphism
\[
\wt{b}_{\tri}\circ p_{\tri}\isom (\wt{b}\circ p)_{\tri}=a_{\tri},
\]
from Section~\re{ren}(c), which is an isomorphism because $p$ is proper and safe that thus by \rco{safe} we have
\[
p_*(\Shv(A)^c)=p_!(\Shv(A)^c)\subseteq \Shv(B\times_D C)^c;
\]

\smallskip

\quad\quad $\bullet$ the second morphism is induced by the base change morphism
\[
 g_!\circ\wt{b}_{\tri}\to b_{\tri}\circ \wt{f}_!,
\]
obtained by adjunction from the composition
\[
\wt{b}_{\tri}\overset{\on{unit}}{\lra}\wt{b}_{\tri}\circ \wt{f}^!\circ \wt{f}_!\overset{\re{ren}(d)}{\lra} f^!\circ b_{\tri}\circ \wt{f}_!,
\]

\smallskip

\quad\quad $\bullet$ the last morphism is induced by the isomorphism
\[
\wt{f}_!\circ p_{\tri}\overset{\re{safe}(d)}{\simeq}\wt{f}_!\circ p_{*} \overset{\ref{C:safe}}{\simeq}\wt{f}_!\circ p_!\simeq f_!.
\]

\smallskip

(c) In the situation of part~(a), assume that morphism $f$ is safe. Hence  morphism $g$ is safe as well (see Section~\re{safe}(c)).
Replacing $(-)_{\tri}$ by $(-)_*$ in part~(b) in all places, we have a canonical morphism
\begin{equation*} %\label{Eq:BC1}
f_!\circ a_*\to b_*\circ g_!
\end{equation*}
of functors $\Shv(A)^{\ren}\to\Shv(D)^{\ren}$.

\smallskip

(d) Assume that morphism $f$ is proper and safe. Then, by Section~\re{safe}(c), a diagram \form{basic} is pushable if and only if morphism $g$ is proper and safe.

\smallskip

Moreover, it is not difficult to see that in this case, the morphism of part~(b) decomposes as
\[
f_!\circ a_{\tri}\overset{\re{safe}(d)\circ\ref{C:safe}}{\simeq} f_{\tri}\circ a_{\tri}\overset{\re{ren}(c)}{\lra}  (f\circ a)_{\tri}
\simeq (b\circ g)_{\tri} \overset{\re{ren}(c)}{\simeq} b_{\tri}\circ g_{\tri}\overset{\ref{C:safe}\circ\re{safe}(d)}{\simeq} b_{\tri}\circ g_!,
\]
while the morphism of part~(c) decomposes as
\[
f_!\circ a_{*}\overset{\ref{C:safe}}{\simeq} f_{*}\circ a_{*}\simeq(f\circ a)_{*}
\simeq (b\circ g)_{*}\simeq b_{*}\circ g_{*}\overset{\ref{C:safe}}{\simeq} b_{*}\circ g_!.
\]
\end{Emp}

\begin{Emp} \label{E:gysin}
{\bf Quasi-smooth morphisms and Gysin maps.}
\smallskip

(a) For an Artin stack $X$, an object $\A\in\Shv(X)$ and $n\in\B{Z}$, we set $\A\lan n\ran:=\A(n)[2n]$. More generally, for a locally constant function $\un{n}:X\to\B{Z}$, we denote by $\A\lan \un{n}\ran$ an object of $\Shv(X)$ such that for every connected component $X'\subseteq  X$, we have $\A\lan \un{n}\ran|_{X'}=\A\lan\un{n}|_{X'}\ran$.

\smallskip

(b) To every Artin stack $X$ we associate the dimension function
\[
\un{\dim}_X:X\to\B{Z},
\]
given by the formula $\un{\dim}_X(x)=\dim_x(X)$ for every $x\in X$.

To every morphism $f:X\to Y$ of Artin stacks, we associate the dimension function
$$\un{\dim}_f:X\to\B{Z},$$ given by the formula $\un{\dim}_f=\un{\dim}_X-f^{\cdot}(\un{\dim}_Y)$,
where $f^{\cdot}$ denotes pullback of functions.

For every composition $X\overset{f}{\to}Y\overset{g}{\to}Z$, we have an equality
\[
\un{\dim}_{g\circ f}=\un{\dim}_f+f^{\cdot}(\un{\dim}_g).
\]
In particular, $f$ is of relative dimension zero, that is, $\un{\dim}_f=0$, if and only if $\un{\dim}_{g\circ f}=f^{\cdot}(\un{\dim}_g)$.

\smallskip

(c) Note that if $f:X\to Y$ is quasi-smooth (also called lci), then the cotangent complex $T^*(X/Y)$ is perfect, and
the dimension function $\un{\dim}_f$ is locally constant and equals the Euler-characteristic of $T^*(X/Y)$.

Also for every composition $X\overset{f}{\to}Y\overset{g}{\to}Z$ of Artin stacks such that $g$ is smooth, the morphism $f$ is quasi-smooth if and only if $g\circ f$ is quasi-smooth.

\smallskip

(d) To every quasi-smooth morphism $f:X\to Y$ of Artin stacks (and more generally of derived Artin stacks) one can associate
the {\em relative fundamental class} map
\[
\on{cl}_f:f^*(\qlbar)\to f^!(\qlbar)\lan-\un{\dim}_f\rangle
\]
(see, for example, \cite[Construction~3.6]{Kh}), hence the {\em Gysin map}
\[
\on{Gys}_f:f^*\to f^!\lan-\un{\dim}_f\rangle
\]
of functors $\Shv(Y)\to \Shv(X)$, defined as a composition
\[
f^*=f^*\otimes f^*(\qlbar)\overset{\Id\otimes \on{cl}_f}{\lra} f^*\otimes f^!(\qlbar)\lan-\un{\dim}_f\rangle\overset{\on{can}}{\lra} f^!\lan-\un{\dim}_f\rangle,
\]
where $\on{can}$ is the canonical map $f^*(\CA)\otimes f^!(\CB)\to f^!(\CA\otimes\CB)$.

\smallskip

(e) By construction the map $\on{Gys}_f$ from part~(d) is a canonical isomorphism $f^*\isom f^!\lan-\un{\dim}_f\rangle$ when $f$ is smooth, and Gysin maps are compatible with compositions (by \cite[Theorem~3.12]{Kh}).

\smallskip

(f) Moreover, for every homotopy Cartesian diagram of Artin stacks
\form{basic} such that $f$ is quasi-smooth, the morphism $g$ is quasi-smooth, satisfies $\un{\dim}_g=a^{\cdot}(\un{\dim}_f)$, and the following
diagram is homotopy commutative:
\[
\begin{CD}
a^*\circ f^* @>\on{Gys}_f>> (a^*\circ f^!)\lan -a^{\cdot}\un{\dim}_f\ran)\\
@V\sim VV @VV\text{base change}V\\
g^*\circ b^* @>\on{Gys}_f>> (g^!\circ b^*)\lan -\un{\dim}_g\ran)
\end{CD}
\]
(by \cite[Theorem~3.13]{Kh}).
\smallskip

(g) Clearly, the Gysin map $\on{Gys}_f$ from part~(d) can be viewed as a morphism of functors $\Shv(Y)^{\constr}\to \Shv(X)^{\constr}$, and hence
as a morphism of functors $\Shv(Y)^{\ren}\to \Shv(X)^{\ren}$.
\end{Emp}

\begin{Emp} \label{E:pullable}
{\bf Pullable squares.}

\smallskip

(a) Slightly modifying \cite[Definition~3.1.1]{FYZ}, we call a commutative diagram \form{basic} {\em pullable}, if the morphism $p$ from diagram \form{basic2} is quasi-smooth.

\smallskip

(b) In the situation of part~(a), we have a canonical morphism
\begin{equation} \label{Eq:genbc}
g^* \circ b^!\to (a^! \circ f^*)\lan -\un{\dim}_p\ran
\end{equation}
of functors $\Shv(D)\to\Shv(A)$ and $\Shv(D)^{\ren}\to\Shv(A)^{\ren}$ defined to be the composition
\begin{equation*}
g^* \circ b^!\simeq p^*\circ \wt{f}^* \circ b^!\overset{\text{base change}}{\lra}p^*\circ \wt{b}^! \circ f^*\overset{\on{Gys}_p}{\lra} p^!\lan -\un{\dim}_p\ran\circ \wt{b}^! \circ f^*\simeq a^! \circ f^*\lan -\un{\dim}_p\ran,
\end{equation*}
where $\on{Gys}_p:p^*\to p^!\lan -\un{\dim}_p\ran$ is the Gysin map (see Section~\re{gysin}(d)).

\smallskip

(c) Assume that morphism $f$ is smooth. Then, by Sections~\re{gysin}(c), a diagram \form{basic} is pullable if and only if morphism $g$ is quasi-smooth. Moreover, using Sections~\re{gysin}(e),(b), it is not difficult to see that morphism \form{genbc} decomposes in this case as
\[
g^* \circ b^!\overset{\on{Gys}_g}{\lra} (g^!\circ b^!)\lan-\un{\dim}_g\ran\simeq (a^!\circ f^!)\lan-\un{\dim}_g\ran
\overset{\on{Gys}_f}{\simeq}(a^!\circ f^*)\lan-\un{\dim}_g+a^{\cdot}(\un{\dim}_f)\ran\simeq  a^! \circ f^*\lan -\un{\dim}_p\ran.
\]

\end{Emp}

\begin{Emp} \label{E:rempullable}
{\bf Remarks.} Though our notion of a pullable square is motivated by the corresponding notion of \cite[Definition~ 3.1.1]{FYZ}, the two notions are not equivalent. Namely, in \cite{FYZ} the authors consider  commutative diagram of derived Artin stacks and require that the induced map $A\to B\times^h_B C$ to the homotopy fiber product is quasi-smooth.

However, if one restricts to commutative diagrams of classical Artin stacks, then our notion is more general. Indeed, every pullable square in the sense of \cite{FYZ} is also pullable in our sense (by \rl{derived} below), but the converse is false. For example, if a commutative diagram is Cartesian, but not homotopy Cartesian, then it is pullable in our sense but not in the sense of \cite{FYZ} (by Section~\re{FYZ}(b) below).
\end{Emp}

\section{Functoriality of correspondences and traces}

\begin{Emp} \label{E:corr}
{\bf Correspondences.}

\smallskip

(a) By a {\em correspondence} $c$ on $X$, we will mean a morphism of Artin stacks
\[
c=(c_l,c_r):C\to X\times X,
\]
i.e., a diagram $X\overset{c_l}{\lla}C\overset{c_r}{\lra}X$.

\smallskip

(b) A correspondence $c:C\to X\times X$ gives rise to a Cartesian diagram
\[
\begin{CD}
\on{Fix}(c) @>\Dt_c>> C\\
@Vc_{\Dt}VV @VVcV\\
X @>\Dt_X>> X\times X,
\end{CD}
\]
where $\Dt_X:X\to X\times X$ is the diagonal morphism on $X$. We will refer to $\on{Fix}(c)$ as {\em the stack of fixed points} of $c$.

\smallskip

(c) A correspondence $c:C\to X\times X$ induces continuous functors
\[
[c]:=(c_l)_{\tri}\circ c_r^!:\Shv(X)\to\Shv(X) \text{ and } [c]:=(c_l)_{*}\circ c_r^!:\Shv(X)^{\ren}\to\Shv(X)^{\ren}.
\]
Since DG categories $\Shv(X)$ and $\Shv(X)^{\ren}$ are compactly generated thus dualizable, the trace formalism (see \cite[Section~3]{GKRV}) applies. In particular, one can associate to $c$ the vector spaces
\[
\Tr(\Shv(X),[c]), \Tr(\Shv(X)^{\ren},[c])\in\Vect,
\]
where we remind that $\Vect$ denotes the stable $\infty$-category of $\qlbar$-vector spaces.

\smallskip

(d) Furthermore, a correspondence $c:C\to X\times X$ gives rise to a lax-commutative diagram

\begin{equation} \label{Eq:ren}
\xy
(10,10)*+{\Shv(X)}="A";
(40,10)*+{\Shv(X)^{\ren}}="B";
(10,-10)*+{\Shv(X)}="C";
(40,-10)*+{\Shv(X)^{\ren},}="D";
{\ar@{->}^{\ren_X} "A";"B"};
{\ar@{->}_{\ren_X} "C";"D"};
{\ar@{->}_{[c]} "A";"C"};
{\ar@{->}^{[c]} "B";"D"};
{\ar@{=>}_{\al} "C";"B"}
\endxy
\end{equation}
where $\al$ is the composition
\[
\ren_X\circ [c]=\ren_X\circ (c_l)_{\tri}\circ c_r^!\to  (c_l)_{*}\circ\ren_C\circ c_r^!\to  (c_l)_{*}\circ c_r^!\circ \ren_X=[c]\circ\ren_X
\]
of morphisms from \rl{functren}(d).

Moreover, since functor $\ren_X$ has a continuous right adjoint (given by $\unren_X$) it induces a morphism of traces
\[
\Tr(\ren_X,[c]):\Tr(\Shv(X),[c])\to \Tr(\Shv(X)^{\ren},[c]).
\]

%(e) Note also that for every Artin stack $X$ we have a canonical morphism
%\[
%\Gm_{\tri}(X, \om_{X})\to \Gm^{\ren}(X, \om^{\ren}_{X}),
%\]
%defined as a composition
%\[
%\ren_{\pt}(p_X)_{\tri}p_X^!(\om_{\pt})\to (p_X)_{\tri}p_X^!(\ren_{\pt(\om_{\pt}))})
%\]
%from \rl{functren}(b).
\end{Emp}

\begin{Emp} \label{E:chern}
{\bf Chern character.}

\vskip 4truept

Following \cite[Section~3.5.4]{GKRV}, to a correspondence $c:C\to X\times X$ and a constructible
sheaf $\A\in \Shv(X)^{\constr}$ one associates the {\em Chern character} map
\[
\on{ch}_{c,\A}:\CHom_{\Shv(X)^{\ren}}(\A,[c](\A))\to  \Tr(\Shv(X)^{\ren},[c]).
\]

Namely, every point $u\in \CHom_{\Shv(X)^{\ren}}(\A,[c](\A))$ gives rise to a lax-commutative diagram
\begin{equation} \label{Eq:chern}
\xy
(10,10)*+{\Vect}="A";
(40,10)*+{\Shv(X)^{\ren}}="B";
(10,-10)*+{\Vect}="C";
(40,-10)*+{\Shv(X)^{\ren}.}="D";
{\ar@{->}^{\A} "A";"B"};
{\ar@{->}_{\A} "C";"D"};
{\ar@{=} "A";"C"};
{\ar@{->}^{[c]} "B";"D"};
{\ar@{=>}_{u} "C";"B"}
\endxy
\end{equation}
Moreover, since $\A\in\Shv(X)^{\constr}$ is a compact object in $\Shv(X)^{\ren}$, the corresponding morphism $\A:\Vect\to \Shv(X)^{\ren}$ has a right adjoint. Thus
diagram \form{chern} induces a morphism of traces
\[
\qlbar=\Tr(\Vect,\Id)\to \Tr(\Shv(X)^{\ren},[c]),
\]
hence defines a point of $\Tr(\Shv(X)^{\ren},[c])$.
\end{Emp}

\begin{Emp} \label{E:morphcorr}
{\bf Morphisms of correspondences.}
Let $c:C\to X\times X$ and $d:D\to Y\times Y$ be correspondences.

By a {\em morphism of correspondences} $c\to d$, we mean a pair of morphisms
$[f]=(f,g)$, making the following diagram commute:
\begin{equation} \label{Eq:morph}
\begin{CD}
X @<c_l<< C @>c_r>> X \\
@VfVV  @VgVV @VVfV\\
Y @<d_l<< D @>d_r>> Y.
\end{CD}
\end{equation}
%
%In particular, a morphism $[f]$ gives rise to a commutative diagram of Artin stacks
%
%\begin{equation} \label{Eq:morph2}
%\begin{CD}
%X @<\ov{c}_l<< C_l @<p_l<< C @>p_r>> C_r @>\ov{c}_r>> X \\
%@VfVV  @Vg_lVV @VgVV @Vg_rVV @VVfV\\
%Y @<d_l<< D @= D @= D @>d_r>> Y,
%\end{CD}
%\end{equation}
%
%in which the left and the right inner square is Cartesian.
\end{Emp}

\begin{Emp} \label{E:pf}
{\bf Pushforward.} Let $[f]:c\to d$ be a morphism of correspondences (see Section~\re{morphcorr}) such that the left inner square of diagram~\form{morph} is pushable (see Section~\re{pushable}).

\smallskip

Notice that this condition is satisfied if either

\smallskip

\quad\quad (i) morphisms $f$ and $g$ are proper and safe (see Section~\re{pushable}(d))

or

\quad\quad (ii) the left inner square of diagram~\form{morph} is Cartesian.

\smallskip

\smallskip

(a) In this case we have a natural morphism
\[
[f]_!:f_! \circ [c]\to [d]\circ f_!
\]
of functors $\Shv(X)\to\Shv(Y)$, defined as a composition
\[
f_! \circ [c] =f_!\circ (c_l)_{\tri} \circ c_r^!\to (d_l)_{\tri} \circ g_! \circ c_r^!\to (d_l)_{\tri} \circ d_r^! \circ f_!=[d]\circ f_!,
\]
where

\smallskip

$\bullet$ the first morphism is induced by a canonical morphism of functors
\begin{equation} \label{Eq:BC1}
f_!\circ (c_l)_{\tri}\to (d_l)_{\tri}\circ g_!
\end{equation}
from Section~\re{pushable}(b), corresponding to the left inner square of diagram~\form{morph};

\smallskip

$\bullet$ the second morphism is induced by the base change morphism  $g_! \circ c_r^!\to d_r^! \circ f_!$,
corresponding to the right inner square of diagram~\form{morph}.

\smallskip

Moreover, since functor $f_!:\Shv(X)\to\Shv(Y)$ has a continuous right adjoint (given by $f^!$), the morphism $[f]_!$ induces a map of traces
\[
\Tr([f]_!):\Tr(\Shv(X),[c])\to \Tr(\Shv(Y),[d]).
\]

\smallskip

 (b) Assume in addition that morphism $f$ is safe. Then by a version of the arguments of part (a) (which are much simpler now), we see that $[f]$ induces a canonical morphism
\[
[f]_!:f_! \circ [c]\to [d]\circ f_!
\]
of functors $\Shv(X)^{\ren}\to\Shv(Y)^{\ren}$, hence map of traces
\[
\Tr([f]_!):\Tr(\Shv(X)^{\ren},[c])\to \Tr(\Shv(Y)^{\ren},[d]).
\]

\smallskip

(c) In the situation of part (b), for every $\A\in\Shv(X)^{\constr}$ we have $f_!(\A)\in \Shv(Y)^{\constr}$ (by \rp{safe}(b)).
Next, the morphism $[f]_!$ of part~(b) gives rise to a map
\[
\CHom_{\Shv(X)^{\ren}}(\A,[c](\A))\overset{f_!}{\lra}\CHom_{\Shv(Y)^{\ren}}(f_!(\A),f_!([c](\A))) \overset{[f]_!}{\lra}
\CHom_{\Shv(Y)^{\ren}}(f_!(\A),[d](f_!(\A))),
\]
which we denote again by $[f]_!$.

\smallskip

Moreover, unwinding definitions, it follows from compatibility of trace maps with compositions that the following diagram is homotopy commutative:
\begin{equation} \label{Eq:chernpf}
\begin{CD}
\CHom_{\Shv(X)^{\ren}}(\A,[c](\A))@>\on{ch}_{c,\A}>> \Tr(\Shv(X)^{\ren},[c])\\
@V[f]_! VV @VV\Tr([f]_!)V\\
\CHom_{\Shv(X)^{\ren}}(f_!(\A),[d](f_!(\A)))@>\on{ch}_{d,f_!(\A)}>> \Tr(\Shv(Y)^{\ren},[d]).
\end{CD}
\end{equation}
\end{Emp}

The following result, whose proof will be given in \rs{proofs1}, asserts that morphisms from Sections~\re{pf}(a) and \re{pf}(b) are compatible with renormalization functors.

\begin{Prop} \label{P:compatpf}
In the situation of Section~\re{pf}(b), we have a homotopy commutative diagram
\begin{equation} \label{Eq:tracespf}
\begin{CD}
\Tr(\Shv(X),[c])@>\Tr([f]_!)>> \Tr(\Shv(Y),[d])\\
@V\Tr(\ren_X,[c]) VV @VV\Tr(\ren_Y,[d]) V \\
\Tr(\Shv(X)^{\ren},[c])@>\Tr([f]_!)>> \Tr(\Shv(Y)^{\ren},[d]).
\end{CD}
\end{equation}
\end{Prop}

\begin{Emp} \label{E:pb}
{\bf Pullback.} Let $[f]:c\to d$ be a morphism of correspondences (see Section~\re{morphcorr})
such that such that the right inner square of diagram~\form{morph} is pullable in the sense of Section~\re{pullable}.

\smallskip

Notice that this condition is satisfied if either

\smallskip

\quad (i) $f$ is smooth, $g$ is quasi-smooth such that $\un{\dim}_g=c_r^{\cdot}(\un{\dim}_f)$ (see Section~\re{pullable}(c))

or

\quad (ii) the right inner square of diagram~\form{morph} is Cartesian.

\smallskip

(a) Assume in addition that morphisms $f$ and $g$ are safe. In this case we have a natural morphism
\[
[f]^*:f^*\circ [d] \to [c]\circ f^*
\]
of functors $\Shv(Y)\to \Shv(X)$, defined as a composition
\[
f^*\circ [d] =f^*\circ (d_l)_{\tri}\circ d_r^!\to (c_l)_{\tri} \circ g^* \circ d_r^!\to (c_l)_{\tri}\circ c_r^! \circ f^*=[c]\circ f^*,
\]
where

%Note first that we have a canonical isomorphism
%
%\begin{equation} \label{Eq:isompb}
%  f^*\circ\ren_Y\simeq\ren_X\circ f^*.
%\end{equation}

%Indeed, for every $\A\in\Shv(Y)^c$, we have $f^*(\A)\in \Shv(X)^c$, since $f$ is safe, and hence both $f^*\circ\ren_Y(\A)$ and
%$\ren_X\circ f^*(\A)$ are simply $f^*(\A)\in\Shv(X)^c\subseteq\Shv(X)^{\constr}\subseteq\Shv(X)^{\ren}$.

\smallskip

$\bullet$ the first morphism is induced by the morphism
\begin{equation} \label{Eq:BC2}
f^*\circ (d_l)_{\tri}\to (c_l)_{\tri}\circ g^*
\end{equation}
(see Section~\re{safe1}(b)), which is defined because $g$ is safe;

%which under the isomorphism \rl{functren}(c) corresponds to the composition
%\[
%f^*\circ\unren_Y\circ(d_l)_*\circ\ren_D \overset{\ref{L:functren}(a)}{\simeq} \unren_X\circ f^*\circ(d_l)_*\circ\ren_D\overset{\text{base change}}{\lra}\]
%\[
%\overset{\text{base change}}{\lra} \unren_X\circ (c_l)_*\circ g^*\circ\ren_D\overset{\form{isompb}}{\simeq}  \unren_X\circ (c_l)_*\circ \ren_C\circ g^*,
% \]
%while

\smallskip

$\bullet$ the second morphism is induced by a canonical morphism
\begin{equation} \label{Eq:bc}
g^* \circ d_r^!\to c_r^! \circ f^*,
\end{equation}
from Section~\re{pullable}(b).

\smallskip

Moreover, since functor $f^*$ has a continuous right adjoint (because $f$ is safe), morphism $[f]^*$ induces a map of traces
\[
\Tr([f]^*):\Tr(\Shv(Y),[d])\to\Tr(\Shv(X),[c]).
\]

\smallskip

(b) Even without the assumptions that morphisms $f$ and $g$ are safe one can show (slightly modifying the arguments of part (a))
that morphism $[f]$ induces a canonical morphism
\[
[f]^*:f^*\circ [d] \to [c]\circ f^*
\]
of functors $\Shv(Y)^{\ren}\to \Shv(X)^{\ren}$, hence a map of traces
\[
\Tr([f]^*):\Tr(\Shv(Y)^{\ren},[d])\to\Tr(\Shv(X)^{\ren},[c]).
\]

\smallskip

(c) In the situation of part (b), for every $\A\in\Shv(Y)^{\constr}$ we have $f^*(\A)\in \Shv(Y)^{\constr}$.
Then the morphism $[f]^*$ of part (b) gives rise to a map
\[
\CHom_{\Shv(Y)^{\ren}}(\A,[d](\A))\overset{f^*}{\lra}\CHom_{\Shv(X)^{\ren}}(f^*(\A),f^*([d](\A)))\overset{[f]^*}{\lra}
\CHom_{\Shv(Y)^{\ren}}(f^*(\A),[c](f^*(\A))),
\]
which we denote again by $[f]^*$.

\smallskip

Moreover, unwinding definitions, it follows from compatibility of trace maps with compositions that the following diagram is homotopy commutative
\begin{equation} \label{Eq:chernpb}
\begin{CD}
\CHom_{\Shv(Y)^{\ren}}(\A,[d](\A))@>\on{ch}_{d,\A}>> \Tr(\Shv(Y)^{\ren},[d])\\
@V[f]^* VV @VV\Tr([f]^*)V\\
\CHom_{\Shv(X)^{\ren}}(f^*(\A),[c](f^*(\A)))@>\on{ch}_{c,f^*(\A)}>> \Tr(\Shv(X)^{\ren},[c]).
\end{CD}
\end{equation}

\end{Emp}

The following result, whose proof will be given in \rs{proofs1}, asserts that morphisms  from Sections~\re{pb}(a) and \re{pb}(b) are compatible with functors $\ren_X$ and $\ren_Y$.

\begin{Prop} \label{P:compatpb}
In the situation of Section~\re{pb}(a), we have a homotopy commutative diagram:
\begin{equation} \label{Eq:tracespb}
\begin{CD}
\Tr(\Shv(Y),[d])@>\Tr([f]^*)>> \Tr(\Shv(X),[c])\\
@V\Tr(\ren_Y,[d]) VV @VV\Tr(\ren_X,[c]) V \\
\Tr(\Shv(Y)^{\ren},[d])@>\Tr([f]^*)>> \Tr(\Shv(X)^{\ren},[c]).
\end{CD}
\end{equation}
\end{Prop}

\begin{Emp} \label{E:restr}
{\bf Restriction to open and closed substacks.}
Let $c:C\to X\times X$ be a correspondence, $Z\subseteq X$  a closed substack and $U:=X\sm Z\subseteq X$ the complementary open substack.

\smallskip

(a) We denote by $c|_Z:c^{-1}(Z\times Z)\to Z\times Z$ and $c|_U:c^{-1}(U\times U)\to U\times U$
the restrictions of $c$ to $Z$ and $U$, respectively. Then the pair of inclusions $[i_c]=(Z\hra X,c^{-1}(Z\times Z)\hra C)$ defines a morphism of correspondences
$c|_Z\to c$, and $[j_c]=(U\hra X,c^{-1}(U\times U)\hra C)$ defines a morphism of correspondences
$c|_U\to c$.

\smallskip

(b) Note that $[i_c]$ satisfies assumptions of Section~\re{pf}(b), while $[j_c]$ satisfies assumptions of Section~\re{pb}(a). Therefore we have maps of traces
\[
\Tr([i_c]_!):\Tr(\Shv(Z),[c|_Z])\to\Tr(\Shv(X),[c]),\quad\Tr(\Shv(Z)^{\ren},[c|_Z])\to\Tr(\Shv(X)^{\ren},[c])
\]
and
\[
\Tr([j_c]^*):\Tr(\Shv(X),[c])\to\Tr(\Shv(U),[c|_U]),\quad\Tr(\Shv(X)^{\ren},[c])\to\Tr(\Shv(U)^{\ren},[c|_U]).
\]

(c) We say that a closed substack $Z\subseteq X$ is {\em $c$-invariant}, if we have an equality of schematic preimages $c_r^{-1}(Z)=c^{-1}(Z\times Z)$, that is, an inclusion $c_r^{-1}(Z)\subseteq c_l^{-1}(Z)$.  In this case, morphism $[i_c]$ satisfies assumptions of Section~\re{pb}(a), thus it induces maps of traces
\[
\Tr([i_c]^*):\Tr(\Shv(X),[c])\to\Tr(\Shv(Z),[c|_Z]),\quad \Tr(\Shv(X)^{\ren},[c])\to\Tr(\Shv(Z)^{\ren},[c|_Z]) .
\]

(d) In the situation of part (c), the open substack $U:=X\sm Z\subseteq X$ satisfies $c_l^{-1}(U)=c^{-1}(U\times U)$.
In this case, morphism $[j_c]$ satisfies assumptions of Section~\re{pf}(b), thus it induces maps of traces
\[
\Tr([j_c]_!):\Tr(\Shv(U),[c|_U])\to\Tr(\Shv(X),[c]),\quad \Tr(\Shv(U)^{\ren},[c|_U])\to\Tr(\Shv(X)^{\ren},[c]).
\]

(e) By a straightforward verification, in the situation of part~(c) isomorphisms $i^*\circ i_!\simeq\Id_Z$ and  $j^*\circ j_!\simeq \Id_U$
induce canonical homotopies
\[
\Tr([i_c]^*)\circ \Tr([i_c]_!)\simeq\Id\text{ and  }\Tr([j_c]^*)\circ \Tr([j_c]_!)\simeq \Id.
\]
\end{Emp}

\begin{Lem} \label{L:additivity}
 Let $Z\subseteq X$ be a closed $c$-invariant substack, and $U:=X\sm Z\subseteq X$.
 Then the map
\[
\Tr([j_c]_{!})\oplus \Tr([i_c]_{!}):\Tr(\Shv(U),c|_U)\oplus\Tr(\Shv(Z),c|_Z)\to \Tr(\Shv(X),c)
\]
is an isomorphism, whose inverse map is  $(\Tr([j_c]^*),\Tr([i_c]^*))$, and similarly for $\Shv(-)^{\ren}$.
 \end{Lem}
 \begin{proof}
Using the observations in Section~\re{restr}(e), it suffices to show that
\begin{equation} \label{Eq:fibered}
\Tr(\Shv(U),[c|_U])\overset{\Tr([j_c]_!)}{\lra}\Tr(\Shv(X),[c])\overset{\Tr([i_c]^*)}{\lra}\Tr(\Shv(Z),[c|_Z])
\end{equation}
is a fiber sequence in $\Vect$. By definition (see \cite[Section~3.2]{GKRV}), $\Tr([j_c]_!)$ is the composition
\[
\Tr(\Shv(U),[c|_U])\simeq\Tr(\Shv(U),j^*\circ j_!\circ [c|_U])\simeq \Tr(\Shv(X),j_!\circ [c|_U] \circ j^*)\simeq
\]
\[
\simeq \Tr(\Shv(X),j_!\circ j^*\circ [c])\overset{\on{counit}}{\lra} \Tr(\Shv(X),[c]),
\]
and similarly $\Tr([i_c]^*)$ is the composition
\[
\Tr(\Shv(X),[c])\overset{\unit}{\lra}\Tr(\Shv(X),i_!\circ i^*\circ [c])\simeq \Tr(\Shv(Z),i^*\circ [c] \circ i_!)\simeq
\]
\[
\simeq \Tr(\Shv(Z),i^*\circ i_!\circ [c|_Z])\simeq \Tr(\Shv(Z),[c]).
\]
Therefore sequence \form{fibered} is isomorphic to the sequence
\begin{equation} \label{Eq:fibered2}
\Tr(\Shv(X),j_!\circ j^*\circ [c])\overset{\on{counit}}{\lra}\Tr(\Shv(X),[c])\overset{\unit}{\lra}\Tr(\Shv(X),i_!\circ i^*\circ [c]),
\end{equation}
induced by the fiber sequence $j_!\circ j^*\to\Id\to i_!\circ i^*$. Hence \form{fibered2} is also a fiber sequence. The proof of $\Shv(-)^{\ren}$ is the same.
\end{proof}

\section{True local terms maps}

\begin{Emp} \label{E:goal}
{\bf Goals of this section.}

\smallskip

(a) To a correspondence $c:C\to X\times X$ we are going to  associate a canonically defined map
\[
\LT_c^{\true}:\Tr(\Shv(X)^{\ren},[c])\to  \Gm(\on{Fix}(c), \om_{\on{Fix}(c)}),
\]
called the {\em true local terms map}. Then we are going to consider the composition
\[
\Tr(\Shv(X),[c])\overset{\Tr(\ren_X,[c])}{\lra} \Tr(\Shv(X)^{\ren},[c])\overset{\LT_c^{\true}}{\lra} \Gm(\on{Fix}(c), \om_{\on{Fix}(c)}),
\]
denote it again by $\LT_c^{\true}$ and call it the {\em true local terms map} as well.

\smallskip

(b) Assume now that Artin stacks $X$, $C$ and $\Fix(c)$ are Verdier-compatible (see Section~\re{verdier} below). In this case,  we are going to  associate a canonically defined map
\[
\LT_{c,\tri}^{\true}:\Tr(\Shv(X),[c])\to  \Gm_{\tri}(\on{Fix}(c), \om_{\on{Fix}(c)})
\]
which we call the {\em refined true local terms map} such that the following diagram is homotopy commutative:
\begin{equation} \label{Eq:refined}
\begin{CD}
\Tr(\Shv(X),[c]) @>\LT_{c,\tri}^{\true}>>  \Gm_{\tri}(\on{Fix}(c), \om_{\on{Fix}(c)})\\
@V\Tr(\ren_X,[c])VV @VV\re{ren}(e)V\\
\Tr(\Shv(X)^{\ren},[c])@>\LT_c^{\true}>>  \Gm(\on{Fix}(c), \om_{\on{Fix}(c)}).
\end{CD}
\end{equation}
\end{Emp}

\begin{Emp} \label{E:terminology}
{\bf Remark on terminology.}
Our terminology slightly differs from that of \cite{main}. Namely, in \cite{main} only $\LT_{c,\tri}^{\true}$ is considered and it is called
{\em true local terms map} there. The reason why we call the map {\em refined} is because the (usual) true local terms map $\LT_c^{\true}$ is defined for an arbitrary correspondence, while in the situation of Section~\re{goal}(b), the true local terms map $\LT_c^{\true}$ has a natural lift to a map $\LT_{c,\tri}^{\true}$.
\end{Emp}

\begin{Emp} \label{E:obs}
We start with the following few observations: Let $X$ be an Artin stack of finite presentation over $k$. \hfill

\smallskip

(a) The category $\Shv(X)^{\ren}$ is canonically self dual, and the equivalence of categories
\[
\Shv(X)^{\ren}\to(\Shv(X)^{\ren})^\vee=\Fun_{\cont}(\Shv(X)^{\ren},\Vect),
\]
is given by the formula
\[
\mathcal{B}\mapsto \CHom_{\Shv(X)^{\ren}}(\B{D}_X(\mathcal{B}),-), \quad \mathcal{B}\in\Shv(X)^{\constr},
\]
where $\B{D}_X$ denotes the Verdier duality functor $(\Shv(X)^{\constr})^{\on{op}}\to \Shv(X)^{\constr}$.

\smallskip

The above equivalence corresponds to the pairing
\[
\ev_X: \Shv(X)^{\ren}\otimes\Shv(X)^{\ren}\to\Vect:\A\otimes \mathcal{B}\mapsto \Gm(X,\A\overset{!}{\otimes}\mathcal{B}).
\]

\smallskip

(b) Since the DG category $\Shv(X)^{\ren}$ is dualizable,
for every Artin stack $Y$ the map
\[
\Shv(Y)^{\ren}\otimes(\Shv(X)^{\ren})^\vee\to\Fun_{\cont}(\Shv(X)^{\ren},\Shv(Y)^{\ren}),
\]
given by the formula $\A\otimes \F\mapsto \A\otimes (\F(-))$, is an equivalence.

\smallskip

(c) We denote by
\[
\varphi:\Shv(Y)^{\ren}\otimes\Shv(X)^{\ren}\to \Shv(Y)^{\ren}\otimes(\Shv(X)^{\ren})^{\vee}\to\Fun_{\on{cont}}(\Shv(X)^{\ren},\Shv(Y)^{\ren})
\]
the composition of the equivalences from parts (a) and (b). Explicitly, for every  $\A\in\Shv(Y)^{\constr}$ and  $\mathcal{B}\in\Shv(X)^{\constr}$, we have
\[
\varphi(\A\otimes \mathcal{B})=\A\otimes \CHom_{\Shv(X)^{\ren}}(\B{D}_X(\mathcal{B}),-).
\]

(d) Consider the functor
\[
\boxtimes: \Shv(Y)^{\ren}\otimes\Shv(X)^{\ren}\to \Shv(Y\times X)^{\ren}
\]
such that for every  $\A\in\Shv(Y)^{\constr}$ and  $\CB\in\Shv(X)^{\constr}$ we have $\boxtimes(\A\otimes \CB)= \A\boxtimes \CB$.
This functor preserves compactness, hence admits a continuous right adjoint, to be denoted $\boxtimes^R$.
\end{Emp}

\begin{Emp} \label{E:right action}

In what follows, for $\CK\in \Shv(Y\times X)^{\ren}$ we define its \emph{action by right functors} to be
\[
\CK(-):=(p_Y)_{*}(\CK\overset{!}{\otimes}p_X^!(-))):\Shv(X)^{\ren}\to \Shv(Y)^{\ren}.
\]

\end{Emp}

\begin{Lem} \label{L:rboxtimes}
The right adjoint $\boxtimes^R: \Shv(Y\times X)^{\ren}\to \Shv(Y)^{\ren}\otimes\Shv(X)^{\ren}$ is characterized by the property
that
\[
\varphi\circ \boxtimes^R:\Shv(Y\times X)^{\ren}\to \Shv(Y)^{\ren}\otimes\Shv(X)^{\ren}\to \Fun_{\cont}(\Shv(X)^{\ren},\Shv(Y)^{\ren})
\]
is the action by right functors, e.i., the map $\CK\mapsto \CK(-)$ from Section~\re{right action}.
\end{Lem}

\begin{proof}
Unwinding the definitions, we need to
show that for $\A\in \Shv(Y)^{\constr}$ and $\CB\in\Shv(X)^{\constr}$, we have a canonical identification
\begin{equation} \label{e:boxR}
\CHom_{\Shv(Y\times X)^{\ren}}(\A\boxtimes \CB,\CK)\simeq \CHom_{\Shv(Y)^{\ren}}(\A,(p_Y)_{*}(\CK\overset{!}{\otimes}p_X^!(\B{D}_X(\CB)))).
\end{equation}

Now isomorphism \eqref{e:boxR} follows from the adjunction between
\[
p_Y^*(-)\overset{*}\otimes p_X^*(\CB)\simeq -\boxtimes \CB\text{ and }(p_Y)_*(-\overset{!}{\otimes}p_X^!(\B{D}_X(\CB))).
\]
\end{proof}

\begin{Emp} Let $u^{\ren}_X=u_{\Shv(X)^{\ren}}$ be the \emph{unit} object of $\Shv(X)^{\ren}\otimes \Shv(X)^{\ren}$, i.e., the object so that the
functor $\varphi(u_X):\Shv(X)^{\ren}\to\Shv(X)^{\ren}$ is the identity, and let $\om^{\ren}_X\in\Shv(X)^{\ren}$ be the dualizing sheaf.
\end{Emp}

\begin{Cor} \label{C:rboxtimes}
We have a canonical isomorphism $u^{\ren}_X\simeq \boxtimes^R((\Dt_X)_*(\om^{\ren}_X))$.
\end{Cor}

\begin{proof}
The assertion follows from the fact that the action of $(\Dt_X)_*(\om^{\ren}_X)\in\Shv(X\times X)^{\ren}$ by right functors
is the identity functor on $\Shv(X)^{\ren}$.
\end{proof}

\begin{Emp} \label{E:truelocal}
{\bf Construction.}

\smallskip

(a) Consider a lax commutative diagram
\begin{equation} \label{Eq:true}
\xy
(0,10)*+{\Vect}="A";
(40,10)*+{\Shv(X)^{\ren}\otimes \Shv(X)^{\ren}}="B";
(90,10)*+{\Shv(X)^{\ren}\otimes \Shv(X)^{\ren}}="C";
(130,10)*+{\Vect}="D";
(0,-10)*+{\Vect}="A'";
(40,-10)*+{\Shv(X\times X)^{\ren}}="B'";
(90,-10)*+{\Shv(X\times X)^{\ren}}="C'";
(130,-10)*+{\Vect,}="D'";
{\ar@{->}^{u^{\ren}_X} "A";"B"};
{\ar@{->}_{(\Dt_X)_*(\om^{\ren}_X)} "A'";"B'"};
{\ar@{->}^{\ev_X} "C";"D"};
{\ar@{->}_{\Gm\circ\Dt_X^!} "C'";"D'"};
{\ar@{->}^{[c]\otimes\Id} "B";"C"};
{\ar@{->}_{[c\times\Id]} "B'";"C'"};
{\ar@{->}^{\on{id}} "A";"A'"};
{\ar@{->}^{\on{id}} "D";"D'"};
{\ar@{->}^{\boxtimes} "B";"B'"};
{\ar@{->}^{\boxtimes} "C";"C'"};
{\ar@{=>}_{\al_l} "B";"A'"};
{\ar@{=>}_{\al_m} "C";"B'"};
{\ar@{=>}_{\al_r} "D";"C'"};
\endxy
\end{equation}
where

\quad $\bullet$ $\al_l$ corresponds to the morphism $\boxtimes(u^{\ren}_{X})\to (\Dt_X)_*(\om^{\ren}_X)$, obtained by adjunction from the isomorphism
$u^{\ren}_X\isom\boxtimes^R(\Dt_X)_*(\om^{\ren}_X)$ of \rco{rboxtimes};

\smallskip

\quad $\bullet$ $\al_m$ is the canonical isomorphism $\boxtimes\circ ([c]\otimes\Id)\isom [c\times\Id]\circ\boxtimes$;

\smallskip

\quad $\bullet$ $\al_r$ is the tautological isomorphism $\overset{!}{\otimes}\isom \Dt_X^!\circ\boxtimes$.

\smallskip

(b) By definition, the composition of the top arrow in the lax-commutative diagram \form{true} is $\Tr(\Shv(X)^{\ren},[c])$, while the composition of the bottom arrow is
$\Gm(X, \Dt_X^![c\times\Id]((\Dt_X)_*(\om^{\ren}_X)))$. We denote by $\LT_c^{\true}$ the composition
\[
\Tr(\Shv(X)^{\ren},[c])\overset{\form{true}}{\lra} \Gm(X, \Dt_X^![c\times\Id]((\Dt_X)_*(\om^{\ren}_X)))\simeq
\Gm(\on{Fix}(c), \om^{\ren}_{\on{Fix}(c)})\simeq \Gm(\on{Fix}(c), \om_{\on{Fix}(c)}),
\]
where the first map is induced by the lax commutative diagram \form{true}, the second map is the base change isomorphism, and
the last isomorphism is the one from Section~\re{dualizing}.
\end{Emp}

\begin{Emp}
{\bf Relation with the trace maps of  \cite{Va}.}

\smallskip

For every constructible sheaf $\A\in \Shv(X)^{\constr}$, we denote by ${\cal Tr}_{c,\A}$ the composition
\[
\CHom_{\Shv(X^{\ren})}(\A,[c](\A))\overset{\on{ch}_{c,\A}}{\lra} \Tr(\Shv(X)^{\ren},[c])\overset{\LT^{\true}_c}{\lra}
\Gm(\on{Fix}(c), \om_{\on{Fix}(c)}).
\]

Uniwinding the definition, the last map can be written as a composition
\[
\CHom_{\Shv(X^{\ren})}(\A,[c](\A))\simeq \Gm(X,\Dt_X^!(\B{D}_X(\A)\boxtimes [c](\A)))\simeq
\Gm(X, \Dt_X^![c\times\Id](\B{D}_X(\A)\boxtimes \A))\to
\]
\[
\to\Gm(X, \Dt_X^![c\times\Id]((\Dt_X)_*(\om^{\ren}_X)))\simeq
\Gm(\on{Fix}(c), \om_{\on{Fix}(c)}),
\]
where

\smallskip

$\bullet$ the first two isomorphisms are standard;

\smallskip

$\bullet$ the third map is induced by the canonical morphism $\B{D}_X(\A)\boxtimes \A\to (\Dt_X)_*(\om^{\ren}_X)$,
obtained by adjointness from the evaluation map $\B{D}_X(\A)\otimes \A)\to\om_X$;

\smallskip

$\bullet$ the last isomorphism is the one from Section~\re{truelocal}(b).

\smallskip

In particular, the map ${\cal Tr}_{c,\A}$ is an extensions to stacks and complexes of the  trace map of \cite[Section~1.2.2]{Va}, which in its turn is motivated by the Verdier pairing of Illusie \cite{Il}.
\end{Emp}

\begin{Emp} \label{E:verdier}
{\bf Verdier-compatible stacks.}

\smallskip

(a) Following \cite[Section~2.6]{main}, we say that an Artin stack $X$ is {\em Verdier-compatible}, if the Verdier duality preserves the subcategory of compact objects $\Shv(X)^c\subseteq \Shv(X)^{\constr}$.

\smallskip

(b) As it is shown in \cite[Theorem~F.2.8]{main} any Artin stack $X$ that can be covered by open substacks of the form $S/G$, where $S$ is an algebraic space of finite type over $k$ and $G$ is an algebraic group (of finite type) over $k$ is Verdier-compatible.\footnote{Though the assertion is only shown  under an assumptions that $S$ is a scheme and $G$ is affine, these assumptions are not needed. Namely, the assumption that $S$ is a scheme was never used, and the only place, where the assumption on $G$ was used, was to establish an isomorphism \cite[Formula~(F.4)]{main} asserting that if $\pi_{\pt}:\pt\to BG$ is the projection, then $(\pi_{\pt})_*(\qlbar)\simeq (\pi_{\pt})_!(\qlbar)[d]$ for some $d$. However, this assertion for a general $G$ is a formal consequence of the assertion for the affine $G$. Namely,
by Chevalley theorem, there exists a (unique) normal closed connected subgroup $H$ of $G$ such that $H$ is affine and $G/H$ is proper. Therefore the projection $\pt\to BG$ decomposes as a composition $\pt\to BH\to BG$, where the second morphism is proper. Therefore isomorphism
\cite[Formula~(F.4)]{main} for $G$ formally follows from that for $H$.} Note that this class of Artin stacks is closed under finite products and fiber products.

\smallskip

(c) Note that for every morphism $f:X\to Y$ between Verdier-compatible Artin stacks the functor $f_*$ preserves compact objects (because $f_!$ does), and   $f^!$ preserves compact objects if $f$ is safe (because $f^*$ does).

\smallskip

(d) Combining part (c) with Section~\re{ren}(c),(d), we see that the morphism \form{compren} is an isomorphism if $X$ and $Y$ are Verdier-compatible, and that the base change morphism \form{bcren} is an isomorphism if $A$ and $C$ are Verdier-compatible and $a$ is safe
(compare \cite[Section~A.3]{serre}).
\end{Emp}

Until end of this section we assume that $X$ is Verdier-compatible.

\begin{Emp} \label{E:reftrue}
{\bf Construction.}

\smallskip

(a) Since $X$ is assumed to be Verdier-compatible, the DG category $\Shv(X)$ is canonically self dual, and the equivalence of categories
\[
\Shv(X)\to\Shv(X)^\vee=\Fun_{\cont}(\Shv(X),\Vect)
\]
is given by the formula
\[
\CB\mapsto \CHom_{\Shv(X)}(\B{D}_X(\CB),-), \quad \CB\in\Shv(X)^c,
\]
where $\B{D}_X$ denotes the Verdier duality functor $(\Shv(X)^c)^{\on{op}}\to \Shv(X)^c$.

\smallskip

The above equivalence corresponds to the pairing
\[
\ev_X: \Shv(X)\otimes\Shv(X)\to\Vect:\A\otimes \CB\mapsto \Gm_{\tri}(X,\A\overset{!}{\otimes}\CB).
\]

\smallskip

(b) Let $u_X=u_{\Shv(X)}$ be the \emph{unit} object of $\Shv(X)\otimes \Shv(X)$. Then, as it was shown in \cite[Section~22.2.4]{main}, we have
a canonical isomorphism $u_X\simeq \boxtimes^R((\Dt_X)_*(\om_X))$, where $\boxtimes:\Shv(X)\otimes\Shv(X)\to \Shv(X\times X)$ is the exterior product functor, and  $\boxtimes^R$ is its right adjoint.\footnote{Alternatively, it can be shown by modifying the argument of \rco{rboxtimes}.
Namely, observations \re{obs} and the formulation of \rl{rboxtimes} will continue to hold without any changes if one replaces $\Shv(-)^{\ren}$ by $\Shv(-)$,   $\Shv(-)^{\constr}$ by $\Shv(-)^c$, $\Gm$ by $\Gm_{\tri}$ and $(p_Y)_*$ by $(p_Y)_{\tri}$ in all places, so it remains to show the
modified version of  formula \eqref{e:boxR}. For this we observe that since $\A$ and $\A\boxtimes \CB$ are compact, both sides in the above modified formula \eqref{e:boxR} commute with colimits in $K$, and hence we can assume that $K$ is compact. In this case, $\CK\overset{!}{\otimes}p_X^!(\B{D}_X(\CB))$ is compact as well (see \cite[Lemma~F.4.4]{main}), thus $(p_Y)_{\tri}$ can be replaced by $(p_Y)_{*}$, and we finish the proof as in \rl{rboxtimes}.}

\smallskip

(c) Consider a lax commutative diagram

\begin{equation} \label{Eq:reftrue}
\xy
(0,10)*+{\Vect}="A";
(40,10)*+{\Shv(X)\otimes \Shv(X)}="B";
(80,10)*+{\Shv(X)\otimes \Shv(X)}="C";
(120,10)*+{\Vect}="D";
(0,-10)*+{\Vect}="A'";
(40,-10)*+{\Shv(X\times X)}="B'";
(80,-10)*+{\Shv(X\times X)}="C'";
(120,-10)*+{\Vect,}="D'";
{\ar@{->}^{u_X} "A";"B"};
{\ar@{->}_{(\Dt_X)_*(\om_X)} "A'";"B'"};
{\ar@{->}^{\ev_X} "C";"D"};
{\ar@{->}_{\Gm_{\tri}\circ\Dt_X^!} "C'";"D'"};
{\ar@{->}^{[c]\otimes\Id} "B";"C"};
{\ar@{->}_{[c\times\Id]} "B'";"C'"};
{\ar@{->}^{\on{id}} "A";"A'"};
{\ar@{->}^{\on{id}} "D";"D'"};
{\ar@{->}^{\boxtimes} "B";"B'"};
{\ar@{->}^{\boxtimes} "C";"C'"};
{\ar@{=>}_{\al_l} "B";"A'"};
{\ar@{=>}_{\al_m} "C";"B'"};
{\ar@{=>}_{\al_r} "D";"C'"};
\endxy
\end{equation}
where the $2$-morphisms are defined as in Section~\re{truelocal}(a).

\smallskip

(d) By definition, the composition of the top arrow in diagram \form{reftrue} is $\Tr(\Shv(X),[c])$, while the composition of the bottom arrow is
$\Gm_{\tri}(X, \Dt_X^![c\times\Id]((\Dt_X)_*(\om_X)))$. In particular, the lax commutative diagram \form{reftrue} gives rise to the morphism
\[
\Tr(\Shv(X),[c])\to\Gm_{\tri}(X, \Dt_X^![c\times\Id]((\Dt_X)_*(\om_X))).
\]

(e) Assume in addition that Artin stacks $C$ and $\Fix(c)$ are Verdier-compatible. In this case, using Section~\re{verdier}(d),
we have a canonical (base change) isomorphism
\[
 \Gm_{\tri}(X, \Dt_X^![c\times\Id]((\Dt_X)_*(\om_X)))\simeq  \Gm_{\tri}(X, \Dt_X^! (c_{\tri}(\om_C)))\simeq
 \Gm_{\tri}(\on{Fix}(c), \om_{\on{Fix}(c)}).
 \]

(f) We denote by $\LT_{c,\tri}^{\true}$ the composition
\[
\Tr(\Shv(X),[c])\overset{(d)}{\lra} \Gm_{\tri}(X, \Dt_X^![c\times\Id]((\Dt_X)_*(\om_X)))\overset{(e)}{\simeq}
\Gm_{\tri}(\on{Fix}(c), \om_{\on{Fix}(c)}).
\]
\end{Emp}

The proof of the following result is obtained by unwinding the definitions and will be given in \rs{proofs}.

\begin{Prop} \label{P:refined}
The diagram \form{refined} is homotopy commutative.
\end{Prop}

\section{Functoriality of true local terms} \label{S:functor}

\begin{Emp} \label{E:functorfixed}
{\bf Notation.}

\smallskip

(a) Every morphism $[f]:c\to d$ of correspondences (see Section~\re{morphcorr}) induces a morphism $g_{\Dt}:\on{Fix}(c)\to \on{Fix}(d)$ between stacks of fixed points.

\smallskip

(b) Assume that the induced morphism $g_{\Dt}:\on{Fix}(c)\to \on{Fix}(d)$ is proper and safe. Then it follows from \rco{safe} and Sections~ \re{ren}(c), \re{safe}(d) that we have natural maps
\[
(g_{\Dt})_{!}:\Gm(\on{Fix}(c), \om_{\on{Fix}(c)})\simeq \Gm(\on{Fix}(d), (g_{\Dt})_{!}(\om_{\on{Fix}(c)}))\to \Gm(\on{Fix}(d), \om_{\on{Fix}(d)})
\]
and
\[
(g_{\Dt})_{!}:\Gm_{\tri}(\on{Fix}(c), \om_{\on{Fix}(c)})\simeq \Gm_{\tri}(\on{Fix}(d), (g_{\Dt})_{!}(\om_{\on{Fix}(c)}))\to \Gm_{\tri}(\on{Fix}(d), \om_{\on{Fix}(d)}).
\]

\smallskip

Furthermore, these maps extend to a homotopy commutative diagram:
\[
\CD
\Gm_{\tri}(\on{Fix}(c), \om_{\on{Fix}(c)}) @>(g_{\Dt})_{!}>> \Gm_{\tri}(\on{Fix}(d), \om_{\on{Fix}(d)})\\
@V\re{ren}(e)VV @VV\re{ren}(e)V\\
\Gm(\on{Fix}(c), \om_{\on{Fix}(c)}) @>(g_{\Dt})_{!}>> \Gm(\on{Fix}(d), \om_{\on{Fix}(d)}).
\endCD
\]

\smallskip

(c) Assume now that morphism $g_{\Dt}$ is safe and quasi-smooth of relative dimension zero, e.g. \'etale. Then we have a natural map

\[
g_{\Dt}^*:\Gm_{\tri}(\on{Fix}(d), \om_{\on{Fix}(d)})\to
\Gm_{\tri}(\on{Fix}(c), g_{\Dt}^*(\om_{\on{Fix}(d)}))\overset{\on{Gys}_{g_{\Dt}}}{\lra} \Gm_{\tri}(\on{Fix}(c), \om_{\on{Fix}(c)}),
\]
where the first morphism is induced by the composition
\[
\Gm_{\tri}(\on{Fix}(d),-)\overset{\on{unit}}{\lra} \Gm_{\tri}(\on{Fix}(d),-)\circ(g_{\Dt})_*\circ  g_{\Dt}^*\overset{\re{safe}(d)}{\simeq}
\Gm_{\tri}(\on{Fix}(d),-)\circ(g_{\Dt})_{\tri}\circ  g_{\Dt}^*\overset{\re{ren}(c)}{\lra}\Gm_{\tri}(\on{Fix}(c),-)\circ  g_{\Dt}^*,
\]
and $\on{Gys}_{g_{\Dt}}:g_{\Dt}^*\to g_{\Dt}^!$ is the Gysin map (see Section~\re{gysin}). Moreover, even without the assumption that $g_{\Dt}$ is safe, we have a natural map
\[
g_{\Dt}^*:\Gm(\on{Fix}(d), \om_{\on{Fix}(d)})\to \Gm(\on{Fix}(c), g_{\Dt}^*(\om_{\on{Fix}(d)}))\overset{\on{Gys}_{g_{\Dt}}}{\lra}\Gm(\on{Fix}(c), \om_{\on{Fix}(c)}).
\]

\smallskip

Furthermore, these maps extend to a homotopy commutative diagram:
\[
\CD
\Gm_{\tri}(\on{Fix}(d), \om_{\on{Fix}(d)}) @>g_{\Dt}^*>> \Gm_{\tri}(\on{Fix}(c), \om_{\on{Fix}(c)})\\
@V\re{ren}(e)VV @VV\re{ren}(e)V\\
\Gm(\on{Fix}(d), \om_{\on{Fix}(d)}) @>g_{\Dt}^*>> \Gm(\on{Fix}(c), \om_{\on{Fix}(c)}).
\endCD
\]

\end{Emp}

The following result, whose proof will be given in Sections~\ref{S:proofs} and \ref{S:proofs2}, asserts that
the {\em true local terms} commute with proper pushforwards and smooth pullbacks.

\begin{Thm} \label{T:true} \hfill

\smallskip

\noindent{\em(a)} In the situation of Section~\re{pf}(i), the following diagram commutes up to a canonical homotopy:
\begin{equation*}
\begin{CD}
\Tr(\Shv(X)^{\ren},[c]) @>\LT_{c}^{\true}>>\Gm(\on{Fix}(c), \om_{\on{Fix}(c)})\\
@V\Tr([f]_!) VV @VV(g_{\Dt})_!V\\
\Tr(\Shv(Y)^{\ren},[d]) @>\LT_{d}^{\true}>>\Gm(\on{Fix}(d), \om_{\on{Fix}(d)}).
\end{CD}
\end{equation*}

\smallskip

\noindent{\em(b)} In the situation of Section~\re{pb}(i), assume that morphism $g_{\Dt}$ is quasi-smooth of relative dimension zero. Then the following diagram commutes up to a canonical homotopy:
\begin{equation*}
\begin{CD}
\Tr(\Shv(Y)^{\ren},[d]) @>\LT_{d}^{\true}>>\Gm(\on{Fix}(d), \om_{\on{Fix}(d)})\\
@V\Tr([f]^*) VV @VVg_{\Dt}^*V\\
\Tr(\Shv(X)^{\ren},[c]) @>\LT_{c}^{\true}>>\Gm(\on{Fix}(c), \om_{\on{Fix}(c)}).
\end{CD}
\end{equation*}
\end{Thm}

\begin{Emp} \label{E:contracting}
{\bf Contracting correspondences.} Let $c:C\to X\times X$ be a correspondence, and let $Z\subseteq X$ be a $c$-invariant closed substack
(see Section~\re{restr}).

\smallskip

(a) Since $Z\subseteq X$ is $c$-invariant, we have $c_l^{\cdot}(\I_Z)\subseteq c_r^{\cdot}(\I_Z)$, where
$\I_Z\subseteq\O_X$ is the sheaf of ideals of $Z$, and $c_l^{\cdot}$ and $c_r^{\cdot}$ denote pullbacks of sheaves.

\smallskip

(b) Following \cite[Section~2.1.1]{Va}, we say that $c$ is {\em contracting near} $Z$, if there exists $n\in\B{N}$ such that
$c_l^{\cdot}(\I_Z)^n\subseteq c_r^{\cdot}(\I_Z)^{n+1}$.

\smallskip

(c) It was proved in \cite[Theorem~2.1.3(a)]{Va} that if $c$ {\em contracting near} $Z$, then the inclusion
$i_{\Dt}:\on{Fix}(c|_Z)_{\red}\hra \on{Fix}(c)_{\red}$ is an open embedding, so we have restriction functors (see Section~\re{functorfixed}(c))
\[
i_{\Dt}^*:\Gm(\on{Fix}(c), \om_{\on{Fix}(c)})\to \Gm(\on{Fix}(c|_Z), \om_{\on{Fix}(c|_Z)})
\]
and
\[
i_{\Dt}^*:\Gm_{\tri}(\on{Fix}(c), \om_{\on{Fix}(c)})\to \Gm_{\tri}(\on{Fix}(c|_Z), \om_{\on{Fix}(c|_Z)}) .
\]

\smallskip

(d) {\bf Example.} Assume that $X$ is defined over $\fq$, and let $c$ be the correspondence
\[(\Fr,\Id):X\to X\times X.\]
Then $c$ is contracting near every closed substack $Z\subseteq X$ defined over $\fq$.
\end{Emp}

The proof of the following result will be given in \rs{contracting}.

\begin{Thm} \label{T:contracting}
Assume that a correspondence $c:C\to X\times X$ is contracting near $Z\subseteq X$. Then
the true local terms map commute with $*$-restriction with respect to the embedding $i:Z\hra X$.
In other words, in the notation of Sections~\re{restr}(c) and \re{contracting}(c), the following diagram commutes up
to a canonical homotopy:
\begin{equation*}
\begin{CD}
\Tr(\Shv(X)^{\ren},[c]) @>\LT_{c}^{\true}>>\Gm(\on{Fix}(c), \om_{\on{Fix}(c)})\\
@V\Tr([i_c]^*) VV @VVi_{\Dt}^*V\\
\Tr(\Shv(Z)^{\ren},[c|_Z]) @>\LT_{c|_Z}^{\true}>>\Gm(\on{Fix}(c|_Z), \om_{\on{Fix}(c|_Z)}).
\end{CD}
\end{equation*}
\end{Thm}

\begin{Emp}
{\bf Remark.}
We expect that one can show that \rt{contracting} continues to hold if one replaces the assumption
that  {\em $c$ is contracting near $Z\subseteq X$} by a weaker assumption that $c$ has {\em no almost fixed points in the punctured tubular neighborhood of $Z$} (see \cite[Definition~4.4 and Theorem~4.10]{Va2}).
\end{Emp}

Combining Theorems \ref{T:true} and \ref{T:contracting} with the commutativity of diagram \form{chernpf} and \form{chernpb}, we get the following
corollary:

\begin{Cor} \label{C:true1} \hfill

\smallskip

\noindent(a) In the situation of \rt{true}(a), for every constructible sheaf $\A\in\Shv(X)^{\constr}$, the following diagram commutes up to a canonical homotopy:

\smallskip
\begin{equation*}
\begin{CD}
\CHom_{\Shv(X)^{\ren}}(\A,[c](\A)) @>{\cal Tr}_{c,\A}>>\Gm(\on{Fix}(c), \om_{\on{Fix}(c)})\\
@V[f]_! VV @VV(g_{\Dt})_!V\\
\CHom_{\Shv(Y)^{\ren}}(f_!(\A),[d](f_!(\A))) @>{\cal Tr}_{d,f_!(\A)}>>\Gm(\on{Fix}(d), \om_{\on{Fix}(d)}).
\end{CD}
\end{equation*}

\smallskip

\noindent(b)  In the situation of \rt{true}(b), for every constructible sheaf $\A\in\Shv(Y)^{\constr}$, the following diagram commutes up to a canonical homotopy:

\begin{equation*}
\begin{CD}
\CHom_{\Shv(Y)^{\ren}}(\A,[d](\A)) @>{\cal Tr}_{d,\A}>> \Gm(\on{Fix}(d), \om_{\on{Fix}(d)})\\
@V[f]^* VV @VVg_{\Dt}^*V\\
\CHom_{\Shv(X)^{\ren}}(f^*(\A),[c](f^*(\A))) @>{\cal Tr}_{c,f^*(\A)}>> \Gm(\on{Fix}(c), \om_{\on{Fix}(c)}).
\end{CD}
\end{equation*}

\smallskip

\noindent(c)  In the situation of \rt{contracting} for every constructible sheaf $\A\in\Shv(X)^{\constr}$, the following diagram commutes up to a canonical homotopy:
\begin{equation*}
\begin{CD}
\CHom_{\Shv(X)^{\ren}}(\A,[c](\A)) @>{\cal Tr}_{c,\A}>>\Gm(\on{Fix}(c), \om_{\on{Fix}(c)})\\
@V[i_c]^* VV @VVi_{\Dt}^*V\\
\CHom_{\Shv(Z)^{\ren}}(\A|_Z,[c|_Z](\A|_Z)) @>{\cal Tr}_{c|_Z,\A|_Z}>>\Gm(\on{Fix}(c|_Z), \om_{\on{Fix}(c|_Z)}).
\end{CD}
\end{equation*}
\end{Cor}

\begin{Emp}
{\bf Remark.} Notice that while parts (a) and (c) of \rco{true1} are extensions to stacks and complexes of the corresponding results of \cite{Va},
part (b) seems to be completely new.\footnote{Recently, a derived version of part~(b) also appeared in \cite[Proposition~4.5.4]{FYZ}.}
\end{Emp}

The following result, whose proof will be given in Sections~\ref{S:proofs} and \ref{S:proofs2}, asserts that Theorem \ref{T:true}  has
an analog for refined true local terms maps. This result is not needed for \rt{local terms} and is only included for completeness and future references.

\begin{Thm} \label{T:reftrue} Assume that Artin stacks $X, Y, C, D, \Fix(c)$ and $\Fix(d)$ are Verdier-compatible.

\smallskip

\noindent{\em(a)} In the situation of Section~\re{pf}(i), the following diagram commutes up to a canonical homotopy:
\begin{equation*}
\begin{CD}
\Tr(\Shv(X),[c]) @>\LT_{c,\tri}^{\true}>>\Gm_{\tri}(\on{Fix}(c), \om_{\on{Fix}(c)})\\
@V\Tr([f]_!) VV @VV(g_{\Dt})_!V\\
\Tr(\Shv(Y),[d]) @>\LT_{d,\tri}^{\true}>>\Gm_{\tri}(\on{Fix}(d), \om_{\on{Fix}(d)}).
\end{CD}
\end{equation*}

\smallskip

\noindent{\em(b)} In the situation of Section~\re{pb}(i)(b), assume that morphism $g_{\Dt}$ satisfies the assumptions of Section~\re{functorfixed}(c). Then the following diagram commutes up to a canonical homotopy:
\begin{equation*}
\begin{CD}
\Tr(\Shv(Y),[d]) @>\LT_{d,\tri}^{\true}>>\Gm_{\tri}(\on{Fix}(d), \om_{\on{Fix}(d)})\\
@V\Tr([f]^*) VV @VVg_{\Dt}^*V\\
\Tr(\Shv(X),[c]) @>\LT_{c,\tri}^{\true}>>\Gm_{\tri}(\on{Fix}(c), \om_{\on{Fix}(c)}).
\end{CD}
\end{equation*}
\end{Thm}

For completeness we also state a result (without proof) asserting that Theorem \ref{T:true} and \ref{T:reftrue} have a common refinement.

\begin{Thm} \label{T:refinedtrue} \hfill

\smallskip

(a) In the situation of \rt{reftrue}(a), we have a natural homotopy commutative cube:
\[
\xy
(0,0)*+{\Tr(\Shv(X),[c])}="A";
(25,15)*+{\Tr(\Shv(X)^{\ren},[c])}="B";
(0,-35)*+{\Tr(\Shv(Y),[d])}="C";
(25, -20)*+{\Tr(\Shv(Y)^{\ren},[d])}="D";
(55,0)*+{\Gm_{\tri}(\on{Fix}(c), \om_{\on{Fix}(c)})}="A'";
(80,15)*+{\Gm(\on{Fix}(c), \om_{\on{Fix}(c)})}="B'";
(55,-35)*+{\Gm_{\tri}(\on{Fix}(d), \om_{\on{Fix}(d)})}="C'";
(80,-20)*+{\Gm(\on{Fix}(d), \om_{\on{Fix}(d)})}="D'";
{\ar@{->}^{\Tr(\ren_X,[c])} "A";"B"};
{\ar@{->}_{\Tr([f]_!)} "A";"C"};
{\ar@{->}_{\Tr([f]_!)} "B";"D"};
{\ar@{->}_{\Tr(\ren_Y,[d])} "C";"D"};
{\ar@{->}^{\re{ren}(e)} "A'";"B'"};
{\ar@{->}^{(g_{\Dt})_!} "A'";"C'"};
{\ar@{->}^{(g_{\Dt})_!} "B'";"D'"};
{\ar@{->}_{\re{ren}(e)} "C'";"D'"};
{\ar@{->}^{\LT_{c,\tri}^{\true}} "A";"A'"};
{\ar@{->}^{\LT_{c}^{\true}} "B";"B'"};
{\ar@{->}_{\LT_{d,\tri}^{\true}} "C";"C'"};
{\ar@{->}_{\LT_{d}^{\true}} "D";"D'"};
\endxy
\]
whose boundary consists of homotopy commutative squares of Theorems~\ref{T:true}(a) and \ref{T:reftrue}(a), diagram \form{refined}, \rp{compatpf} and Section~\re{functorfixed}(b).

\smallskip

(b) In the situation of \rt{reftrue}(b), we have a natural homotopy commutative cube (similar to part (a)), whose boundary consists of homotopy commutative squares of Theorems~\ref{T:true}(b) and \ref{T:reftrue}(b), diagram \form{refined}, \rp{compatpb} and Section~\re{functorfixed}(c).
\end{Thm}

We finish this section by stating another result without a proof asserting that Theorem \ref{T:contracting} also have a refinement in which refined true local terms maps are taken into an account.

\begin{Thm} \label{T:refinedcontracting}
In the situation of \rt{contracting}, assume that $X$ and $C$ satisfy the assumption of Section~\re{verdier}(b).
Then the following diagram commutes up to a canonical homotopy:

\begin{equation*}
\begin{CD}
\Tr(\Shv(X),[c]) @>\LT_{c,\tri}^{\true}>>\Gm_{\tri}(\on{Fix}(c), \om_{\on{Fix}(c)})\\
@V\Tr([i_c]^*) VV @VVi_{\Dt}^*V\\
\Tr(\Shv(Z),[c|_Z]) @>\LT_{c|_Z,\tri}^{\true}>>\Gm_{\tri}(\on{Fix}(c|_Z), \om_{\on{Fix}(c|_Z)}).
\end{CD}
\end{equation*}
Moreover, the above diagram together with that of \rt{contracting}, diagram \form{refined}, \rp{compatpb} and Section~\re{functorfixed}(c) naturally  extends to a homotopy commutative cube (as in \rt{refinedtrue}(a)).
\end{Thm}

\section{Local terms and the Grothendieck--Lefschetz trace formula}
In this section we will prove \rt{local terms}. Our ground field $k$ will be $\Fqb$.
However, all Artin stacks $X$ and all morphisms $f:X\to Y$ that will appear in this
section will be assumed defined over $\Fq$, so that $X$ carries the geometric
Frobenius endomorphism $\Fr$ and $f$ intertwines endomorphisms $\Fr$ on $X$ and $Y$.

\begin{Emp} \label{E:naive}
{\bf $*$-Pullbacks and the naive local terms for $\Shv(-)^{\ren}$.} \hfill

\smallskip

(a) Notice that for every morphism $f:X\to Y$, the pullback functor
\[
f^*:\Shv(Y)^{\ren}\to \Shv(X)^{\ren}
\]
induces a map of traces
\[
\Tr(f^*,\Fr):\Tr(\Shv(Y)^{\ren},\Fr)\to \Tr(\Shv(X)^{\ren},\Fr),
\]
(see Section~\re{pb}(ii)). Furthermore, the assignment $f\mapsto \Tr(f^*,\Fr)$ is compatible with compositions.

\smallskip

(b) For every $x\in X(\fq)$, let $\Spec(\Fqb)=:\pt \overset{\eta_x}\to X$ be the corresponding Frobenius-equivariant morphism.
Hence, it gives rise to a map
\[
\Tr(\eta_x^*,\Fr):\Tr(\Shv(X)^{\ren},\Fr)\to \Tr(\Shv(\pt),\Fr)\simeq\qlbar.
\]

\smallskip

(c) We define the {\em naive local terms map}
\[
\LT_X^{\naive}:\Tr(\Shv(X)^{\ren},\Fr)\to \Func(X(\fq),\qlbar)
\]
as the unique map such that
\[
\eta_x(\fq)^*\circ \LT_X^{\naive}=\Tr(\eta_x^*,\Fr)\text{ for all }x\in X(\fq).
\]

\smallskip

(d) By definition, naive local terms commute with $*$-pullbacks.
\end{Emp}

\begin{Emp} \label{E:naive2}
{\bf The naive local terms for $\Shv(-)$.}

\smallskip

(a) By analogy with Section~\re{goal}, we denote by $\LT_X^{\naive}:\Tr(\Shv(X),\Fr)\to \Func(X(\fq),\qlbar)$ the composition
\[
\Tr(\Shv(X),\Fr)\overset{\Tr(\ren_X,\Fr)}{\lra}\Tr(\Shv(X)^{\ren},\Fr)\overset{\LT_X^{\naive}}{\lra}\Func(X(\fq),\qlbar)
\]
and also call it the {\em naive local terms map}.

\smallskip

(b) Using \rp{compatpb}, one can see that the map of part (a) can be also defined directly. Namely, the pullback map $\Tr(f^*,\Fr)$ between the $\Shv(-)$'s is defined for all safe morphisms, and the naive local terms map for $\Shv(-)$ is characterized by the formula of Section~\re{naive}(c) taking into an account that every morphism $\eta_x$ is representable, hence safe.

\smallskip

(c) By part (b), the map of part (a) commutes with $*$-pullbacks with respect to safe morphisms.
\end{Emp}

\begin{Emp} \label{E:truefrob}
{\bf The true local terms.} Consider correspondence
\begin{equation} \label{Eq:grfr}
c=(\Fr,\Id):X\to X\times X.
\end{equation}
The induced endofunctors $[c]$ of both $\Shv(X)$ and $\Shv(X)^{\ren}$ then coincide with $\Fr_*$, so we have identifications
\[
\Tr(\Shv(X)^{\ren},[c])\simeq \Tr(\Shv(X)^{\ren},\Fr_*)\text{ and }\Tr(\Shv(X),[c])\simeq \Tr(\Shv(X),\Fr_*).
\]

By \rl{frob} below, we can identify
$\Fix(c)\simeq X(\fq)$, thus we get an identification
\[
\Gm(\Fix(c),\om_{\Fix(c)})\simeq \Func(X(\fq),\qlbar).
\]

Therefore the maps $\LT_c^{\true}$ from Section~\re{goal}(a) can be interpreted as maps
\[
\Tr(\Shv(X)^{\ren},\Fr)\to \Func(X(\fq),\qlbar)\text{ and } \Tr(\Shv(X),\Fr)\to \Func(X(\fq),\qlbar),
\]
are will be denoted by $\LT_X^{\true}$.
\end{Emp}

\begin{Lem} \label{L:frob}
The stack of fixed points $\Fix(c)$ of the correspondence \form{grfr} is canonically identified with
the discrete Deligne--Mumford stack, corresponding to a finite groupoid $X(\fq)$.
\end{Lem}

\begin{proof}
Our argument is an almost a repetition of that of \cite[Lemma~3.3]{Va1}, where the assertion was shown under a certain
(unnecessary) assumption.

\smallskip

Note first that $Y:=\Fix(c)$ is an algebraic stack of finite presentation over $\fqbar$, defined over $\fq$.

\smallskip

\begin{Cl} \label{C:frob}
The natural maps
\[
\pi:Y(\fqbar[t]/(t^2))\to Y(\fqbar)\text{ and  }i:Y(\fqbar)\to Y(\fqbar[t]/(t^2)),
\] corresponding to the homomorphisms of $\fqbar$-algebras $\pi:\fqbar[t]/(t^2)\to\fqbar$ and  $i:\fqbar\to\fqbar[t]/(t^2)$, are equivalences of categories.
\end{Cl}

\begin{proof}
Since $\pi\circ i\simeq\Id$, it suffices to show that the composition
\[
i\circ\pi: Y(\fqbar[t]/(t^2))\to Y(\fqbar[t]/(t^2))
\]
is isomorphic to the identity. But this follows from the observation that for every $\wt{x}\in X(\fqbar[t]/(t^2))$, we have a canonical isomorphism
\[
\Fr(\wt{x})\simeq i\circ\pi\circ\Fr(\wt{x}).
\]
Namely, this assertion is well-known when $X$ is a scheme, and the general case follows from it.
\end{proof}

Let us now come back to the proof of the lemma. Using \rcl{frob} we conclude that  the diagonal morphism $\Dt_Y:Y\to Y\times Y$ is unramified,
hence $Y$ is a Deligne--Mumford stack (see \cite[Tag~06N3]{Stacks}). Using \rcl{frob} again we conclude that $Y$ is \'etale over $\fqbar$.

\smallskip

Note that we have a natural morphism $X(\fq)\to Y$ of Deligne--Mumford stacks, where $X(\fq)$ is viewed as a discrete Deligne--Mumford stack,  and it suffices to show that the induced functor $\psi:X(\fq)\to Y(\fqbar)$ on $\fqbar$-points is an equivalence of categories.

\smallskip

The fact that $\psi$ is fully faithful follows from (and is actually equivalent to) the first axiom of a stack (the sheaf axiom for $\Isom(x,y)$) applied to \'etale covers $\Spec\B{F}_{q^m}\to \Spec \fq$ for all $m\in\B{N}$.

\smallskip

Thus, it suffices to show that functor $\psi$ is essentially surjective. Fix any object $y\in Y(\fqbar)$. Then $y$ corresponds to a pair $(x,\phi)$, where $x\in X(\fqbar)$ and $\phi\in\Isom_{X(\fqbar)}(\Fr(x),x)$, and we want to show that there exists an object $x_0\in X(\fq)$ such that $\psi(x_0)$ is
isomorphic to $(x,\phi)$.

\smallskip

Choose $m$ such that $x$ is (a pull-back of) an object of $X(\B{F}_{q^m})$, and $\phi\in\Isom_{X(\B{F}_{q^m})}(\Fr(x),x)$. Then the norm
\[
N_m(\phi):=\phi\circ \Fr(\phi)\circ\ldots\circ\Fr^{m-1}(\phi):x=\Fr^m(x)\isom x
\]
defines an $\B{F}_{q^m}$-point of an algebraic group $H:=\Isom_{X}(x,x)$.

If $N_m(\phi)=1$, then the existence of $x_0$ is equivalent to the second
axiom of a stack applied to the \'etale covering $\Spec \B{F}_{q^m}\to\Spec \fq$.
To show the general case, note that $H$ is of finite type over $\fqbar$, hence an element $N_m(\phi)\in H(\B{F}_{q^m})$ is of finite order $d$, and thus $N_{dm}(\phi)=(N_m(\phi))^d=1$.
\end{proof}

\begin{Emp} \label{E:remtrue}
{\bf Remark.} Assume that $X$ is Verdier-compatible.

\smallskip

(a) Since the stack of fixed points $\Fix(c)\simeq X(\fq)$ is safe, the canonical map
$\Gm_{\tri}(\Fix(c),\om_{\Fix(c)})\to \Gm(\Fix(c),\om_{\Fix(c)})$ of Section~\re{ren}(e) is an isomorphism. Therefore it follows from
\rp{refined} that the refined true local terms map $\LT_{c,\tri}^{\true}$ coincides with $\LT_X^{\true}$.

\smallskip

(b) Assume in addition that $X$  has finitely many isomorphism classes of $\fqbar$-points. In this case,
it follows from a combination of  \cite[Remark~22.2.6]{main} and part (a) that the true local map $\LT_X^{\true}$ is automatically an isomorphism.
\end{Emp}

The proof of \rt{local terms} is based on the following corollary of Theorems~\ref{T:true} and \ref{T:contracting}:

\begin{Cor} \label{C:true}
The true local term functor
\[
\LT^{\true}:\Tr(\Shv(-)^{\ren},\Fr)\to \Func(-(\fq),\qlbar)\text{ (resp. }\LT^{\true}:\Tr(\Shv(-),\Fr)\to \Func(-(\fq),\qlbar)\text{)}
\]
commutes with:

\smallskip

\noindent{\em(a)} $!$-pushforwards with respect to proper safe morphisms;

\smallskip

\noindent{\em(b)} $*$-pullbacks with respect to smooth morphisms (resp. smooth safe morphisms);

\smallskip

\noindent{\em(c)} $*$-pullbacks with respect to closed embeddings.
\end{Cor}

\begin{proof}
The assertions (a) and (b) for $\Shv(-)^{\ren}$ follows from \rt{true}, while assertion (c) follows from a combination of Example \re{contracting}(d) and \rt{contracting}. Then the assertion for $\Shv(-)$ follows from a combination of an assertion for $\Shv(-)^{\ren}$
and Propositions~\ref{P:compatpf} and \ref{P:compatpb}.
\end{proof}

\begin{Emp}
\begin{proof}[Proof of \rt{local terms}(a)]
By the definition of the naive local term map, we have to show that for every $x\in X(\fq)$ there exists a canonically defined homotopy
\[
\eta_x(\fq)^*\circ \LT_X^{\true}\simeq \Tr(\eta_x^*,\Fr).
\]
In other words, we have to show that true local terms commute with $*$-pullback with respect to the morphism $\eta_x:\pt\to X$.

Let $Y\subseteq  X$ be the closure of $\eta_x(\pt)$, equipped with a reduced stack structure. Let $G_x:=\Aut_X(x)$
denote the group scheme of automorphisms of $x$, and let $G_{x,\red}$ be the underlying reduced group scheme of $G_x$. Then $\eta_x$ factors as
\begin{equation} \label{Eq:comp}
\eta_x:\pt\overset{\pr_1}{\lra}B(G_{x,\red})\overset{\pr_2}{\lra}B(G_x) \overset{\ov{\eta}_x}{\lra}Y\overset{i}{\lra}X,
\end{equation}
and it suffices to show that true local terms commute with $*$-pullbacks with respect to all the morphisms that
appear in the diagram \form{comp}.

\smallskip

The required assertions follow from \rco{true}. Namely, the assertion for $\pr_1$ follows by \rco{true}(b), the assertion for $i$
follows by \rco{true}(c). The assertion for $\pr_2$ follows by \rco{true}(a): indeed $\pr^*_2$ is an equivalence of
categories with inverse $(\pr_2)_!$, while $\pr_2$ is proper. Finally, the assertion for $\ov{\eta}_x:B(G_x)\to Y$ follows
again by \rco{true}(b), since $\ov{\eta}_x$ is an open embedding.
\end{proof}
\end{Emp}

\begin{Emp}
{\bf Proof of \rt{local terms}(b).} \hfill

\smallskip

First, we record the following particular case of \rl{additivity}:

\begin{Cor} \label{C:additivity}
Let $Z\subseteq X$ be a closed substack, and let $U\subseteq X$ be the complementary open.
Then the functors $i_!:\Shv(U)\to\Shv(X)$ and $j_!:\Shv(U)\to\Shv(X)$
induce an isomorphism of traces
\[\Tr(j_!,\Fr)\oplus\Tr(i_!,\Fr):\Tr(\Shv(U),\Fr)\oplus\Tr(\Shv(Z),\Fr)\to \Tr(\Shv(X),\Fr).\]
\end{Cor}

First we will show the assertion of \rt{local terms}(b) for $\Shv(-)$. We will carry out the proof in six steps.

\vskip 5truept
{\bf Step 1.}
Note first that the assertion for morphisms $g:Z\to X$ and $f:X\to Y$ implies the assertion for $f\circ g$.
Conversely, if morphism $\Tr(g_!,\Fr)$ is isomorphism, then the assertion for morphisms $g$ and $f\circ g$ implies the assertion for $f$.

%\begin{proof}
%Since $\Tr(g_!,\Fr)$ is isomorphism, it suffices to show that
%\[
%f(\fq)_!\circ \LT_X^{\naive}\circ \Tr(g_!,\Fr)\simeq \LT_Y^{\naive}\circ \Tr(f_!,\Fr)\circ \Tr(g_!,\Fr).
%\]
%Therefore the reduction follows from isomorphisms $(f\circ g)(\fq)_!\simeq f(\fq)_!\circ g(\fq)_!$ and
%\[\Tr(f_!,\Fr)\circ \Tr(g_!,\Fr)\simeq \Tr((f\circ g)_!,\Fr).\]
%\end{proof}

\vskip 5truept
{\bf Step 2.} We claim that we can assume that $Y=\pt$. More precisely, the assertion for $f:X\to Y$ is equivalent
to the assertions for morphisms  $f_y:X_y\to\pt$ for all $y\in Y(\fq)$, where $X_y:=X\times_Y \{y\}$.

\begin{proof}
For each $y\in Y(\fq)$, let $\eta_y:\pt\to Y$ and $\wt{\eta}_y:X_y\to X$ be the corresponding morphisms.
We want to show that
\begin{equation} \label{Eq:isomGL}
\eta_y(\fq)^*\circ f(\fq)_!\circ \LT_X^{\naive}\simeq \eta_y(\fq)^*\circ  \LT_Y^{\naive}\circ \Tr(f_!,\Fr).
\end{equation}
The LHS of formula \form{isomGL} is homotopic to
\[
f_y(\fq)_!\circ \wt{\eta}_y(\fq)^*\circ \LT_X^{\naive}\simeq  f_y(\fq)_!\circ \LT_{X_y}^{\naive}\circ \Tr(\wt{\eta}_y^*,\Fr).
\]
The RHS of formula \form{isomGL} is homotopic to
\[
\LT_{\pt}^{\naive}\circ \Tr(\eta_y^*,\Fr)\circ \Tr(f_!,\Fr) \simeq \LT_{\pt}^{\naive}\circ\Tr((f_y)_!,\Fr)\circ \Tr(\wt{\eta}_y^*,\Fr),
\]
by the base change isomorphism $\eta_y^*\circ f_!\simeq (f_y)_!\circ \wt{\eta}_y^*$. Hence homotopy \form{isomGL} would follow
once we establish homotopy
$$f_y(\fq)_!\circ \LT_{X_y}^{\naive}\simeq \LT_{\pt}^{\naive}\circ\Tr((f_y)_!,\Fr).$$
\end{proof}

\vskip 2truept
{\bf Step 3.} Let $Z\subseteq X$ be a closed substack, and let $U\subseteq X$ be the complementary open. Then
the assertion for morphism $f$ is equivalent to the assertions for morphisms $f|_U$ and $f|_Z$.
\begin{proof}
Let $g:U\sqcup Z\to X$ be the canonical map. By \rco{additivity}, the induced map
$$\Tr(g_!,\Fr):\Tr(\Shv(U\sqcup Z),\Fr)\to \Tr(\Shv(X),\Fr)$$ is an isomorphism. Hence, by Step 1, it suffices to show the assertion for morphisms $f\circ g$ and $g$.  Since $g$ is a monomorphism, the assertion for $g$ follows immediately from Step 2. Finally, the
assertion for morphism $f\circ g$ is equivalent to the assertions for morphisms $f|_U$ and $f|_Z$.
\end{proof}

\vskip 2truept
{\bf Step 4.} The assertion holds for representable morphisms.

\begin{proof}
By Step 2, it is enough to treat the case when $X$ is an algebraic space and $Y=\pt$. Then $X$ has an open subspace $U$, which is an affine scheme. Thus, by Step 3 and Noetherian induction, we can assume that $X$ is affine.
In this case, $X$ has a compactification $\ov{X}$, hence using Step 3 again, we can assume that $X$ is projective.
In this case, the assertion for the \emph{true} local terms map follows from \rco{true}(a).
Hence, the assertion for the \emph{naive} local terms map follows from \rt{local terms}(a), which has already been proved.
\end{proof}

\vskip 2truept
{\bf Step 5.} The assertion holds, if $X$ is a classifying stack $BG$ for an algebraic group $G$ and $Y=\pt$.

\begin{proof}
Note that the projection $\pi:\pt\to BG$ is schematic. Therefore the assertion for $\pi$ follows from Step 4.

Assume first that $G$ is connected. In this case, the map $\pi(\fq)_!$ is an isomorphism by Lang's theorem. Since both $\LT^{\true}_{BG}$ and $\LT^{\true}_{\pt}$ are isomorphisms by Remark \re{remtrue}(b), we conclude that $\LT^{\naive}_{BG}$ and $\LT^{\naive}_{\pt}$ are isomorphisms by \rt{local terms}(a), hence the map $\Tr(\pi_!,\Fr)$ is an isomorphism, because  the assertion holds for
$\pi$. Finally, since the assertion trivially holds for the composition $\pt\to BG\to \pt$, the assertion for the projection $BG\to\pt$ now follows from Step 1.

Assume next that $G$ is finite. In this case, there exists an embedding  $G\hra GL_n$. Then the projection $BG\to\pt$ factors as
$BG\to B(GL_n)\to\pt$. By Step 1, it suffices to prove the assertion for $BG\to B(GL_n)$ and $B(GL_n)\to\pt$.
The assertion for $B(GL_n)\to \pt$ follows because $GL_n$ is connected. The assertion for
$BG\to B(GL_n)$ follows from Step 4, because this map is schematic.

In the general case, let $G^0$ be the connected component of the identity of $G$, and let $\pi_0(G)$ be the group of connected components. Then the projection $BG\to \pt$ factors as $BG\to B(\pi_0(G))\to\pt$, and the geometric fiber of the first map is isomorphic to $B(G^0)$.
Therefore the assertion follows from Steps 1 and 2 and the particular cases, proven above.
\end{proof}

\vskip 2truept
{\bf Step 6.} The assertion holds in general.

\begin{proof}
By Step 2, we can assume that $Y=\pt$. Also we can replace $X$ by $X_{\red}$. Recall
that for every reduced algebraic stack $X$ such that $I_X\to X$ is quasi-compact, there exists a (canonical) dense open substack $U\subseteq X$, which is a gerbe (see \cite[Tag~06RC]{Stacks}). Hence, using Step 3 and Noetherian induction, we can assume that $X$ is a gerbe over an algebraic space $X'$. Then the projection $X\to\pt$ decomposes as $X\overset{\pi}{\to} X'\to\pt$. Then, by Steps 1 and 2, the assertion for $X$ follows from  that for $X'$  and geometric fibers of $\pi$. Since every geometric fiber of $\pi$ is isomorphism to $BG$ for some group scheme $G$, the assertion
therefore follows from Steps 4 and 5.
\end{proof}

\vskip 2truept

Finally, we show the assertion for $\Shv(-)^{\ren}$. Arguing as in Step 2, it suffices to show the assertion under an assumption that $Y=\pt$.
Since $f$ was assumed to be safe, we conclude that $X$ is safe, thus, by Remark \re{remsafe}, the functor $\ren_X:\Shv(X)\to\Shv(X)^{\ren}$ is an equivalence. Therefore the assertion follows from the already shown assertion for $\Shv(-)$.
\end{Emp}

\section{Deformation to the normal cone, and proof of \rt{contracting}} \label{S:contracting}
As in \cite{Va}, our proof is based on the observation that the true local term map commutes with specialization.
With future applications in mind, we will work in a slightly more general set-up.

\begin{Emp} \label{E:spec}
{\bf Specialization of correspondences.}
Let $\cD$ be the spectrum of a DVR over $k$, and let $\eta$ and $s$  be the generic and the special points of $\cD$, respectively.
We will consider Artin stacks (and correspondences between them) of finite presentation over $\cD$.

\smallskip

(a) Every correspondence $\wt{c}=(\wt{c}_l,\wt{c}_r):\wt{C}\to \wt{X}\times \wt{X}$ over $\cD$ gives rise to correspondences
\[
\wt{c}_{\eta}=(\wt{c}_{\eta,l},\wt{c}_{\eta,r}):\wt{C}_{\eta}\to \wt{X}_{\eta}\times \wt{X}_{\eta}\text{ over }\eta
\]
and
\[
\wt{c}_{s}=(\wt{c}_{s,l},\wt{c}_{s,r}):\wt{C}_{s}\to \wt{X}_{s}\times \wt{X}_{s}\text{ over }s.
\]

\smallskip

(b) Recall that we have the nearby cycle functor $\Psi_{\wt{X}}:\Shv(\wt{X}_{\eta})^{\constr}\to \Shv(\wt{X}_{s})^{\constr}$, which uniquely extends to a continuous functor $\Psi_{\wt{X}}:\Shv(\wt{X}_{\eta})^{\ren}\to \Shv(\wt{X}_{s})^{\ren}$, and similarly for $\wt{C}$.

\smallskip

(c) We have the base change morphism
\[
\Psi_{\wt{X}}\circ [\wt{c}_{\eta}]= \Psi_{\wt{X}}\circ (\wt{c}_{\eta,l})_*\circ \wt{c}_{\eta,r}^!\to (\wt{c}_{s,l})_*\circ \Psi_{\wt{C}}\circ \wt{c}_{\eta,r}^!\to
(\wt{c}_{s,l})_*\circ \wt{c}_{s,r}^!\circ \Psi_{\wt{X}}=[\wt{c}_{s}]\circ \Psi_{\wt{X}}.
\]
Since $\Psi_{\wt{X}}$ maps constructible sheaves to constructible sheaves, it has a continuous right adjoint, and therefore induces a map of traces
\[
\Tr(\Psi_{\wt{c}}):\Tr(\Shv(\wt{X}_{\eta})^{\ren},[\wt{c}_{\eta}])\to \Tr(\Shv(\wt{X}_{s})^{\ren},[\wt{c}_{s}]).
\]

\smallskip

(d) We also have the natural maps
\[
\Psi_{\on{Fix}(\wt{c})}:\Gm(\on{Fix}(\wt{c}_{\eta}), \om_{\on{Fix}(\wt{c}_{\eta})})\to \Gm(\on{Fix}(\wt{c}_{s}), \Psi_{\on{Fix}(\wt{c})}(\om_{\on{Fix}(\wt{c}_{\eta})}))\to
\Gm(\on{Fix}(\wt{c}_{s}), \om_{\on{Fix}(\wt{c}_{s})}).
\]
\end{Emp}

 \begin{Emp} \label{E:bc} {\bf Extension of scalars.} \hfill

\smallskip

(a) Let $c:C\to X\times X$ be a correspondence, $k'/k$ a field extension. Let $c':C'\to X'\times X'$ be the base change of $C$ to $k'$.
Let $\pi$ denote any of the morphisms $X'\to X$, $C'\to C$, etc.

\smallskip

(b) Since $\pi^*:\Shv(X)\to\Shv(X')$ maps constructible objects to constructible objects, it extends to a continuous functor
$\pi^*:\Shv(X)^{\ren}\to\Shv(X')^{\ren}$ admitting a continuous right adjoint.

\smallskip

(c) We have the base change morphism
\[
\pi^*\circ [c] =\pi^*\circ (c_l)_* \circ c_r^!\simeq (c'_l)_* \circ \pi^* \circ d_r^!\simeq (c'_l)_*\circ c'^!_r \circ \pi^*=[c']\circ\pi^*.
\]

Hence, we obtain a map of traces
\[
\Tr([\pi_c]^*):\Tr(\Shv(X)^{\ren},[c])\to\Tr(\Shv(X')^{\ren},[c']).
\]

\smallskip

(d) We also have a canonically defined map
\[
\pi^*_{\Dt}:\Gm(\on{Fix}(c), \om_{\on{Fix}(c)})\to \Gm(\on{Fix}(c'), \pi^*(\om_{\on{Fix}(c)}))\simeq \Gm(\on{Fix}(c'), \om_{\on{Fix}(c')}).
\]
\end{Emp}

The following result, whose proof will be given in \rs{proofs}, asserts that
true local terms commute with nearby cycles and extension of scalars.

\begin{Thm} \label{T:true2} \hfill

\smallskip

\noindent{\em(a)} In the situation of Section~\re{spec}, the following diagram commutes up to a canonical
homotopy:

\begin{equation*}
\begin{CD}
\Tr(\Shv(\wt{X}_{\eta})^{\ren},[\wt{c}_{\eta}]) @>\LT_{\wt{c}_{\eta}}^{\true}>>
\Gm(\on{Fix}(\wt{c}_{\eta}), \om_{\on{Fix}(\wt{c}_{\eta})})\\
@V\Tr(\Psi_{\wt{c}}) VV @VV\Psi_{\on{Fix}(\wt{c})} V\\
\Tr(\Shv(\wt{X}_{s})^{\ren},[\wt{c}_{s}]) @>\LT_{\wt{c}_{s}}^{\true}>>
\Gm(\on{Fix}(\wt{c}_{s}), \om_{\on{Fix}(\wt{c}_{s})}).
\end{CD}
\end{equation*}
\smallskip

\noindent{\em(b)} In the situation of Section~\re{bc}, the following diagram commutes up to a canonical
homotopy:
\begin{equation*}
\begin{CD}
\Tr(\Shv(X)^{\ren},[c]) @>\LT_{c}^{\true}>>\Gm(\on{Fix}(c), \om_{\on{Fix}(c)})\\
@V\Tr([\pi_c]^*) VV @VV \pi_{\Dt}^*V\\
\Tr(\Shv(X')^{\ren},[c']) @>\LT_{c'}^{\true}>>\Gm(\on{Fix}(c'),\om_{\on{Fix}(c')})
\end{CD}
\end{equation*}
\end{Thm}

We are now going to deduce \rt{contracting} from \rt{true2}.

\begin{Emp} \label{E:def}
{\bf Deformation to the normal cone.} We set $R:=k[t]_{(t)}$ and $\cD:=\Spec(R)$.

\smallskip

(a) Let $Z\subseteq X$ be a closed substack. To a pair $(X,Z)$ one associates a morphism
\[
\phi:\wt{X}_Z\to X_{\cD}:=X\times_k{\cD}
\]
such that $\phi_{\eta}$ is an isomorphism, and $\phi_s$ is the projection $N_Z(X)\to Z\subseteq X$, where $N_Z(X)$ is the (classical) normal cone of $X$ to $Z$.
Namely, in the case of schemes this is the usual construction of deformation to the normal cone (see, for example, \cite[Section~1.4]{Va}), and the extension to Artin stacks is immediate, because the assignment
$(X,Z)\mapsto \wt{X}_Z$ commutes with smooth (even flat) pullbacks.  %Moreover, the assignment $(X,Z)\mapsto \wt{X}_Z$ is functorial.

\smallskip

(b) Let $c:C\to X\times X$ be a correspondence such that $Z\subseteq X$ is $c$-invariant. By the functoriality of the construction of part (a), we obtain a correspondence
\[
\wt{c}_Z:\wt{C}_{c_r^{-1}(Z)}\to\wt{X}_Z\times \wt{X}_Z
\]
over $\cD$, whose generic fiber is the base change
\[
c_{\eta}:C_{\eta}\to X_{\eta}\times X_{\eta}
\]
of $c$, and the special fiber is the induced correspondence
\[
N_Z(c):N_{c_r^{-1}(Z)}(C)\to N_Z(X)\times N_Z(X)
\]
between the normal cones.

\smallskip

(c) We now use the following key observation (see \cite[Remark~2.1.2]{Va}): the correspondence $c$ is contracting near $Z$ if and only if
the set-theoretic image of the morphism $$N_Z(c_l): N_{c_r^{-1}(Z)}(C)\to N_Z(X)$$ lies in the zero section $Z\subseteq N_Z(X)$.
\end{Emp}

\begin{Emp} \label{E:spec normal cone}
{\bf Specialization to the normal cone.} %(compare \cite[1.4]{Va}).
Let us be in the situation of Section~\re{def}.% let us introduce the following notation:

\smallskip

(a) Denote by $\on{sp}_Z:\Shv(X)^{\ren}\to \Shv(N_Z(X))^{\ren}$ the functor $\A\mapsto\Psi_{\wt{X}_Z}(\A_{\eta})$;

\smallskip

(b) Let $i$ be the closed embedding $Z\hra X$, and let $N_Z(i)$ denote the embedding $Z\hra N_Z(X)$
(whose image is the zero section of $N_Z(X)$). We have a natural transformation of functors
\[
N_Z(i)^*\circ\on{sp}_Z\to i^*:\Shv(X)^{\ren}\to \Shv(Z)^{\ren}
\]
and a theorem of Verdier (see \cite{Ve} or \cite[Proposition~1.4.2]{Va}) asserts that this transformation is an isomorphism.

\smallskip

(c) Combining the constructions of Sections~\re{bc} and \re{spec}, we obtain a map
\[
\Tr(\on{sp}_Z(c)):\Tr(\Shv(X)^{\ren},[c])\to \Tr(\Shv(X_{\eta})^{\ren},[c_{\eta}])\to \Tr(\Shv(N_Z(X))^{\ren},[N_Z(c)])
\]
and a map
\[
\on{sp}_{\on{Fix}(c)}:\Gm(\on{Fix}(c), \om_{\on{Fix}(c)})\to  \Gm(\on{Fix}(c_{\eta}), \om_{\on{Fix}(c_{\eta})})\to
\Gm(\on{Fix}(N_Z(c)), \om_{\on{Fix}(N_Z(c))}).
\]
\end{Emp}

%\begin{Thm} \label{T:contracting}
%Assume that correspondence $c:C\to X\times X$ is contracting near $Z$, then
%true local term commute with $*$-restriction with respect to the embedding $i:Z\hra X$.
%\end{Thm}
\begin{Emp}
\begin{proof}[Proof of \rt{contracting} assuming \rt{true2}]
The argument will essentially follow that of \cite[Theorem~2.1.3]{Va}.

\smallskip

Using the observation in Section~\re{contracting}(c) and replacing $C$ by an open substack, we can assume that $\on{Fix}(c)_{\red}=\on{Fix}(c|_Z)_{\red}$.
In this case morphism
\[
i_{\Dt}^*:\Gm(\on{Fix}(c), \om_{\on{Fix}(c)})\to \Gm(\on{Fix}(c|_Z), \om_{\on{Fix}(c|_Z)})
\]
is an isomorphism, with inverse $(i_{\Dt})_!$ (see Section~\re{functorfixed}). Therefore it suffices to show that
\begin{equation} \label{Eq:contr1}
\LT^{\true}_c\simeq  (i_{\Dt})_!\circ \LT^{\true}_{c|_Z}\circ \Tr([i_c]^*).
\end{equation}
Hence, by \rt{true}(a), it suffices to show that
\[
\LT^{\true}_{c}\simeq \LT^{\true}_{c}\circ \Tr([i_c]_{!}) \circ \Tr([i_c]^*).
\]
Furthermore, by \rl{additivity}, it suffices to show that
\[
\LT^{\true}_c\circ \Tr([j_c]_!)\simeq 0.
\]

It is shown in the proof of \cite[Theorem~2.1.3]{Va} that the assumptions that $c$ is contracting near $Z$ and that $\on{Fix}(c|_Z)_{\red}=\on{Fix}(c)_{\red}$ imply that
$\on{Fix}(\wt{c}_Z)_{\red}$ is the constant family $\on{Fix}(c)_{\cD,\red}$. Therefore, the  specialization map
\[
\on{sp}_{\on{Fix}(c)}:\Gm(\on{Fix}(c), \om_{\on{Fix}(c)})\to \Gm(\on{Fix}(N_Z(c)), \om_{\on{Fix}(N_Z(c))})
\]
is an isomorphism. Hence, it suffices to show that
\begin{equation} \label{Eq:main equation}
\on{sp}_{\on{Fix}(c)}\circ \LT^{\true}_c\circ \Tr([j_c]_!)\simeq 0.
\end{equation}

Applying \rt{true2}(b) to the morphism $\eta\to\Spec(k)$ and  \rt{true2}(a) to the correspondence $\wt{c}_Z$, we obtain an isomorphism
\[
\on{sp}_{\on{Fix}(c)}\circ \LT^{\true}_c\simeq \LT^{\true}_{N_Z(c)}\circ \Tr(\on{sp}_{Z}(c)).
\]
Therefore it suffices to show that the composition
\[
\Tr(\on{sp}_{Z}(c))\circ \Tr([j_c]_!)\simeq \Tr(\on{sp}_{Z}(c)\circ [j_c]_!):\Tr(\Shv(U)^{\ren},[c|_U])\to\Tr(\Shv(N_Z(X))^{\ren},[N_Z(c)])
\]
is homotopic to zero.

\smallskip

By \rl{trace zero} below, it suffices to show that the composition
\[
(\on{sp}_{Z}\circ j_!)^R\circ [N_Z(c)]:\Shv(N_Z(X))^{\ren}\to \Shv(U)^{\ren}
\]
is homotopic to zero.

\smallskip

Since $c$ is contracting near $Z$, the morphism $N_Z(c_l)_{\red}: N_{c_r^{-1}(Z)}(C)_{\red}\to N_Z(X)$ factors through
$N_Z(i):Z\hra N_Z(X)$ (see Section~\re{def}(c)). Since $[N_Z(c)]=N_Z(c_l)_*\circ N_Z(c_r)^!$, it suffices to show that
\[
(\on{sp}_{Z}\circ j_!)^R\circ N_Z(i)_*\simeq 0,
\]

Passing to left adjoints,
it suffices to show that the composite map $N_Z(i)^*\circ \on{sp}_{Z}\circ j_!$
\[
\Shv(U)^{\ren}\to\Shv(X)^{\ren}\to\Shv(N_Z(X))^{\ren}\to\Shv(Z)^{\ren}
\]
is homotopic to zero. However, by the theorem of Verdier (see Section~\re{spec normal cone}(b)), this composition is homotopic to $i^*\circ j_!\simeq 0$.
\end{proof}
\end{Emp}

\begin{Lem} \label{L:trace zero}
Assume that we are given a lax commutative square

\begin{equation*}
\xy
(0,0)*+{C}="A";
(20,0)*+{C}="B";
(0,-20)*+{D}="C";
(20,-20)*+{D}="D";
{\ar@{->}^{F} "A";"B"};
{\ar@{->}^{t} "A";"C"};
{\ar@{->}^{t} "B";"D"};
{\ar@{->}_{G} "C";"D"};
{\ar@{=>}_\alpha "B";"C"};
\endxy
\end{equation*}
such that $1$-morphism $t$ has a right adjoint $t^R$ and $t^R \circ G\simeq 0$.  Then the induced map between traces
\[
\Tr(\alpha):\Tr(C,F)\to \Tr(D,G)
\]
is homotopic to zero.
\end{Lem}

\begin{proof}
By definition, $\Tr(\alpha)$ is the composition
\[
\Tr(C,F)\to \Tr(C,t^R\circ t\circ F)\overset{\al}{\to}\Tr(C,t^R \circ G \circ t)\simeq \Tr(D,G \circ t\circ t^R)\to \Tr(D,G).
\]
Therefore it is homotopic to zero, because  $t^R \circ G\simeq 0$, thus $\Tr(C,t^R \circ G \circ t)\simeq 0$.
\end{proof}

\section{Proof of Propositions~\ref{P:compatpf} and \ref{P:compatpb}} \label{S:proofs1}
%In this section we will prove Propositions~\ref{P:compatpf} and \ref{P:compatpb}.

%\begin{proof}[Proof of \rp{compatpf}]
%{\bf Proof of \rp{compatpf}.}
\begin{Emp} \label{E:compatpf1}
\begin{proof}[Proof of \rp{compatpf}]. Since trace maps are compatible with compositions, it suffices to show that the horizontal compositions of the lax-commutative diagrams

\begin{equation*} %\label{Eq:ren}
\xy
(10,10)*+{\Shv(X)}="A";
(40,10)*+{\Shv(X)^{\ren}}="B";
(10,-10)*+{\Shv(X)}="C";
(40,-10)*+{\Shv(X)^{\ren}}="D";
(70,10)*+{\Shv(Y)^{\ren}}="E";
(70,-10)*+{\Shv(Y)^{\ren}}="F";
{\ar@{->}^{\ren_X} "A";"B"};
{\ar@{->}^{f_!} "B";"E"};
{\ar@{->}_{\ren_X} "C";"D"};
{\ar@{->}_{f_!} "D";"F"};
{\ar@{->}_{[c]} "A";"C"};
{\ar@{->}_{[c]} "B";"D"};
{\ar@{->}^{[d]} "E";"F"};
{\ar@{=>}_{\al} "C";"B"};
{\ar@{=>}_{[f]_!} "D";"E"}
\endxy
\end{equation*}
and
\begin{equation*} %\label{Eq:ren}
\xy
(10,10)*+{\Shv(X)}="A";
(40,10)*+{\Shv(Y)}="B";
(10,-10)*+{\Shv(X)}="C";
(40,-10)*+{\Shv(Y)}="D";
(70,10)*+{\Shv(Y)^{\ren}}="E";
(70,-10)*+{\Shv(Y)^{\ren}}="F";
{\ar@{->}^{f_!} "A";"B"};
{\ar@{->}^{\ren_Y} "B";"E"};
{\ar@{->}_{f_!} "C";"D"};
{\ar@{->}_{\ren_Y} "D";"F"};
{\ar@{->}_{[c]} "A";"C"};
{\ar@{->}_{[d]} "B";"D"};
{\ar@{->}^{[d]} "E";"F"};
{\ar@{=>}_{[f]_!} "C";"B"};
{\ar@{=>}_{\al} "D";"E"}
\endxy
\end{equation*}
are canonically isomorphic. In other words, we have to show that there is a canonical isomorphism
\begin{equation} \label{Eq:isompf}
f_!\circ\ren_X\simeq \ren_Y\circ f_!
\end{equation}
of functors $\Shv(X)\to \Shv(Y)^{\ren}$ making the following diagram homotopy commutative:
\begin{equation} \label{Eq:1}
\begin{CD}
f_!\circ \ren_X\circ [c] @>\al>> f_!\circ [c]\circ \ren_X @>[f]_!>> [d]\circ f_!\circ \ren_X\\
@V\form{isompf}V\sim V @. @V\sim V\form{isompf}V\\
\ren_Y\circ f_! \circ [c] @>[f]_!>> \ren_Y\circ [d]\circ f_!  @>\al>> [d]\circ \ren_Y\circ f_!.
\end{CD}
\end{equation}

To get an isomorphism \form{isompf}, we notice that for every $\A\in\Shv(X)^c$, both $(f_!\circ\ren_X)(\A)$ and
$(\ren_Y\circ f_!)(\A)$ are simply $f_!(\A)\in\Shv(Y)^c\subseteq\Shv(Y)^{\constr}\subseteq\Shv(Y)^{\ren}$.
Alternatively, it can be obtained from the isomorphism $\unren_X\circ f^!\isom f^!\circ\unren_Y$ from \rl{functren}(a) by passing to left  adjoints.

\medskip

Next, unwinding the definitions, diagram \form{1} decomposes as
\begin{equation} \label{Eq:2}
\begin{CD}
f_!\circ \ren_X\circ (c_l)_{\tri}\circ c_r^!  @>\ref{L:functren}(d)>> f_!\circ (c_l)_*\circ\ren_C\circ c_r^! @>\ref{L:functren}(d)>> f_!\circ(c_l)_*\circ c_r^!\circ\ren_X\\
@V\form{isompf} V\sim V  @V\form{BC1}_*VV @VV\form{BC1}_*V \\
\ren_Y\circ f_!\circ (c_l)_{\tri}\circ c_r^!  @. (d_l)_{*}\circ g_!\circ\ren_C\circ c_r^!
 @>\ref{L:functren}(d)>> (d_l)_{*}\circ g_!\circ c_r^!\circ \ren_X\\
@V\form{BC1}VV  @V\form{isompf}V\sim V @VV\text{base change}V \\
\ren_Y\circ (d_l)_{\tri} \circ g_!\circ c_r^! @>\ref{L:functren}(d)>> (d_l)_{*}\circ\ren_D\circ g_!\circ c_r^!
 @. (d_l)_{*}\circ d_r^!\circ f_!\circ \ren_X\\
@V\text{base change}VV  @V\text{base change}VV @V\sim V\form{isompf} V \\
\ren_Y\circ (d_l)_{\tri}\circ d_r^!\circ f_! @>\ref{L:functren}(d)>> (d_l)_{*}\circ\ren_D\circ d_r^!\circ f_! @>\ref{L:functren}(d)>> (d_l)_{*}\circ d_r^!\circ \ren_Y\circ f_!,
\end{CD}
\end{equation}
where by $\form{BC1}_*$ we denote the (easier) version of morphism $\form{BC1}$, where ${-}_{\tri}$ is replaced by ${-}_*$.

\medskip

It remains to show that all inner diagrams of diagram \form{2} are homotopy commutative. This is clear for the top right and the bottom left inner diagrams. For the remaining ones, it remains to show that the following two diagrams are homotopy commutative:
\begin{equation} \label{Eq:3}
\begin{CD}
\ren_Y\circ f_!\circ (c_l)_{\tri}  @>\form{isompf}>\sim> f_!\circ \ren_X\circ (c_l)_{\tri}  @>\ref{L:functren}(d)>> f_!\circ (c_l)_*\circ\ren_C \\
@V\form{BC1}VV @. @V\form{BC1}_*VV \\
\ren_Y\circ (d_l)_{\tri} \circ g_! @>\ref{L:functren}(d)>> (d_l)_{*}\circ\ren_D\circ g_! @>\form{isompf}>\sim>
(d_l)_{*}\circ g_!\circ\ren_C
\end{CD}
\end{equation}
\begin{equation} \label{Eq:4}
\begin{CD}
 \ren_D\circ g_!\circ c_r^! @>\form{isompf}>\sim>  g_!\circ\ren_C\circ c_r^! @>\ref{L:functren}(d)>>  g_!\circ c_r^!\circ \ren_X\\
  @V\text{base change}VV @.  @VV\text{base change}V \\
\ren_D\circ d_r^!\circ f_! @>\ref{L:functren}(d)>>  d_r^!\circ \ren_Y\circ f_!  @ >\form{isompf}>\sim>  d_r^!\circ f_!\circ \ren_X.
\end{CD}
\end{equation}

\medskip
%\end{Emp}

%\begin{Emp}
By adjunction, the homotopy commutativity of diagram \form{4} follows from the fact that isomorphisms from \rl{functren}(a) are compatible with compositions.

\smallskip

Next, decomposing the left inner square of diagram \form{morph} as
\begin{equation} \label{Eq:morphleft}
\begin{CD}
X @<\wt{d}_l<< C_l:=X\times_Y D @<p<< C\\
@VfVV        @V\wt{f}VV  @VVgV\\
Y @<d_l<< D @= D
\end{CD}
\end{equation}
and unwinding the definition, diagram \form{3} decomposes as
\begin{equation} \label{Eq:3'}
\begin{CD}
\ren_Y\circ f_!\circ (c_l)_{\tri}  @>\form{isompf}>\sim> f_!\circ \ren_X\circ (c_l)_{\tri}  @>\ref{L:functren}(d)>> f_!\circ (c_l)_*\circ\ren_C \\
@A\sim A\re{ren}(c)A @A\sim A\re{ren}(c)A @A\sim AA \\
\ren_Y\circ f_!\circ (\wt{d}_l)_{\tri}\circ p_{\tri}  @>\form{isompf}>\sim> f_!\circ \ren_X\circ (\wt{d}_l)_{\tri}\circ p_{\tri}  @>\ref{L:functren}(d)\circ\ref{L:functren}(d)>> f_!\circ (\wt{d}_l)_*\circ p_*\circ\ren_C \\
@VV\text{base change}V @. @V\text{base change}VV \\
\ren_Y\circ ({d}_l)_{\tri}\circ \wt{f}_!\circ p_{\tri}  @>\ref{L:functren}(d)>> ({d}_l)_{*}\circ\ren_D\circ \wt{f}_!\circ p_{\tri}
  @>\ref{L:functren}(d)\circ \form{isompf}>> (d_l)_*\circ \wt{f}_!\circ p_*\circ\ren_C \\
@V\sim V\re{safe}(d)\circ\ref{C:safe}V @V\sim V\re{safe}(d)\circ\ref{C:safe}V @V\ref{C:safe}V\sim V \\
\ren_Y\circ (d_l)_{\tri} \circ g_! @>\ref{L:functren}(d)>> (d_l)_{*}\circ\ren_D\circ g_! @>\form{isompf}>\sim>
(d_l)_{*}\circ g_!\circ\ren_C.
\end{CD}
\end{equation}

Thus, it remains to show that all inner squares of diagram \form{3'} are homotopy commutative.  This is clear for the top left and the bottom left inner squares.

\smallskip

The assertion for the top right inner square follows from the fact that isomorphisms of \rl{functren}(c) are compatible with compositions, while for the assertion for the bottom right square we also observe that the diagram
\begin{equation*}
\begin{CD}
\ren_{C_l}\circ p_!  @>\form{isompf}>\sim> p_!\circ \ren_C \\
@V\re{safe}(d)\circ\ref{C:safe}V\sim V  @V\ref{C:safe}V\sim V \\
\ren_{C_l}\circ p_{\tri} @>\ref{L:functren}(d)>> p_{*}\circ\ren_C
\end{CD}
\end{equation*}
is homotopy commutative.

\smallskip

Finally, to see the commutativity of the middle inner square of diagram \form{3'}, we have to show the commutativity of the diagram

\begin{equation} \label{Eq:5}
\begin{CD}
\ren_X\circ (\wt{d}_l)_{\tri}\circ \wt{f}^! @>\ref{L:functren}(d)>> (\wt{d}_l)_{*}\circ\ren_{C_l}\circ\wt{f}^!  @>\ref{L:functren}(d)>>
(\wt{d}_l)_{*}\circ \wt{f}^!\circ \ren_{D}\\
@V\re{ren}(d) VV @. @V\sim V\text{base change}V\\
\ren_X\circ f^!\circ (d_l)_{\tri}  @>\ref{L:functren}(d)>> f^!\circ \ren_Y\circ (d_l)_{\tri} @>\ref{L:functren}(d)>>
f^!\circ (d_l)_{*}\circ \ren_{D},
\end{CD}
\end{equation}
which follows from \rco{bcren}(b).
\end{proof}
\end{Emp}

%\end{proof}

\begin{Emp}
\begin{proof}[Proof of \rp{compatpb}]
Arguing as in the proof of \rp{compatpf} but interchanging $X$ with $Y$, $C$ with $D$, $c$ with $d$, $f_!$ with $f^*$, $g_!$ with $g^*$, $[f]_!$ with $[f]^*$ and isomorphism \form{isompf} with \form{isompb} in all places,
%
%Note first that we have a canonical isomorphism
%\begin{equation} \label{Eq:isompb}
%f^*\circ\ren_Y\simeq\ren_X\circ f^*.
%\end{equation}
%
%Indeed, for every $\A\in\Shv(Y)^c$, we have $f^*(\A)\in \Shv(X)^c$, since $f$ is safe, and hence both $f^*\circ\ren_Y(\A)$ and
%$\ren_X\circ f^*(\A)$ are simply $f^*(\A)\in\Shv(X)^c\subseteq\Shv(X)^{\constr}\subseteq\Shv(X)^{\ren}$.
we end up showing the homotopy commutativity of diagrams

\begin{equation} \label{Eq:6}
\begin{CD}
\ren_X\circ f^*\circ (d_l)_{\tri}  @>\form{isompb}>\sim> f^*\circ \ren_Y\circ (d_l)_{\tri}  @>\ref{L:functren}(d)>> f^*\circ (d_l)_*\circ\ren_D \\
@V\form{BC2}VV @. @VV\text{base change}V \\
\ren_X\circ (c_l)_{\tri} \circ g^* @>\ref{L:functren}(d)>> (c_l)_{*}\circ\ren_C\circ g^* @>\form{isompb}>\sim>
(c_l)_{*}\circ g^*\circ\ren_D
\end{CD}
\end{equation}
and
\begin{equation} \label{Eq:7}
\begin{CD}
 \ren_C\circ g^*\circ d_r^! @>\form{isompb}>\sim>  g^*\circ\ren_D\circ d_r^! @>\ref{L:functren}(d)>>  g^*\circ d_r^!\circ \ren_Y\\
  @V\text{base change}VV @.  @VV\text{base change}V \\
\ren_C\circ c_r^!\circ f^* @>\ref{L:functren}(d)>>  c_r^!\circ \ren_X\circ f^*  @ >\form{isompb}>\sim>  c_r^!\circ f^*\circ \ren_Y,
\end{CD}
\end{equation}
which replace diagrams \form{3} and \form{4} in our case.

\smallskip

%However, the commutativity of both diagrams \form{6} and \form{7} easily follow from the observation that the isomorphism
%\[
%\ren_X\circ f^*\isom f^*\circ\ren_Y
%\]
%of \form{isompb} corresponds by adjointness to the isomorphism
%$$ f^*\circ\unren_Y\isom\unren_X\circ f^*$$ of \rl{functren}(a).

The homotopy commutativity of diagram \form{6} follows from the fact that morphism \form{BC2} is induced by the base  change morphism $f^*\circ (d_l)_*\to (c_l)_* \circ g^*$ of functors $\Shv(D)^{\ren}\to \Shv(X)^{\ren}$.

\smallskip

Next, by adjointness, for the homotopy commutativity of diagram \form{7} it suffices to show the homotopy commutativity of a diagram
\begin{equation} \label{Eq:8}
\begin{CD}
 g^*\circ d_r^!\circ\unren_Y @>\ref{L:functren}(a)>\sim>  g^*\circ\unren_D\circ d_r^! @>\ref{L:functren}(a)>\sim>  \unren_C\circ g^*\circ d_r^!\\
  @V\form{bc}VV @.  @VV\form{bc}V \\
c_r^!\circ f^*\circ\unren_Y @>\ref{L:functren}(a)>\sim>  c_r^!\circ \unren_X\circ f^*  @ >\ref{L:functren}(a)>\sim>  \unren_C\circ c_r^!\circ f^*.
\end{CD}
\end{equation}

Finally, the homotopy commutativity of diagram \form{8} follows from the fact that both the Gysin map and the base change isomorphism
in $\Shv(-)^{\ren}$ are induced by the Gysin map and the base change isomorphism in $\Shv(-)$ (see the proof of \rco{bcren}(a)).
 \end{proof}
\end{Emp}

\section{Completion of proofs, I} \label{S:proofs}
In this section we will prove Theorems~\ref{T:true}(a), \ref{T:reftrue}(a), \ref{T:true2}(a),(b) and \rp{refined}.  Since arguments are similar in all cases, we will first describe the proof of \rt{reftrue}(a) in detail, and then indicate what changes are needed to treat the other cases.

\vskip 5truept
\noindent{\bf Proof of \rt{reftrue}(a).}

\begin{Emp} \label{E:pftruea1}
(a) Recall that since $f:X\to Y$ is proper and safe, the functor $f_!:\Shv(X)\to \Shv(Y)$ is {\em self-dual}.
Namely, under identifications
\[
\Shv(X)\simeq\Shv(X)^{\vee}\text{ and }\Shv(Y)\simeq\Shv(Y)^{\vee}
\]
from Section~\re{reftrue}(a), the dual functor
$(f_!^R)^{\vee}:\Shv(X)^{\vee}\to \Shv(Y)^{\vee}$ is naturally identified with $f_!:\Shv(X)\to \Shv(Y)$. Indeed, this is equivalent to the assertion $(\B{D}_Y\circ f_!\circ \B{D}_X)(\A)\simeq f_!(\A)$ for all $\A\in\Shv(X)^c$, and hence follows \rco{safe}.

\smallskip

(b) Using part~(a), one checks that
the map
\[
\Tr([f]_!):\Tr(\Shv(X),[c])\to\Tr(\Shv(Y),[d])
\]
is induced by the lax commutative square
\begin{equation} \label{Eq:tracepf}
\xy
(0,10)*+{\Vect}="A";
(40,10)*+{\Shv(X)\otimes \Shv(X)}="B";
(80,10)*+{\Shv(X)\otimes \Shv(X)}="C";
(120,10)*+{\Vect}="D";
(0,-10)*+{\Vect}="A'";
(40,-10)*+{\Shv(Y)\otimes \Shv(Y)}="B'";
(80,-10)*+{\Shv(Y)\otimes \Shv(Y)}="C'";
(120,-10)*+{\Vect,}="D'";
{\ar@{->}^{u_X} "A";"B"};
{\ar@{->}_{u_Y} "A'";"B'"};
{\ar@{->}^{\ev_X} "C";"D"};
{\ar@{->}_{\ev_Y} "C'";"D'"};
{\ar@{->}^{[c]\otimes\Id} "B";"C"};
{\ar@{->}_{[d]\otimes\Id} "B'";"C'"};
{\ar@{->}^{\on{id}} "A";"A'"};
{\ar@{->}^{\on{id}} "D";"D'"};
{\ar@{->}^{f_!\otimes f_!} "B";"B'"};
{\ar@{->}^{f_!\otimes f_!} "C";"C'"};
{\ar@{=>}_{\al_l} "B";"A'"};
{\ar@{=>}_{\al_m} "C";"B'"};
{\ar@{=>}_{\al_r} "D";"C'"};
\endxy
\end{equation}
where

\smallskip

$\bullet$ $\al_l$ corresponds to the morphism
\[
(f_!\otimes f_!)(u_X)\to u_{Y}\simeq\boxtimes^R((\Dt_Y)_*(\om_Y)),
\]
obtained by adjunction from the composition
\[
\boxtimes((f_!\otimes f_!)(u_X))\simeq (f\times f)_!(\boxtimes(u_X))\to (f\times f)_!((\Dt_X)_*(\om_X))\simeq(\Dt_Y)_*(f_!(\om_X))\to (\Dt_Y)_*(\om_Y),
\]
induced by natural maps $\boxtimes(u_X)\to (\Dt_X)_*(\om_X)$ and $f_!(\om_X)\to\om_Y$;

\smallskip

$\bullet$ $\al_m$ corresponds to the morphism $[f]_!:f_!\circ [c]\to [d]\circ f_!$ from Section \re{pf}(a);

\smallskip

$\bullet$ $\al_r$ is the canonical morphism

\[
\Gm_{\tri}(X,-\overset{!}{\otimes}-)\simeq \Gm_{\tri}(Y,f_!(-\overset{!}{\otimes}-))\to \Gm_{\tri}(Y,f_!(-)\overset{!}{\otimes}f_!(-)),
\]
induced by the canonical morphism $f_!(-\overset{!}{\otimes}-)\to f_!(-)\overset{!}{\otimes}f_!(-)$.
\end{Emp}

\begin{Emp} \label{E:pftruea2}
Consider the lax commutative square
\begin{equation} \label{Eq:tracepf2}
\xy
(0,10)*+{\Vect}="A";
(40,10)*+{\Shv(X\times X)}="B";
(80,10)*+{\Shv(X\times X)}="C";
(120,10)*+{\Vect}="D";
(0,-10)*+{\Vect}="A'";
(40,-10)*+{\Shv(Y\times Y)}="B'";
(80,-10)*+{\Shv(Y\times Y)}="C'";
(120,-10)*+{\Vect,}="D'";
{\ar@{->}^{(\Dt_X)_*(\om_X)} "A";"B"};
{\ar@{->}_{(\Dt_Y)_*(\om_Y)} "A'";"B'"};
{\ar@{->}^{\Gm_{\tri}\circ\Dt_X^!} "C";"D"};
{\ar@{->}_{\Gm_{\tri}\circ\Dt_Y^!} "C'";"D'"};
{\ar@{->}^{[c\times\Id]} "B";"C"};
{\ar@{->}_{[d\otimes\Id]} "B'";"C'"};
{\ar@{->}^{\on{id}} "A";"A'"};
{\ar@{->}^{\on{id}} "D";"D'"};
{\ar@{->}^{(f\times f)_!} "B";"B'"};
{\ar@{->}^{(f\times f)_!} "C";"C'"};
{\ar@{=>}_{\al_l} "B";"A'"};
{\ar@{=>}_{\al_m} "C";"B'"};
{\ar@{=>}_{\al_r} "D";"C'"};
\endxy
\end{equation}
where

\smallskip

$\bullet$ $\al_l$ corresponds to the morphism
\[
(f\times f)_!((\Dt_X)_*(\om_X))\simeq(\Dt_Y)_*(f_!(\om_X))\to (\Dt_Y)_*(\om_Y)
\]
(similar to the corresponding map in diagram \form{tracepf});

\smallskip

$\bullet$ $\al_m$ is the morphism $(f\times f)_!\circ [c\times\Id]\to [d\times \Id]\circ (f\times f)_!$, given by the morphism
of correspondences $(f\times f,g\times f):c\times\Id_X\to d\times \Id_Y$ as in Section \re{pf}(a);

\smallskip

$\bullet$ $\al_r$ is the canonical morphism
\[
\Gm_{\tri}(X,\Dt_X^!(-))\simeq \Gm_{\tri}(Y,(f_!\circ \Dt_X^!)(-))\to\Gm_{\tri}(Y,(\Dt_Y^!\circ (f\times f)_!)(-)),
\]
induced by the base change morphism $f_!\circ \Dt_X^!\to\Dt_Y^!\circ (f\times f)_!$.
\end{Emp}

\begin{Emp} \label{E:pftruea3}
We can decompose the diagram of \rt{reftrue}(a) as
\begin{equation} \label{Eq:rttruea}
\begin{CD}
\Tr(\Shv(X),[c])@>\form{reftrue}_c>> \Gm_{\tri}(X, (\Dt_X^!\circ [c\times\Id]\circ (\Dt_X)_*)(\om_X))@>{\text{base change}}>> \Gm_{\tri}(\on{Fix}(c), \om_{\on{Fix}(c)})\\
@VV\Tr([f]_!)V                      @V\form{tracepf2}VV                      @VV(g_{\Dt})_!V\\
\Tr(\Shv(Y),[d])@>>\form{reftrue}_d> \Gm_{\tri}(Y, (\Dt_Y^!\circ [d\times\Id]\circ (\Dt_Y)_*)(\om_Y))@>>{\text{base change}}> \Gm_{\tri}(\on{Fix}(d), \om_{\on{Fix}(d)}),
\end{CD}
\end{equation}
where the top and the bottom arrows are the compositions from Section \re{reftrue}(c) for $c$ and $d$, respectively, while the middle vertical arrow is induced by the lax commutative square \form{tracepf2}. Therefore, it suffices to show that both inner squares of diagram \form{rttruea} are canonically homotopy commutative.
\end{Emp}

\begin{Emp} \label{E:pftruea4}
In order to verify the commutativity of the left inner square of diagram \form{rttruea}, it suffices to show that the two maps
$$\Tr(\Shv(X),[c])\rightrightarrows \Gm_{\tri}(Y, (\Dt_Y^!\circ [d\times\Id]\circ(\Dt_Y)_*)(\om_Y))$$
that arise from the following two lax commutative diagrams are canonically homotopic.

%Note that all the arrows in the left inner square of \form{rttruea} are horizontal compositions in the corresponding lax commutative squares.
%In particular, in order to show that  the left inner square of \form{rttruea} is commutative, it suffices to show that the vertical compositions
%$$\form{tracepf2}\circ \form{true}_X \text{ and } \form{true}_Y\circ \form{tracepf}$$ of the
%lax commutative squares are homotopic. Explicitly, this means that the
%vertical compositions of

\begin{equation} \label{Eq:composition1}
\xy
(0,20)*+{\Vect}="A";
(40,20)*+{\Shv(X)\otimes \Shv(X)}="B";
(80,20)*+{\Shv(X)\otimes \Shv(X)}="C";
(120,20)*+{\Vect}="D";
(0,0)*+{\Vect}="A'";
(40,0)*+{\Shv(X\times X)}="B'";
(80,0)*+{\Shv(X\times X)}="C'";
(120,0)*+{\Vect}="D'";
{\ar@{->}^{u_X} "A";"B"};
{\ar@{->}^{(\Dt_X)_*(\om_X)} "A'";"B'"};
{\ar@{->}^{\ev_X} "C";"D"};
{\ar@{->}^{\Gm_{\tri}\circ\Dt_X^!} "C'";"D'"};
{\ar@{->}^{[c]\otimes\Id} "B";"C"};
{\ar@{->}^{[c\times\Id]} "B'";"C'"};
{\ar@{->}^{\on{id}} "A";"A'"};
{\ar@{->}^{\on{id}} "D";"D'"};
{\ar@{->}^{\boxtimes} "B";"B'"};
{\ar@{->}^{\boxtimes} "C";"C'"};
{\ar@{=>} "B";"A'"};
{\ar@{=>} "C";"B'"};
{\ar@{=>} "D";"C'"};
(0,-20)*+{\Vect}="A''";
(40,-20)*+{\Shv(Y\times Y)}="B''";
(80,-20)*+{\Shv(Y\times Y)}="C''";
(120,-20)*+{\Vect,}="D''";
{\ar@{->}_{(\Dt_Y)_*(\om_Y)} "A''";"B''"};
{\ar@{->}_{\Gm_{\tri}\circ\Dt_Y^!} "C''";"D''"};
{\ar@{->}_{[d\otimes\Id]} "B''";"C''"};
{\ar@{->}^{\on{id}} "A'";"A''"};
{\ar@{->}^{\on{id}} "D'";"D''"};
{\ar@{->}^{(f\times f)_!} "B'";"B''"};
{\ar@{->}^{(f\times f)_!} "C'";"C''"};
{\ar@{=>} "B'";"A''"};
{\ar@{=>} "C'";"B''"};
{\ar@{=>} "D'";"C''"};
\endxy
\end{equation}

and

\begin{equation} \label{Eq:composition2}
\xy
(0,20)*+{\Vect}="A";
(40,20)*+{\Shv(X)\otimes \Shv(X)}="B";
(80,20)*+{\Shv(X)\otimes \Shv(X)}="C";
(120,20)*+{\Vect}="D";
(0,0)*+{\Vect}="A'";
(40,0)*+{\Shv(Y)\otimes \Shv(Y)}="B'";
(80,0)*+{\Shv(Y)\otimes \Shv(Y)}="C'";
(120,0)*+{\Vect,}="D'";
{\ar@{->}^{u_X} "A";"B"};
{\ar@{->}^{u_Y} "A'";"B'"};
{\ar@{->}^{\ev_X} "C";"D"};
{\ar@{->}^{\ev_Y} "C'";"D'"};
{\ar@{->}^{[c]\otimes\Id} "B";"C"};
{\ar@{->}^{[d]\otimes\Id} "B'";"C'"};
{\ar@{->}^{\on{id}} "A";"A'"};
{\ar@{->}^{\on{id}} "D";"D'"};
{\ar@{->}^{f_!\otimes f_!} "B";"B'"};
{\ar@{->}^{f_!\otimes f_!} "C";"C'"};
{\ar@{=>} "B";"A'"};
{\ar@{=>} "C";"B'"};
{\ar@{=>} "D";"C'"};
(0,-20)*+{\Vect}="A''";
(40,-20)*+{\Shv(Y\times Y)}="B''";
(80,-20)*+{\Shv(Y\times Y)}="C''";
(120,-20)*+{\Vect}="D''";
{\ar@{->}_{(\Dt_Y)_*(\om_Y)} "A''";"B''"};
{\ar@{->}_{\Gm_{\tri}\circ\Dt_Y^!} "C''";"D''"};
{\ar@{->}_{[d\times\Id]} "B''";"C''"};
{\ar@{->}^{\on{id}} "A'";"A''"};
{\ar@{->}^{\on{id}} "D'";"D''"};
{\ar@{->}^{\boxtimes} "B'";"B''"};
{\ar@{->}^{\boxtimes} "C'";"C''"};
{\ar@{=>} "B'";"A''"};
{\ar@{=>} "C'";"B''"};
{\ar@{=>} "D'";"C''"};
\endxy
\end{equation}
For this it suffices to establish that the vertical composition of the left (resp. middle, resp. right) inner square of diagram
\form{composition1} is canonically homotopic to the vertical composition of the corresponding inner square of diagram \form{composition2}. But all there three assertions are straightforward.
\end{Emp}

\begin{Emp} \label{E:pftruea5}
It remains to show the commutativity of the right inner square of diagram \form{rttruea}. This follows from the compatibility
of base change morphisms with compositions (see \cite[proof of Proposition~1.2.5]{Va}).

%More conceptually, the assertion follows from the fact that the assignment $X\mapsto\Shv(X),c\mapsto [c]$ can be upgraded to a functor
%of $2$-categories
%\[
%\on{Corr}_k\to\DGCat,
%\]
%where $\on{Corr}_k$ is the $2$-category, whose objects are Artin stacks over $k$, morphisms are correspondences and $2$-morphisms are proper morphisms between correspondences, $\DGCat$ is the $2$-category of $\qlbar$-linear compactly generated stable $\infty$-categories.
\end{Emp}

\begin{Emp} \label{E:pftruea}
{\bf Proof of \rt{true}(a).} The proof is almost identical to that of \rt{reftrue}(a), except we have to replace $\Shv(-)$, $\Shv(-)^c$, $u_{-}$, $\om_{-}$ and $\Gm_{\tri}$ by $\Shv(-)^{\ren}$, $\Shv(-)^{\constr}$, $u^{\ren}_{-}$, $\om^{\ren}_{-}$ and $\Gm$, respectively, in all places:

\smallskip

(i) Arguing as in Section~\re{pftruea1}, we see that
the functor $f_!:\Shv(X)^{\ren}\to\Shv(Y)^{\ren}$ is self-dual. Therefore the trace map
\[
\Tr([f]_!):\Tr(\Shv(X)^{\ren},[c])\to\Tr(\Shv(Y)^{\ren},[c])
\]
is induced by a lax commutative square, obtained from diagram~\form{tracepf} by the replacement mentioned above.

\smallskip

(ii) Next, the diagram of \rt{true}(a) decomposes in a similar manner as  diagram~\form{rttruea}, where the top and the bottom arrows are the compositions from Section~\re{truelocal}(b) for $c$ and $d$, respectively, while the middle vertical arrow is induced by a lax commutative square, obtained from  diagram~\form{tracepf2} by the replacement mentioned above. It  remains to show that both inner squares are homotopy commutative.

\smallskip

(iii) Again, to show commutativity of the left inner square, it suffices to show that vertical compositions
of the corresponding lax commutative squares are homotopic, which is a routine verification, while the commutativity of the right inner square follows from the fact that base change morphisms, are compatible with compositions.
\end{Emp}

\begin{Emp}
{\bf Proof of \rt{true2}.} We will only discuss the proof of part~(a), because the proof of part~(b) is similar but easier.

\smallskip

(i) Since functor $\Psi_{\wt{X}}$ commutes with the Verdier duality on constructible objects, it is self-dual.
Therefore the trace map
\[
\Tr(\Psi_{\wt{c}}):\Tr(\Shv(\wt{X}_{\eta})^{\ren},[\wt{c}_{\eta}])\to\Tr(\Shv(\wt{X}_{s})^{\ren},[\wt{c}_{s}])
\]
is induced by a lax commutative square, similar to diagram~\form{tracepf}, in which, in addition to replacements of Section~\re{pftruea}, $X,c$, $Y,d$ and $f_!\otimes f_!$ are replaced by
$\wt{X}_{\eta},\wt{c}_{\eta}$, $\wt{X}_s,\wt{c}_s$ and $\Psi_{\wt{X}}\otimes \Psi_{\wt{X}}$, respectively (and $2$-morphisms are modified appropriately).

\smallskip

(ii) Next, the diagram of \rt{true2}(a) decomposes in a similar manner as  diagram~\form{rttruea}, where the middle vertical
arrow is induced by a lax commutative square, similar to diagram~\form{tracepf2}, in which $(f\times f)_!$,
 are replaced by $\Psi_{\wt{X}\times\wt{X}}$. It  remains to show that both inner squares are homotopy commutative.

\smallskip

(iii) Again, to show commutativity of the left inner square, it suffices to show that vertical compositions
of the corresponding lax commutative squares are homotopic. In this case, it is a routine verification, which uses the
fact that nearby cycles commute with exterior products. Finally, the commutativity of the right inner square follows from the fact that
base change morphisms, corresponding to Cartesian squares, commute with compositions (see \cite[proof of Proposition~1.3.5]{Va}).
\end{Emp}

\begin{Emp}
{\bf Proof of \rp{refined}.}

\smallskip

(i) Since $(\B{D}_X\circ\ren_X\circ \B{D}_X)(\A)\simeq\ren_X(\A)$ for every $\A\in\Shv(X)^c$, we conclude that the functor
 $\ren_X:\Shv(X)\to\Shv(Y)^{\ren}$ is self-dual. Therefore the trace map
\[
\Tr(\ren_X,[c]):\Tr(\Shv(X),[c])\to\Tr(\Shv(X)^{\ren},[c])
\]
is induced by a lax commutative square, similar to diagram~\form{tracepf}, in which $\Shv(Y)$, $f_!$, $u_{Y}$, $[d]$ and $\ev_Y$ are replaced by
$\Shv(X)^{\ren}$, $\ren_X$, $u^{\ren}_{X}$, $[c]$ and $\ev_X$, respectively.

\smallskip

(ii) Next, diagram \form{refined} of \rp{refined} decomposes in a similar manner as  diagram~\form{rttruea}, where the middle vertical
arrow is induced by a lax commutative square, similar to diagram~\form{tracepf2}, in which $\Shv(Y\times Y)$, $(f\times f)_!$,
$(\Dt_Y)_*(\om_Y)$ and $\Gm_{\tri}\circ\Dt_Y^!$ are replaced by $\Shv(X\times X)^{\ren}$, $\ren_{X\times X}$,
$(\Dt_X)_*(\om^{\ren}_X)$ and $\Gm\circ\Dt_X^!$, respectively. It  remains to show that both inner squares are homotopy commutative.

\smallskip

(iii) Again, to show commutativity of the left inner square, it suffices to show that vertical compositions
of the corresponding lax commutative squares are homotopic, which is a routine verification, which uses the
fact renormalization functors commute with exterior products. Finally, the commutativity of the right inner square follows from
\rco{bcren}(b) asserting that renormalization functors commute with base change morphisms.
\end{Emp}

\section{Completion of proofs, II} \label{S:proofs2}
In this section we will prove Theorems~\ref{T:true}(b) and \ref{T:reftrue}(b). Moreover, we will restrict ourselves with the proof of
\rt{reftrue}(b), because the proof of \rt{true}(b) can be obtained from it using the same modification as in Section~\re{pftruea}.
Although the overall strategy is similar to that of \rt{reftrue}(a), some ingredients are new.

\begin{Emp}
{\bf The $(*,!)$-pullback.} Consider the functor
\[
f^*\boxtimes f^!:\Shv(Y\times Y)\to \Shv(X\times X),
\]
defined as a composition of $(f\times\Id)^*$ and $(\Id\times f)^!$ in either order. Namely, there is a canonical base change morphism  of functors
\[
(f\times\Id_X)^*\circ(\Id_Y\times f)^!\to (\Id_X\times f)^!\circ(f\times\Id_Y)^*,
\]
which is an isomorphism because $f$ is smooth.
\end{Emp}

\begin{Emp} \label{E:pftrueb2}
Consider a lax commutative square
\begin{equation} \label{Eq:tracepb2}
\xy
(0,10)*+{\Vect}="A";
(40,10)*+{\Shv(Y\times Y)}="B";
(80,10)*+{\Shv(Y\times Y)}="C";
(120,10)*+{\Vect}="D";
(0,-10)*+{\Vect}="A'";
(40,-10)*+{\Shv(X\times X)}="B'";
(80,-10)*+{\Shv(X\times X)}="C'";
(120,-10)*+{\Vect,}="D'";
{\ar@{->}^{(\Dt_Y)_*(\om_Y)} "A";"B"};
{\ar@{->}_{(\Dt_X)_*(\om_X)} "A'";"B'"};
{\ar@{->}^{\Gm_{\tri}\circ\Dt_Y^!} "C";"D"};
{\ar@{->}_{\Gm_{\tri}\circ\Dt_X^!} "C'";"D'"};
{\ar@{->}^{[d\times\Id]} "B";"C"};
{\ar@{->}_{[c\times\Id]} "B'";"C'"};
{\ar@{->}^{\on{id}} "A";"A'"};
{\ar@{->}^{\on{id}} "D";"D'"};
{\ar@{->}^{f^*\pp f^!} "B";"B'"};
{\ar@{->}^{f^*\pp f^!} "C";"C'"};
{\ar@{=>}_{\al_l} "B";"A'"};
{\ar@{=>}_{\al_m} "C";"B'"};
{\ar@{=>}_{\al_r} "D";"C'"};
\endxy
\end{equation}
where

\smallskip

$\bullet$ $\al_l$ corresponds to the morphism
\begin{equation} \label{Eq:map 1}
(f\times\Id_X)^*(\Id_Y\times  f)^!((\Dt_Y)_*(\om_Y))\simeq
(f\times\Id_X)^*(f\times\Id_X)_*(\Dt_X)_*(f^!(\om_Y)))\overset{\on{counit}}{\lra} (\Dt_X)_*(\om_X),
\end{equation}
obtained from the base change isomorphism $(\Id_Y\times  f)^!\circ(\Dt_Y)_*\simeq (f\times\Id_X)_*\circ(\Dt_X)_*\circ f^!$, corresponding to the Cartesian diagram
\[
\begin{CD}
X @>(f,\Id_X)>> Y\times X\\
@VfVV @VV\Id_Y\times fV\\
Y @>\Dt_Y>> Y\times Y;
\end{CD}
\]

\smallskip

$\bullet$ $\al_m$ is the morphism $(f^*\pp f^!)\circ [d\times\Id]\to [c\times \Id]\circ (f^*\times f^!)$, induced by the morphism
$(f\times\Id)^*\circ[d\times\Id]\to[c\times\Id]\circ (f\times\Id)^*$ (see Section~\re{pb}(a));

\smallskip

$\bullet$ $\al_r$ is the canonical morphism

\[
\Gm_{\tri}(Y,\Dt_Y^!(-))\overset{f^*}{\lra} \Gm_{\tri}(X,(f^*\circ\Dt_Y^!)(-))\to \Gm_{\tri}(X,(\Dt_X^!\circ(f^*\pp f^!))(-)),
\]
induced by the base change morphism
\[
f^*\circ\Dt_Y^!\to\Dt_X^!\circ(\Id_X\times f)^!\circ(f\times\Id_Y)^*\simeq \Dt_X^!\circ(f^*\pp f^!),
\]
corresponding to the Cartesian diagram
\[
\begin{CD}
X @>(\Id_X,f)>> X\times Y\\
@VfVV @VVf\times\Id_Y V\\
Y @>\Dt_Y>> Y\times Y.
\end{CD}
\]
\end{Emp}

\begin{Emp}
We have an identification $(\B{D}_X \circ f^*\circ \B{D}_Y)(\A)\simeq f^!(\A)$ for every  $\A\in\Shv(Y)^c$. Thus, unwinding the definitions, the trace map
\[
\Tr([f]^*):\Tr(\Shv(Y),[d])\to\Tr(\Shv(X),[c])
\]
is induced by a lax commutative square
\begin{equation} \label{Eq:tracepb}
\xy
(0,10)*+{\Vect}="A";
(40,10)*+{\Shv(Y)\otimes \Shv(Y)}="B";
(80,10)*+{\Shv(Y)\otimes \Shv(Y)}="C";
(120,10)*+{\Vect}="D";
(0,-10)*+{\Vect}="A'";
(40,-10)*+{\Shv(X)\otimes \Shv(X)}="B'";
(80,-10)*+{\Shv(X)\otimes \Shv(X)}="C'";
(120,-10)*+{\Vect,}="D'";
{\ar@{->}^{u_Y} "A";"B"};
{\ar@{->}_{u_X} "A'";"B'"};
{\ar@{->}^{\ev_Y} "C";"D"};
{\ar@{->}_{\ev_X} "C'";"D'"};
{\ar@{->}^{[d]\otimes\Id} "B";"C"};
{\ar@{->}_{[c]\otimes\Id} "B'";"C'"};
{\ar@{->}^{\on{id}} "A";"A'"};
{\ar@{->}^{\on{id}} "D";"D'"};
{\ar@{->}^{f^*\otimes f^!} "B";"B'"};
{\ar@{->}^{f^*\otimes f^!} "C";"C'"};
{\ar@{=>}_{\al_l} "B";"A'"};
{\ar@{=>}_{\al_m} "C";"B'"};
{\ar@{=>}_{\al_r} "D";"C'"};
\endxy
\end{equation}
where

\smallskip

$\bullet$ $\al_l$ corresponds to the morphism
\[
(f^*\otimes f^!)(u_Y)\to u_{X}\simeq\boxtimes^R((\Dt_X)_*(\om_X)),
\]
obtained by adjunction from the composition
\[
\boxtimes((f^*\otimes f^!)(u_Y))\simeq (f^*\pp f^!)(\boxtimes(u_Y))\to
(f^*\pp f^!)((\Dt_Y)_*(\om_Y))\overset{\form{map 1}}{\lra}(\Dt_X)_*(\om_X),
\]
induced by the natural map $\boxtimes(u_Y)\to (\Dt_Y)_*(\om_Y)$,

\smallskip

$\bullet$ $\al_m$ corresponds to the morphism $[f]^*:f^*\circ [d]\to [c]\circ f^*$ from Section~\re{pb}(a);

\smallskip

$\bullet$ $\al_r$ is the composition

\[
\Gm_{\tri}(Y,-\overset{!}{\otimes}-)\overset{f^*}{\lra} \Gm_{\tri}(X,f^*(-\overset{!}{\otimes}-))\to \Gm_{\tri}(X,f^*(-)\overset{!}{\otimes}f^!(-)).
\]
induced by the canonical morphism $f^*(-\overset{!}{\otimes}-)\to f^*(-)\overset{!}{\otimes}f^!(-)$.
\end{Emp}

\begin{Emp} \label{E:pftrueb3}
(i) We can decompose the diagram of \rt{reftrue}(b) as
\begin{equation} \label{Eq:rttrueb}
\begin{CD}
\Tr(\Shv(Y),[d])@>\form{true}_d>> \Gm_{\tri}(Y, (\Dt_Y^!\circ [d\times\Id]\circ (\Dt_Y)_*)(\om_Y))@>{\text{base change}}>> \Gm_{\tri}(\on{Fix}(d), \om_{\on{Fix}(d)})\\
@VV\Tr([f]^*)V                      @V\form{tracepb2}VV                      @VVg_{\Dt}^*V\\
\Tr(\Shv(X),[c])@>>\form{true}_c> \Gm_{\tri}(X, (\Dt_X^!\circ [c\times\Id]\circ (\Dt_X)_*)(\om_X))@>>{\text{base change}}> \Gm_{\tri}(\on{Fix}(c), \om_{\on{Fix}(c)}),
\end{CD}
\end{equation}
where the top and the bottom arrows are the compositions from Section~\re{reftrue}(b) for $d$ and $c$, respectively, while the middle vertical arrow is induced by the lax commutative square \form{tracepb2}. Therefore, it suffices to show that both inner squares of diagram~\form{rttrueb} are canonically homotopy commutative.

\smallskip

(ii) As in the proof of \rt{reftrue}(a),  all arrows in the left inner square of diagram~\form{rttrueb} are horizontal compositions of the corresponding lax commutative squares.
In particular, in order to show that  the left inner square of diagram~\form{rttrueb} is commutative, it suffices to show that vertical compositions of the corresponding lax commutative squares are homotopic. As in the proof of \rt{reftrue}(a),  it is a routine verification.
\end{Emp}

It remains to show the commutativity of the right inner square of diagram~\form{rttrueb}.

\begin{Emp}
Set  $d^{\on{op}}:=(d_r,d_l):D\to Y\times Y$ and $c^{\on{op}}:=(c_r,c_l):C\to X\times X$.

\medskip

Note that the right inner square of diagram~\form{rttrueb} decomposes as
\begin{equation} \label{Eq:rttrueb2}
\begin{CD}
\Gm_{\tri}(Y, \Dt_Y^![d\times\Id](\Dt_Y)_*(\om_Y))@>>{BC_1}> \Gm_{\tri}(D, (d^{\on{op}})^!(\Dt_Y)_*(\om_Y)) @>>{BC_4}> \Gm_{\tri}(\on{Fix}(d), \om_{\on{Fix}(d)})\\
@Vf^*VV        @Vg^*VV              @VVg_{\Dt}^*V\\
\Gm_{\tri}(X, f^*\Dt_Y^![d\times\Id](\Dt_Y)_*(\om_Y))@>{BC_2}>> \Gm_{\tri}(C, g^*(d^{\on{op}})^!(\Dt_Y)_*(\om_Y)) @>{BC_5}>> \Gm_{\tri}(\on{Fix}(c), g_{\Dt}^*\om_{\on{Fix}(d)})\\
@V(1)VV        @V(2)VV              @VV\on{Gys}_{g_{\Dt}}V\\
\Gm_{\tri}(X, \Dt_X^![c\times\Id](f^*\pp f^!) (\Dt_Y)_*(\om_Y))@>{BC_3}>> \Gm_{\tri}(C, (c^{\on{op}})^!(f^*\pp f^!)(\Dt_Y)_*(\om_Y)) @. \Gm_{\tri}(\on{Fix}(c), g_{\Dt}^!\om_{\on{Fix}(d)})\\
@V\form{map 1}VV @V\form{map 1}VV @|\\
\Gm_{\tri}(X, \Dt_X^![c\times\Id](\Dt_X)_*(\om_X))  @>{BC_3}>> \Gm_{\tri}(C, (c^{\on{op}})^!(\Dt_X)_*(\om_X)) @>{BC_6}>> \Gm_{\tri}(\on{Fix}(c), \om_{\on{Fix}(c)}),
\end{CD}
\end{equation}

\smallskip

where

\smallskip

\noindent$\bullet$ morphism (1) is induced by the composition of morphisms defined in Section~\re{pftrueb2}:
\[
f^*\circ\Dt_Y^!\circ[d\times\Id]\to \Dt_X^!\circ(f^*\pp f^!)\circ[d\times\Id]\to  \Dt_X^!\circ[c\times\Id]\circ(f^*\pp f^!);
\]

\smallskip

\noindent$\bullet$ morphism (2) is induced by the composition
\[
g^*\circ(d^{\on{op}})^!\simeq g^*\circ(\Id_D,d_l)^!\circ(d_r\times\Id_Y)^!\overset{\text{base change}}{\lra} (\Id_X,f\circ c_l)^!\circ(g\times\Id_Y)^*\circ(d_r\times\Id_Y)^!\overset{\form{genbc}}{\lra}
\]
\[
\overset{\form{genbc}}{\lra} (\Id_X,f\circ c_l)^!\circ(c_r\times\Id_Y)^!\circ(f\times\Id_Y)^*\simeq (c^{\on{op}})^!\circ(f^*\pp f^!),
\]
induced by diagrams
\[
\begin{CD}
C @>(\Id_C,f\circ c_l)>> C\times Y @.\;\;\;\;\;\;\;\;\;\;\;\;\;\;\;\;\;\;\;\; C\times Y @>c_r\times\Id_Y>> X\times Y\\
@VgVV @Vg\times\Id_Y VV \;\;\;\;\;\;\;\;\;\;\;\;\;\;\;\;\;\;\;\;  @Vg\times\Id_YVV @VVf\times\Id_YV\\
D @>(\Id_D, d_l)>> D\times Y  @.\;\;\;\;\;\;\;\;\;\;\;\;\;\;\;\;\;\;\;\; D\times Y @>d_r\times\Id_Y>> Y\times Y;
\end{CD}
\]
\smallskip

\noindent$\bullet$ morphism $BC_1$ is induced by the morphism
\begin{equation} \label{Eq:bc1}
\Dt_Y^!\circ[d\times\Id]\to (d_l)_{\tri}\circ(d^{\on{op}})^!,
\end{equation}
obtained from the base change morphism
$\Dt_Y^!\circ(d_l\times\Id)_{\tri}\to (d_l)_{\tri}\circ(\Id_D, d_l)^!$ (the inverse of morphism \form{bcren}),
induced by the Cartesian diagram
\[
\begin{CD}
D @>(\Id_D,d_l)>> D\times Y\\
@Vd_lVV @VV d_l\times\Id_Y V\\
Y @>\Dt_Y>> Y\times Y.
\end{CD}
\]
Note that morphism \form{bcren} is an isomorphism in our case, because morphism $\Dt_Y$ (and hence its pullback $(\Id_D,d_l)$) is representable, thus safe (see Section~\re{verdier}(d));
\smallskip

\noindent$\bullet$ morphism $BC_2$ is induced by the composition of the base change morphisms
\[
f^*\circ\Dt_Y^!\circ[d\times\Id]\overset{\form{bc1}}{\lra} f^*\circ(d_l)_{\tri}\circ(d^{\on{op}})^!\overset{\ref{E:safe1}(b)}{\lra} (c_l)_{\tri}\circ g^*\circ (d^{\on{op}})^!;
\]

\smallskip

\noindent$\bullet$ morphism $BC_3$ is induced by the morphism $\Dt_X^!\circ[c\times\Id]\to (c_l)_{\tri}\circ(c^{\on{op}})^!$, defined similarly to \form{bc1}.
\smallskip

\smallskip

\noindent$\bullet$ morphism $BC_4$ is induced by the base change morphism  $(d^{\on{op}})^!\circ(\Dt_Y)_*\to (\Dt_d)_*\circ d_{\Dt}^!$,
induced by the Cartesian diagram
\[
\begin{CD}
\on{Fix}(d) @>\Dt_d>> D\\
@Vd_{\Dt} VV @VVd^{\on{op}}V\\
Y @>\Dt_Y>> Y\times Y,
\end{CD}
\]
and we use the fact that morphism $\Dt_Y$ (and hence its pullback $\Dt_d$) is representable, thus safe;

\smallskip

\noindent$\bullet$ morphism $BC_6$ is induced by the base change morphism  $(c^{\on{op}})^!\circ (\Dt_X)_*\to (\Dt_{c})_*\circ c_{\Dt}^!$, defined similarly;

\smallskip

\noindent$\bullet$ morphism $BC_5$ is induced by the composition of the base change morphisms
\[
g^*\circ (d^{\on{op}})^!\circ (\Dt_Y)_*\to g^*\circ (\Dt_d)_*\circ d_{\Dt}^!\to  (\Dt_{c})_*\circ g_{\Dt}^*\circ d_{\Dt}^!.
\]
\end{Emp}

It remains to show that all inner squares  of diagram~\form{rttrueb2} are homotopy commutative.

\begin{Emp}
This clear for the bottom left square. Next, unwinding the definitions, the assertion for the top right inner square follows from the fact the functor of Section~\re{ren}(c) is compatible with compositions, while the assertion for the top left square follows from this and Section~\re{safe1}(b)'.

\smallskip

Moreover, the assertion for the middle left inner square reduced to the homotopy commutativity of the diagram
\begin{equation} \label{Eq:rttrueb2'}
\begin{CD}
f^*(d_l)_{\tri}(\Id_D,d_l)^! @>\re{safe1}(b)>> (c_l)_{\tri} g^*(\Id_D,d_l)^! @>\form{genbc}>> (c_l)_{\tri}(\Id_C,f\circ c_l)^!(g\times\Id_Y)^* \\
      @VV\form{bcren}V       @.        @V\form{bcren}VV\\
f^*\Dt_Y^!(d_l\times\Id_Y)_{\tri}  @>\text{base change}>> (\Id_X,f)^!(f\times\Id_Y)^*(d_l\times\Id_Y)_{\tri} @>\re{safe1}(b)>> (\Id_X,f)^!(c_l\times\Id_Y)_{\tri}(g\times\Id_Y)^*.
\end{CD}
\end{equation}

Furthermore, it suffices to show the homotopy commutativity of the diagram, obtained from diagram \form{rttrueb2'} by replacing $(-)_{\tri}$ by $(-)_*$ in all places and all morphisms by the base change morphisms. Finally, the homotopy commutativity of the corresponding diagram follows from the fact that base change morphisms are compatible with compositions.

\smallskip

Thus, it suffices to show that the bottom inner square of diagram~\form{rttrueb2} is homotopy commutative as well. In other words, it suffices to show the homotopy commutativity of the diagram
\begin{equation} \label{Eq:rttrueb3}
\begin{CD}
g^*\circ(d^{\on{op}})^!\circ(\Dt_Y)_* @>\text{base change}>> g^*\circ(\Dt_d)_*\circ d_{\Dt}^! @>\text{base change}>> (\Dt_c)_*\circ g_{\Dt}^*\circ d_{\Dt}^!\\
        @V(2)VV      @.        @VV \on{Gys}_{g_{\Dt}}V\\
(c^{\on{op}})^!\circ(f^*\pp f^!)\circ(\Dt_Y)_*  @>\form{map 1}>> (c^{\on{op}})^!\circ(\Dt_X)_*\circ f^! @>\text{base change}>>  (\Dt_c)_*\circ c_{\Dt}^!\circ f^!.
\end{CD}
\end{equation}

\end{Emp}

\begin{Emp}

Recall that morphism $f$ is  a smooth,  morphism $g$ is quasi-smooth and $\un{\dim}_g=c_r^{\cdot}(\un{\dim}_f)$.
By Section~\re{safe1}(b)', using identifications  $f^!\simeq f^*\lan\un{\dim}_f\ran$ and  $f^*\pp f^!\simeq (f\times f)^*\lan\pr_2^{\cdot}(\un{\dim}_f)\ran$ diagram \form{rttrueb3} can be rewritten as
\begin{equation} \label{Eq:rttrueb4}
\begin{CD}
g^*\circ (d^{\on{op}})^!\circ (\Dt_Y)_* @>\text{base change}>> g^*\circ (\Dt_d)_*\circ d_{\Dt}^! @>\text{base change}>>
(\Dt_c)_*\circ g_{\Dt}^*\circ d_{\Dt}^!\\
        @V\form{genbc}VV      @.        @VV\form{genbc}V\\
(c^{\on{op}})^!\circ(f\times f)^*\circ(\Dt_Y)_*\lan \un{d}_g\ran  @>\text{base change}>> (c^{\on{op}})^!\circ(\Dt_X)_*\circ f^* \lan \un{d}_g\ran @>\text{base change}>>  (\Dt_c)_*\circ c_{\Dt}^!\circ f^* \lan \un{d}_g\ran,
\end{CD}
\end{equation}
where we set $\un{d}_g:=\un{\dim}_g$. Thus it suffices to show that diagram \form{rttrueb4} commutes.

\smallskip

Moreover, by adjunction, it suffices to show the homotopy commutativity of the diagram
\begin{equation} \label{Eq:rttrueb5}
\begin{CD}
\Dt_c^*\circ g^*\circ (d^{\on{op}})^! @>\form{genbc}>> \Dt_c^*\circ (c^{\on{op}})^!\circ(f\times f)^*\lan\un{d}'_g\ran @>\text{base change}>>  c_{\Dt}^!\circ\Dt_X^*\circ(f\times f)^*\lan \un{d}'_g\ran\\
        @|      @.        @|\\
g_{\Dt}^*\circ\Dt_d^*\circ(d^{\on{op}})^!  @>\text{base change}>>  g_{\Dt}^*\circ d_{\Dt}^!\circ \Dt_Y^*   @>\form{genbc}>>      c_{\Dt}^! \circ f^*\circ \Dt_Y^* \lan\un{d}'_g\ran,
\end{CD}
\end{equation}
where we set $\un{d}'_g:=\Dt_c^{\cdot}( \un{\dim}_g)=c_{\Dt}^{\cdot}(\un{\dim}_f)$.
\end{Emp}

\begin{Emp}
Finally, to see the commutativity of diagram~\form{rttrueb5}, it suffices to show that both horizontal compositions are naturally identified with morphism \form{genbc}, corresponding to the commutative diagram
\begin{equation} \label{Eq:final}
\CD
\on{Fix}(c) @>c_{\Dt}>> X \\
 @Vg\circ \Dt_c VV       @VV\Dt_Y\circ f V \\
D @>d^{\on{op}} >> Y\times Y.
\endCD
\end{equation}

\smallskip

Note that diagram \form{final} has two decompositions
\begin{equation} \label{Eq:final1}
\CD
\on{Fix}(c) @>c_{\Dt}>> X @.\;\;\;\;\;\;\;\;\;\;\;\;\;\;\;\;\;\;\;\; \on{Fix}(c) @>c_{\Dt}>> X \\
 @V\Dt_c VV       @VV\Dt_X V \;\;\;\;\;\;\;\;\;\;\;\;\;\;\;\;\;\;\;\; @V g_{\Dt} VV @VVf V\\
C @>c^{\on{op}}>> X\times X  @.\;\;\;\;\;\;\;\;\;\;\;\;\;\;\;\;\;\;\;\; \on{Fix}(d) @>d_{\Dt}>> Y\\
@VgVV             @VVf\times f V\;\;\;\;\;\;\;\;\;\;\;\;\;\;\;\;\;\;\;\; @V\Dt_d VV @VV\Dt_Y V\\
D @>d^{\on{op}} >> Y\times Y, @.\;\;\;\;\;\;\;\;\;\;\;\;\;\;\;\;\;\;\;\; D @>d^{\on{op}} >> Y\times Y.
\endCD
\end{equation}
Using the right diagram of \form{final1}, whose bottom inner square is Cartesian, one sees that diagram \form{final} is pullable, and the
induces morphism
\[
\wt{p}:\on{Fix}(c)\to \on{Fix}(d)\times_Y X\simeq D\times_{Y\times Y}X
\]
is quasi-smooth of relative dimension $\un{\dim}_{g_{\Dt}}-c_{\Dt}^{\cdot}(\un{\dim}_f)=-\Dt^{\cdot}_c(\un{\dim}_g)$.

\smallskip

 Moreover, since
base change morphisms are compatible with compositions, the bottom composition of diagram ~\form{rttrueb5} naturally identifies with morphism \form{genbc}, corresponding to diagram~\form{final}.

\smallskip

Next, the top inner square of the left diagram of \form{final1} decomposes as
\begin{equation} \label{Eq:final2}
\CD
\on{Fix}(c) @>\wt{p}>> D\times_{Y\times Y}X  @>>> X\\
@V\Dt_c VV   @VVV     @VV\Dt_X V \\
C @>p>> D\times_{Y\times Y}(X\times X) @>\wt{d}^{\on{op}}>>  X\times X,
\endCD
\end{equation}
and morphism $p$ is quasi-smooth of relative dimension $\un{\dim}_g-c^{\cdot}(2\un{\dim}_f)=-\un{\dim}_g$.

\smallskip

By \rl{quasi-smooth} below, the left inner square of \form{final2} is homotopy Cartesian. Then the fact that
the top composition of diagram ~\form{rttrueb5} naturally identifies with morphism \form{genbc}, corresponding to diagram~\form{final}
follows from Section~\re{gysin}(f).
\end{Emp}

\appendix

\section{Proof of \rp{safe} and \rco{safe}} \label{S:pfsafe}

Though \rp{safe} can be showed by adapting proofs of the analogous assertions \cite[Theorem~10.2.4 and Corollary~10.2.7]{DG} for $D$-modules, we sketch the argument for completeness.

\smallskip

The following assertion is well-known to specialists.

\begin{Lem} \label{L:group}
Let $G$ be a connected algebraic group over $k$. The following are equivalent:

\smallskip

(i) $G$ is unipotent;

\smallskip

(ii) The canonical morphism  $\qlbar\to\Gm(G,\qlbar)$ is an isomorphism;

\smallskip

(ii)' The cohomology groups $H^i(G,\qlbar)$ vanish for all $i>0$;

\smallskip

(iii) The canonical morphism  $\qlbar\to\Gm(BG,\qlbar)$ is an isomorphism;

\smallskip

(iii)' The complex $\Gm(BG,\qlbar)\in \Vect$ is cohomologically bounded.
\end{Lem}

\begin{proof}
(i)$\implies$(ii) is standard, while (ii)$\implies$(ii)' is clear.

\smallskip

(ii)'$\implies$(i) By Chevalley theorem, there exists an exact sequence of connected algebraic groups
\[
1\to G_1\to G\overset{\pr}{\lra} G_2\to 1,
\]
where $G_1$ is affine and
$G_2$ is an abelian variety. Then $\pr_*(\qlbar)\in \Shv(G_2)$ is a constant complex with value $\Gm(G_1,\qlbar)$, thus
we have
\[
\Gm(G,\qlbar)\simeq \Gm(G_1,\qlbar)\otimes\Gm(G_2,\qlbar).
\]
Hence our assumption (ii)' implies that
\[
H^i(G_1,\qlbar)=H^i(G_2,\qlbar)=0\text{ for all }i>0.
\]
But it is well-known that these conditions imply that $G_2$ is trivial and $G_1$ is unipotent.

\medskip

\noindent{\bf Remarks.} For the rest of the proof, we equip $\Shv(BG)$ with the {\em usual}, that is, non-perverse $t$-structure. Note that the composition \[
\pt\overset{p}{\lra}BG\overset{\pi}{\lra}\pt
\]
is the identity. Since $G$ is connected and $p^*p_*(\qlbar)\simeq \Gm(G,\qlbar)$, the pullback $p^*$ induces an equivalence $\Shv(BG)^{\heartsuit}\to\Shv(\pt)^{\heartsuit}$ on hearts, where the inverse functor is given by the composition
\[
\CH^0\circ p_*:\Shv(\pt)^{\heartsuit}\to \Shv(BG)^{\heartsuit}.
\]

In particular, every object of $\Shv(BG)^{\heartsuit}$ is a constant sheaf. Hence for every $i\in \B{Z}$, the cohomology sheaf $\CH^i(p_*(\qlbar))$ is a  constant sheaf on $BG$ with value $H^i(G,\qlbar)$, thus we have an isomorphism
\begin{equation} \label{Eq:isom}
\pi_*(\CH^i(p_*(\qlbar))\simeq H^i(G,\qlbar)\otimes\Gm(BG,\qlbar).
\end{equation}

\smallskip

\noindent Now we are ready to finish the proof of the lemma:

\medskip

(ii)$\implies$(iii): Assumption (ii) together with isomorphism \form{isom} imply that
\[
\Gm(BG,\qlbar)\simeq\pi_*p_*(\qlbar)\simeq\qlbar.
\]

\smallskip

Since (iii)$\implies$(iii)' is clear, it suffices to show the implication (iii)'$\implies$(ii)'. We set
\[
m:=\max\{i\,|H^i(BG,\qlbar)\neq 0\}\text{ and }n:=\max\{i\,|H^i(G,\qlbar)\neq 0\}.
\]
Then we have a fiber sequence
\[
\tau^{<n}p_*(\qlbar)\to p_*(\qlbar)\to \tau^{\geq n}(p_*(\qlbar))
\]
in $\Shv(BG)$,  hence a fiber sequence
\begin{equation} \label{Eq:fiber}
\pi_*(\tau^{<n}p_*(\qlbar))\to \qlbar\to \pi_*(\tau^{\geq n}p_*(\qlbar))
\end{equation}
in $\Vect$.

Note that since $\tau^{<n}p_*(\qlbar)$ is an extension of $\CH^i(p_*(\qlbar))[-i]$ with $0\leq i\leq n-1$, we get from isomorphism \form{isom} that
$\pi_*(\tau^{<n}p_*(\qlbar))$ is an extension of $H^i(G,\qlbar)[-i]\otimes \Gm(BG,\qlbar)$, therefore
\[
\pi_*(\tau^{<n}p_*(\qlbar))\in \Vect^{<n+m}.
\]

Hence it follows from the fiber sequence \form{fiber} that the induced map
\[
\CH^{m+n}(\qlbar)\to \CH^{n+m}(\pi_*(\tau^{\geq n}p_*(\qlbar)))
\]
between cohomologies is an isomorphism. On the other hand, using isomorphism \form{isom} again we see that
\[
\pi_*(\tau^{\geq n}(p_*(\qlbar)))\simeq H^n(G,\qlbar)[-n]\otimes\Gm(BG,\qlbar),
\]
thus
\[
\CH^{n+m}(\pi_*(\tau^{\geq n}(p_*(\qlbar)))\simeq H^n(G,\qlbar)\otimes H^m(BG,\qlbar)\neq 0.
\]
Hence $\CH^{m+n}(\qlbar)\neq 0$, hence $n+m=0$. Since $n,m\geq 0$, this implies that $n=m=0$.
\end{proof}

The following lemma is standard:

\begin{Lem} \label{L:conserv}
Let $f:X\to Y$ be a surjective morphism between Artin stacks. Then the pullback functors $f^*,f^!:\Shv(Y)\to\Shv(X)$ are conservative and have the property that $\A\in\Shv(Y)$ is constructible if and only if $f^*(\A)\in\Shv(X)$ (resp. $f^!(\A)\in \Shv(X)$) is constructible.
\end{Lem}

\begin{Cor} \label{C:red}
Consider Cartesian diagram of Artin stacks
\begin{equation} \label{Eq:bcred}
\begin{CD}
X'@>f'>> Y'\\
@V a VV    @VVbV\\
X @>f>> Y,
\end{CD}
\end{equation}
where $b$ is surjective. If morphism $f'$ satisfies the property (ii) (resp. (iii)) of \rp{safe}(b), then so is $f$.
\end{Cor}

\begin{proof}
Assume that functor $(f')_*$ is continuous.  Then the composition $b^!\circ f_*\simeq a^!\circ (f')_*$ is continuous. Since $b^!$ is continuous and conservative (by \rl{conserv}), we therefore conclude that $f_*$ is continuous.

Next, if functor $(f')_!$ maps constructible sheaves to constructible, then for every $\A\in \Shv(X)^{\constr}$, the sheaf $b^!\circ f_!(\A)\simeq a^!\circ f'_!(\A)$ is constructible. Hence, the sheaf $f_!(\A)$ is constructible by \rl{conserv}.
\end{proof}

\begin{Emp}
\begin{proof}[Proof of \rp{safe}(a)] \hfill

\medskip

\noindent(ii)$\implies (iii)$: Note that for every constructible $\A\in\Shv(X)^{\constr}$, the functor
\[
\CHom_{\Shv(X)}(\A,-)\simeq\CHom_{\Shv(X)}(\qlbar,\B{D}_X(\A)\overset{!}{\otimes}-)
\]
is continuous, because both functors $\B{D}_X(\A)\overset{!}{\otimes}-$ and $\CHom_{\Shv(X)}(\qlbar,-)$ are continuous.

\medskip

\noindent(iii)$\implies (i)$: Fix a geometric point $x$ of $X$ and set $G:=(G_x)_{\red}^0$. We want to show that $G$ is unipotent.
Note that the natural morphism
$i_x:BG\to X$ is schematic, thus the sheaf $(i_x)_!(\qlbar)\in \Shv(X)$ is constructible. Therefore the sheaf $(i_x)_!(\qlbar)$ is compact (by assumption (iii)), hence the pushforward
\[
(p_X)_!(i_x)_!(\qlbar)=(p_{BG})_!(\qlbar)\in \Shv(\pt)
\]
is compact, thus cohomologically bounded. Then $\Gm(BG,\qlbar)$ is cohomologically bounded, so $G$ is unipotent by \rl{group}.

\medskip

\noindent(i)$\implies$(ii): To make the proof more structural, we will divide it into five steps.

\smallskip

\noindent{\bf Step 1.} Let $j:U\hra X$ be an open embedding, and let $i:Z\hra X$ be the complementary closed embedding. Using fiber sequence
\[
j_!j^*(\qlbar)\to\qlbar\to i_!i^*(\qlbar),
\]
we see that $\qlbar\in \Shv(X)$ is compact if and only if both $\qlbar\in \Shv(U)$ and  $\qlbar\in \Shv(Z)$ are. Thus, by Noetherian induction, we can replace $X$ by an open non-empty substack.

\smallskip

\noindent{\bf Step 2.} Combining Step 1 and \cite[Tag~06RC]{Stacks}, we can assume that $X$ is a gerbe over an algebraic space $Y$.
Let $\pi:X\to Y$ be the projection. Since $\qlbar\in \Shv(Y)$ is compact, it thus suffices to show that the projection $\pi_*$ is continuous.
By \cite[Tag~06QH]{Stacks}, there exists a faithfully flat morphism $Y'\to Y$ such that $X\times_Y Y'\simeq Y'/G$ for some
algebraic group space $G$ over $Y'$. Thus, by \rco{red}, we can replace $\pi$ by its pullback to $Y'$, thus assuming that
$\pi$ is the projection $Y/G\to Y$.

\smallskip

\noindent{\bf Step 3.} Applying Step 1 again, we can replace $Y$ by this open subspace. Thus we can assume that
$Y$ is a connected scheme, and $G$ is a group scheme. Furthermore, we can assume there exists a surjective morphism $Y'\to Y$ such that the reduced pullback $G':=(G\times_Y Y')_{\red}$ satisfies the property that the projection $G'\to Y'$ is smooth, and all fibers of the connected component $(G')^0\to Y$ are integral. Using \rco{red} again, we can replace $\pi$ by its pullback to $Y'$, thus assuming that $G$ is smooth over $Y$ and all geometric fibers of the projection
$G^0\to Y$ are irreducible. Furthermore, since $X/G$ is a safe stack, we conclude that all fibers of the projection $G^0\to Y$ are unipotent.

\smallskip

\noindent{\bf Step 4.} Consider the projection $p:Y\to Y/G$.

\smallskip

(a) Assume first that all geometric fibers of the projection $G\to Y$ are connected and unipotent. In this case, the projection $p:Y\to Y/G$ is smooth with unipotent geometric fibers, hence the pullback $p^*:\Shv(Y/G)\to \Shv(Y)$ is an equivalence of categories with inverse functor $p_*$. Therefore functor $\pi_*\simeq (p_*)^{-1}\simeq p^*$ is an equivalence of categories, thus it is continuous.

\smallskip

(b) Assume now that $G$ is a finite group. In this case, every object $\A\in\Shv(Y/G)$ is a direct factor of $p_*p^*(\A)$.
Thus for every $\CB\in \Shv(Y)^c$ its pullback $\pi^*(\CB)$ is a direct factor of
\[
p_*p^*\pi^*(\CB)\simeq p_*(\CB)\simeq p_!(\CB).
\]
Hence $p^*(\CB)$ is compact, as claimed.

\smallskip

\noindent{\bf Step 5.} Now we are ready to show the assertion. Suppose that we are in the situation of Step 3. Then the morphism $\pi:Y/G\to Y$ decomposes as
\[
Y/G\overset{\pi'}{\lra}Y/G^0\overset{\pi''}{\lra}Y,
\]
so it remains to show that both functors $\pi'_*$ and $\pi''_*$ are continuous. The assertion for $\pi''_*$ follows from Step 4(a).
Next, by \rco{red}, we can replace $\pi'$ with its pullback with respect to the projection $Y\to Y/G^0$. In other words, it suffices to show the continuity of $\pi_*$ when $\pi$ is the projection $Y/\pi_0(G)\to Y$. But this follows from  Step 4(b).
\end{proof}
\end{Emp}

\begin{Emp}
\begin{proof}[Proof of \rp{safe}(b)] First we claim that if $f$ is safe, then $f$ satisfies properties (ii) and (iii) of \rp{safe}(b).

\smallskip

Consider  Cartesian diagram \form{bcred}, where $b$ is a smooth covering and $Y'$ is a scheme of finite type over $k$.
Then $f'$ is safe, and it follows from \rco{red} that the assertion for $f$ follows from that for $f'$. Thus we can assume that $Y$ is a scheme.
In this case, the stack $X$ is safe, so by \rp{safe}(a) every constructible $\A\in \Shv(X)$ is compact.

\smallskip

This implies both assertions: For property (ii) note that every $\A\in\Shv(Y)^c$ is constructible, thus pullback $f^*(\A)$ is constructible, hence  compact. For property (iii) notice that every constructible $\A\in\Shv(X)$ is compact, thus $f_!(\A)$ is compact, hence  $f_!(\A)$ is constructible.

\smallskip

It remains to show that if $f$ satisfies either property (ii) or property (iii) of \rp{safe}(b), then $f$ is safe.

\smallskip

Fix a geometric point $x\in X$, let $G_x:=\Aut_f(x)$ be the stabilizer group, and let $G:=(G_x)_{\red}^0$ be the connected component of its reduction. We want to show that $G$ is unipotent.

\smallskip

Consider the commutative diagram
\[
\begin{CD}
BG@>i_x>> X\\
@V\pi VV    @VVfV\\
\pt @>\eta_{f(x)}>> Y.
\end{CD}
\]
Then $i_x$ and $\eta_{f(x)}$ are safe, thus (by the proven above) functors $(i_x)_*$ and $(\eta_{f(x)})_*$ are continuous. If $f_*$ is continuous, we conclude that $f_*\circ (i_x)_*\simeq (\eta_{f(x)})_*\circ \pi_*$ is continuous. Thus $\pi_*$ is continuous, because $(\eta_{f(x)})_*$ is conservative, which implies that $BG$ is safe (by \rp{safe}(a)), thus $G$ is unipotent.

Similarly, since $i_x$ is safe, we conclude that the pushforward $(i_x)_!(\qlbar)\in\Shv(X)$ is constructible. If $f_!$ preserves constructible sheaves, we conclude that
\[
f_!(i_x)_!(\qlbar)\simeq (\eta_{f(x)})_!\pi_!(\qlbar)
\]
is constructible. Thus  $\pi_!(\qlbar)$ is cohomologically bounded, hence $G$ is unipotent by \rl{group}.
\end{proof}
\end{Emp}

\begin{Emp}
\begin{proof}[Proof of \rco{safe}]
Since both $f_*$ and $f_!$ satisfy smooth base change, we can replace $Y$ by its smooth covering, thus it suffices to show an isomorphism of functors $f_!\simeq f_*$ when $Y$ is a scheme of finite type over $k$.

Since morphism $f: X\to Y$ is supposed to be separated, it follows from a theorem of Olsson \cite[Theorem~1.1]{Ol} that there exists a proper surjective morphism $p:\wt{X}\to X$ from a scheme $\wt{X}$, which is quasi–projective over $Y$. For every $n\in\B{N}$, we denote by $\wt{X}^{(n)}$ the $(n+1)$-times fiber product $\wt{X}\times_X\ldots\times_X \wt{X}$ of $\wt{X}$ over $X$, and let $p^{(n)}:\wt{X}^{(n)}\to X$ be the projection map.

\smallskip

Then both $p^{(n)}:\wt{X}^{(n)}\to Y$ and $f\circ p^{(n)}:\wt{X}^{(n)}\to Y$ are proper morphisms between algebraic spaces, so the assertion of the Corollary holds in these cases\footnote{Alternatively, it can be deduced from the corresponding result for schemes by repeating the argument below.}, and thus we have a canonical isomorphism of functors
\begin{equation} \label{Eq:propersafe}
f_!\circ (p^{(n)})_!\simeq (f\circ p^{(n)})_!\simeq (f\circ p^{(n)})_*\simeq f_*\circ (p^{(n)})_*\simeq f_*\circ (p^{(n)})_!.
\end{equation}

Since $p$ is proper and surjective, the natural functor
\[
\colim_{[n]\in\Dt^{\on{op}}}(p^{(n)})_!\circ (p^{(n)})^!\to\Id_{\Shv(X)}
\]
is an isomorphism\footnote{Moreover, $\Shv(-)$ is a sheaf in the $h$-topology.}. Since both $f_!$ and $f_*$ are continuous, we thus get a natural isomorphism
\[
f_!\simeq  f_!\circ\colim_{[n]\in\Dt^{\on{op}}}(p^{(n)})_!\circ (p^{(n)})^!\simeq\colim_{[n]\in\Dt^{\on{op}}}f_!\circ (p^{(n)})_!\circ (p^{(n)})^!\overset{\form{propersafe}}{\simeq}
\]
\[
\overset{\form{propersafe}}{\simeq} \colim_{[n]\in\Dt^{\on{op}}}f_*\circ (p^{(n)})_!\circ (p^{(n)})^!\simeq f_*\circ \colim_{[n]\in\Dt^{\on{op}}}(p^{(n)})_!\circ (p^{(n)})^!\simeq f_*,
\]
as claimed.
\end{proof}
\end{Emp}

\section{Quasi-smooth morphisms} \label{S:proof}

\begin{Emp} \label{E:FYZ}
{\bf Observations.}

\smallskip

(a) For every derived Artin stack $X$, the cotangent complex $T^*(X_{\on{cl}}/X)$ lies in the cohomological degrees $\leq -2$. Indeed, passing to a smooth covering, we can assume that $\wt{A}$ is an affine derived scheme. In this case, the assertion follows from \cite[Corollary~25.3.6.4]{SAG}.

\smallskip

(b) By part~(a), the morphism $X_{\on{cl}}\to X$ is never quasi-smooth, if $X$ is not classical.
\end{Emp}

\begin{Lem} \label{L:quasi-smooth}
Let
\begin{equation} \label{Eq:basic3}
\CD
A @>a>> C\\
@VgVV @VVfV\\
B @>b>> D
\endCD
\end{equation}
be a Cartesian diagram of Artin stacks that morphisms $f$ and $g$ are quasi-smooth and satisfy $\un{\dim}_g=a^{\cdot}(\un{\dim}_f)$. Then
diagram \form{basic3} is homotopy Cartesian.
\end{Lem}

\begin{proof}
We have to show that the canonical morphism $p:A\to \wt{A}:=B\times_D^h C$ is an equivalence. Since the induced morphism
$A\to\wt{A}_{\on{cl}}$ is an isomorphism (since diagram \form{basic3} is Cartesian), it suffices to show that the cotangent complex $T^*(A/\wt{A})$ vanishes
(by \cite[Corollary~25.3.6.6]{SAG}).

Note that a sequence of morphisms $A\to\wt{A}\to B$ gives rise to a fiber sequence
\[
p^*(T^*(\wt{A}/B))\to T^*(A/B)\to T^*(A/\wt{A}).
\]
Since morphisms $f$ and $g$ are quasi-smooth, we conclude that complexes $T^*(A/B)$ and $p^*(T^*(\wt{A}/B))\simeq a^*(T^*(C/D))$ are perfect
of Tor-amplitude $\geq -1$ with Euler characteristics $\un{\dim}_g$ and $a^{\cdot}(\un{\dim}_f)$, respectively. Therefore the cotangent complex
$T^*(A/\wt{A})$ is perfect of Tor-amplitude $\geq -2$ with Euler characteristic zero $\un{\dim}_g-a^{\cdot}(\un{\dim}_f)=0$.

On the other hand, since $A\simeq\wt{A}_{\on{cl}}$, the cotangent complex $T^*(A/\wt{A})\simeq T^*(\wt{A}_{\on{cl}}/\wt{A})$ lies in the cohomological degrees $\leq -2$ (see Section~\re{FYZ}(a)). Thus  $T^*(A/\wt{A})\simeq \CF[2]$ for certain locally free $\O_A$-module  $\CF$. Moreover, since the Euler characteristic of $T^*(A/\wt{A})\simeq \CF[2]$ is zero, we conclude that $T^*(A/\wt{A})$ vanishes, as claimed.
\end{proof}

\begin{Lem} \label{L:derived}
Let $f:X\to Y$ be a quasi-smooth morphism between derived Artin stacks such that $X$ is classical.
Then the induced morphism $f_{\on{cl}}:X\to Y_{\on{cl}}$ is classical, the canonical map
$T^*(X/Y_{\on{cl}})\to T^*(X/Y)$ is an isomorphism, and we have an equality $\un{\dim}_{f_{\on{cl}}}=\un{\dim}_f$.
\end{Lem}

\begin{proof}
Let $\wt{f}:\wt{X}\to Y_{\on{cl}}$ be the homotopy pullback of $f$ under $Y_{\on{cl}}\to Y$. Then $\wt{f}$ is quasi-smooth, and
$f_{\on{cl}}$ is a retract of $\wt{f}$, hence $f_{\on{cl}}$ is quasi-smooth. To show the second assertion, notice that sequence of morphisms
$X\to Y_{\on{cl}}\to Y$ gives rise to a fiber sequence
\[
f_{\on{cl}}^*T^*(Y_{\on{cl}}/Y)\to T^*(X/Y_{\on{cl}})\to T^*(X/Y).
\]
Moreover, since $T^*(Y_{\on{cl}}/Y)$ and hence also $f_{\on{cl}}^*T^*(Y_{\on{cl}}/Y)$ lies in cohomological degrees $\leq -2$, while
$T^*(X/Y_{\on{cl}})$ and $T^*(X/Y)$ lie in cohomological degrees $\geq -1$, we get that $f_{\on{cl}}^*T^*(Y_{\on{cl}}/Y)=0$, thus
the map $T^*(X/Y_{\on{cl}})\to T^*(X/Y)$ is an isomorphism. The third assertion follows immediately from the second one.
\end{proof}

\end{document}